\documentclass[final]{article}
\usepackage[authoryear,round]{natbib}

\usepackage[utf8]{inputenc}
\usepackage[english]{babel}
\usepackage{csquotes}

\usepackage{graphicx}
\usepackage{subcaption}
\newsavebox{\imagebox}
\usepackage{setspace}

\usepackage{amsmath}
\usepackage{amssymb}
\usepackage{amsthm}
\usepackage{mathtools}
\allowdisplaybreaks
\usepackage{bm}    
\usepackage{bbm}        

\usepackage{changepage}
\usepackage{adjustbox}
\usepackage{tcolorbox}
\usepackage{enumitem}
\usepackage{array}
\usepackage{natbib}
\usepackage{tikz}
\usetikzlibrary{shapes.geometric}

\usepackage{svg}
\usepackage{svg-extract}

\usepackage[appendix=append,bibliography=common]{apxproof}

\usepackage{hyperref}

\newtheoremrep{theorem}{Theorem}[section]
\newtheoremrep{lemma}{Lemma}[section]
\newtheoremrep{corollary}{Corollary}[section]
\newtheoremrep{proposition}{Proposition}[section]
\theoremstyle{definition}
\newtheorem{definition}{Definition}[section]
\newtheorem{notation}{Notation}[section]
\newtheorem{example}{Example}[section]
\newtheorem{remark}{Remark}[section]

\newcommand{\norm}[1]{\left\lVert #1 \right\rVert}
\newcommand{\round}[1]{\left( #1 \right)}
\renewcommand{\square}[1]{\left[ #1\right]}
\newcommand{\curly}[1]{\left\{ #1 \right\}}
\newcommand{\angled}[1]{\left\langle #1 \right\rangle}
\newcommand{\abs}[1]{{\left| #1 \right|}}

\DeclareMathOperator{\diag}{diag}
\DeclareMathOperator{\Var}{\mathbb{V}ar}
\DeclareMathOperator{\vspan}{span}
\DeclareMathOperator{\rank}{rank}

\DeclareMathOperator{\Sch}{\mathcal{S}ch}

\newcommand{\N}[0]{\mathbb{N}}
\newcommand{\Z}[0]{\mathbb{Z}}
\newcommand{\R}[0]{\mathbb{R}}
\renewcommand{\P}{\mathbb{P}}

\newcommand{\ul}[1]{\underline{#1}}
\newcommand{\ol}[1]{\overline{#1}}

\newcommand{\onen}[0]{{\mathbf{1}_{n}}}

\newcommand{\eqDistr}[0]{\overset{\mathcal{D}}{=}}
\newcommand{\ie}{\textit{i.e. }}
\newcommand{\eg}{\textit{e.g. }}
\newcommand{\dVec}[1]{{\bm{#1}}}
\newcommand{\matr}[1]{{\mathbf{#1}}}
\newcommand{\rVec}[1]{{\bm{#1}}}
\newcommand{\integ}[4]{\int_{#1}^{#2}{#3}\,\mathrm{d}{#4}}

\newcommand{\Ind}{\perp\!\!\!\perp}

\newcommand{\Norm}[2]{\mathcal{N}\round{#1,\,#2}}
\newcommand{\NormMV}[3]{\mathcal{N}_{#1}\round{#2,\,#3}}
\newcommand{\Exp}[1]{\mathbb{E}\round{#1}}

\newcommand{\Cov}[2]{\mathbb{C}\text{ov}\round{#1,#2}}

\newcommand{\simiid}{\stackrel{\mathclap{\small\mbox{iid}}}{\sim}}

\usepackage[loadshadowlibrary,textsize=small]{todonotes}
\todostyle{tobia}{color=teal!50, shadow, size=\tiny, author=Tobia}

\todostyle{emilio}{color=orange!50, shadow, size=\tiny, author=Emilio}

\todostyle{claudio}{color=yellow!50, shadow, size=\tiny, author=Claudio}

\todostyle{mohammed}{color=green!50, shadow, size=\tiny, author=Mohammed}

\title{Smooth Isotropic Covariance Functions on Metric Graphs via Polyharmonic Resistance Distances}
\author{Tobia Filosi, Emilio Porcu and Claudio Agostinelli}
\date{September 2026}

\begin{document}

\maketitle

\begin{abstract}
    Metric graphs are generalisations of linear networks and provide a natural framework for the definition of continuously-indexed Gaussian processes. We define a new class of distances on these topologies, termed polyharmonic distances, which unify and extend the spectral construction underlying the effective resistance distance and the biharmonic one. We give both a spectral and a variational characterisation. Furthermore, we show an explicit class of stochastic processes whose variograms coincide with the squared polyharmonic distances. Finally, we show how these metrics can be composed with suitable completely monotonic functions to define isotropic processes having any prescribed finite-order mean-square differentiability along the edges and satisfying the Kirchhoff conditions up to order one at the vertices.
\end{abstract}


\begin{toappendix}
\section{Explicit construction of \texorpdfstring{$Z$}{Z}}
    \label{app:explicit_construction_Z}
    Here we provide all the definitions and calculations necessary to the construction of the process $Z$. Such a process, following \cite{AnderesEtAl_IsotropicCovarianceFunctions_2020}, will be defined on the whole graph and will be the sum of the two independent processes $Z_\mu$ and $Z_E$, being the former a $C^m$ polynomial on $G$, and the latter an additional noise term. The construction is quite involved and requires several technical definitions. However, we will try to give the \emph{rationale} that guides the whole construction. 
    \begin{itemize}
        \item The process $Z_\mu$ will be a random polynomial on $G$, meaning that $Z_\mu$ on each edge $e$ will be a polynomial, and that $Z_\mu$ will satisfy the usual Kirchhoff conditions at the vertices. Now, in analogy with \cite{AnderesEtAl_IsotropicCovarianceFunctions_2020}, we define $Z_\mu$ (and its derivatives) on $V$ and then extend (\ie interpolate) it on each edge $e\in E$. The choice of the covariance matrix of $Z_\mu$ will be guided and motivated by the dot product of the RKHS (\ref{eq:dot_product_RKHS}).
        \item The process $Z_E$ will be defined, independently on each edge $e\in E$, as the $m$-times integrated Wiener process conditioned to have the value and all the $m$ derivatives null at the endpoints $\ul e, \ol e$ of $e$. This will ensure that all the Kirchhoff conditions will trivially hold for $Z_E$. 
    \end{itemize}

    \subsection{Preliminary material}
        In this rather technical subsection, we introduce the necessary linear-algebra objects necessary to the explicit construction of the process $Z_\mu$. Table \ref{tab:objectsSummary} also presents a summary of the main objects involved, their dimensions and roles.
        \begin{table}[t]
            \centering
            \begin{tabular}{|c|ccc|>{\footnotesize}p{4.2cm}|}
                \hline
                \textbf{Object} & \textbf{rows} & \textbf{columns} & \textbf{defined} & \normalsize{\textbf{role}}\\
                \hline
                $\matr B_e$ & $2m+2$ & $2m+2$ & \eqref{eq:def_matrix_B} & Gets the derivatives at the endpoints ($0$ and $l$) of the polynomial with given coefficients. \\
                ${\matr H}_e$ & $2m+2$ & $d_{\dVec z}$ & Notation \ref{not:matrix_H_e} & Extracts from $\dVec z$ the information relative to the edge $e$, changing signs to the odd derivatives at $\ol e$.\\
                $\matr M_e$ & $2m+2$ & $2m+2$ & \eqref{eq:matrix_M_definition} & Computes the scalar product of the two interpolation polynomials on $e$ passed by endpoints properties. \\
                $\matr L$ & $d_{\dVec z}$ & $d_{\dVec z}$ &  \eqref{eq:laplacian_definition} & The generalised Laplacian matrix. \\
                \hline
                $\dVec x$ & $2m+2$ & $1$ & \eqref{eq:polynomial_vector_expression_variables} & Stores the first $2m+2$ integer powers of $x$, starting from the $0$\textsuperscript{th}. \\
                $\dVec j$ & $d_{\dVec z}$ & $1$ & Proposition \ref{prop:matrix_L_Sym_PD_rankL} & Stores 1 on the 0-order derivative positions. \\
                \hline
                $n$ & 1 & 1 & & The number of vertices of $G$.\\
                $q$ & 1 & 1 & & The number of edges of $G$.\\
                $m$ & 1 & 1 & & The differentiability order.\\
                \hline
            \end{tabular}
            \caption{Some objects used in this manuscript, their dimensions and roles.}
            \label{tab:objectsSummary}
        \end{table}
        \begin{definition}[$\matr B$ matrix]
            \label{def:definition_B}
            Let $[0,l] \subset \R$ and let $p:[0,l]\to\R$ be a polynomial of degree $2m+1$. We define the vectors $\dVec a, \dVec x \in \R^{2m+2}$:
            \begin{equation}
                \label{eq:polynomial_vector_expression_variables}
                \dVec a:= \begin{bmatrix}a_0, \dots, a_{2m+1}\end{bmatrix}^\top, \qquad 
                \dVec x := \begin{bmatrix}1,x, \dots, x^{2m+1}\end{bmatrix}^\top.
            \end{equation}
            These allow to write $p(x) = \dVec a^\top \dVec x$. Now, let $\dVec z_0, \dVec z_l \in \R^{m+1}$, representing the derivatives (from order $0$ to $m$) at $0$ and $l$ respectively. Finally set $\dVec z := \begin{bmatrix}
                \dVec z_0^\top & \dVec z_l^\top
            \end{bmatrix}^\top \in \R^{2m+2}$.
            We define $\matr B \in \R^{(2m+2) \times (2m+2)}$ as the matrix that satisfies the relation
            \begin{equation}
                \label{eq:def_matrix_B}
                \matr B \dVec a = \dVec z.
            \end{equation}
             We stress that $\matr B$ depends on $l$: in the following we will write $\matr B_e$ to indicate the matrix relative to the edge $e$ of length $l_e$.
        \end{definition}

        \begin{remark}
            The matrix $\matr B$ is invertible and, by setting
            \begin{align*}
                \matr B_{11} &= \diag
                \begin{bmatrix}
                    0!\\
                    \vdots \\
                    m!
                \end{bmatrix} &
                \begin{bmatrix}
                    \matr B_{21} & \matr B_{22}    
                \end{bmatrix} &= \square{\round{\prod_{j=0}^{k-1} (i-j)} l^{i-k}}_{\substack{k=0,\dots,m\\
                i=0,\dots,2m+1}},
            \end{align*}
            then we have that:
            \begin{equation}
                \matr B := \begin{bmatrix}
                    \matr B_{11} & \matr 0\\
                    \matr B_{21} & \matr B_{22}
                \end{bmatrix}, \qquad 
                \matr B^{-1} = \begin{bmatrix}
                    \matr B_{11}^{-1} & \matr 0\\
                    -\matr B_{22}^{-1} \matr B_{21} \matr B_{11}^{-1} & \matr B_{22}^{-1}
                \end{bmatrix}.
            \end{equation}

        \end{remark}

        \begin{definition}[$\matr M_l$ matrix]
            \label{def:matrix_M}
            We define $\matr M_l$ to be the $(2m+2)\times (2m+2)$ matrix that returns the dot product of the $(m+1)$-order derivatives of two polynomials on $[0,l]$, given their values at the endpoints (as explained in Definition \ref{def:definition_B}). In formulae: if $p_i(x) = \dVec x^\top \matr B_e^{-1} \dVec z_i$, for $i\in\curly{1,2}$, we define $\matr M_l$ such that 
            \begin{equation}
                \label{eq:matrix_M_definition}
                \dVec z_1^\top \matr M_l \dVec z_2 = \integ{0}{l}{p_1^{(m+1)}(x)\;p_2^{(m+1)}(x)}{x}.
            \end{equation}
            Hereafter, we will often use the shortcut $\matr M_e := \matr M_{l_e}$.
        \end{definition}
        \begin{propositionrep}[Explicit expression for $\matr M_l$]
            \label{prop:matrix_M_explicit_expression}
            It holds:
            \begin{equation}
                \label{eq:matrix_M_expression}
                \matr M_l =\begin{bmatrix}
                    \matr B_{11}^{-1} \matr B_{21}^\top  \\
                    -\matr I_{m+1}
                \end{bmatrix} \matr B_{22}^{-\top} \square{\frac{(m+1+i)!\; (m+1+j)!\;l^{i+j+1}}{i!\; j!\; (i+j+1)}}_{i,j=0}^{m} \matr B_{22}^{-1}\begin{bmatrix}
                    \matr B_{11}^{-1} \matr B_{21}^\top  \\
                    -\matr I_{m+1} 
                \end{bmatrix}^\top.
            \end{equation}
        \end{propositionrep}
        \begin{proof}[Proof of Proposition \ref{prop:matrix_M_explicit_expression}]
            Let us compute the dot product. In the following calculations, we set $\matr \Delta \in\R^{(2m+2)\times(2m+2)}$ the linear ($m+1$)-derivative operation on $\dVec x$ (that is $D_x^{(m+1)} \dVec x = \matr \Delta \dVec x$):
            \begin{equation*}
                \matr \Delta := \begin{bmatrix}
                    \matr 0 & \matr 0 \\
                    \matr \Delta_{21} & \matr 0
                \end{bmatrix}, \qquad \matr \Delta_{21} = \diag\square{\frac{(m+1+i)!}{i!}}_{i=0}^m\in\R^{(m+1)\times(m+1)}.
            \end{equation*}
            \begin{align*}
                &\integ{0}{l}{p_1^{(m+1)}(x)\,p_2^{(m+1)}(x)}{x}=\integ{0}{l}{
                D_x^{(m+1)} \round{\dVec z_1^\top \matr B^{-\top} \dVec x} \, D_x^{(m+1)} \round{\dVec z_2^\top \matr B^{-\top} \dVec x}
                }{x}\\
                &=\integ{0}{l}{
                \round{\dVec z_1^\top \matr B^{-\top} D_x^{(m+1)}\dVec x} \,  \round{\dVec z_2^\top \matr B^{-\top} D_x^{(m+1)}\dVec x}^\top
                }{x}\\
                &=\dVec z_1^\top \matr B^{-\top} \round{
                    \integ{0}{l}{\matr \Delta \dVec x \dVec x^\top \matr \Delta^\top }{x}
                } \matr B^{-1} \dVec z_2\\
                &=\dVec z_1^\top \matr B^{-\top}\matr \Delta \square{
                    \integ{0}{l}{ x^{i+j} }{x}
                }_{i,j=0}^{2m+1} \matr \Delta^\top \matr B^{-1} \dVec z_2\\
                &=\dVec z_1^\top \matr B^{-\top}  \begin{bmatrix}
                    \matr 0 & \matr 0 \\
                    \matr \Delta_{21} & \matr 0
                \end{bmatrix}\square{\frac{l^{i+j+1}}{i+j+1}
                }_{i,j=0}^{2m+1} \begin{bmatrix}
                    \matr 0 & \matr \Delta_{21} \\
                    \matr 0 & \matr 0
                \end{bmatrix} \matr B^{-1} \dVec z_2\\
                \intertext{Next, we rewrite the central term (which is a Hilbert matrix) in a convenient block-matrix form.}
                &=\dVec z_1^\top \matr B^{-\top}  \begin{bmatrix}
                    \matr 0 & \matr 0 \\
                    \matr \Delta_{21} & \matr 0
                \end{bmatrix}
                \begin{bmatrix}
                    \matr L_{11} & \matr L_{12} \\
                    \matr L_{21} & \matr L_{22}
                \end{bmatrix}
                \begin{bmatrix}
                    \matr 0 & \matr \Delta_{21} \\
                    \matr 0 & \matr 0
                \end{bmatrix} \matr B^{-1} \dVec z_2\\
                &=\dVec z_1^\top \begin{bmatrix}
                    \matr B_{11}^{-1} & \matr 0\\
                    -\matr B_{22}^{-1} \matr B_{21} \matr B_{11}^{-1} & \matr B_{22}^{-1}
                \end{bmatrix}^{\top}
                \begin{bmatrix}
                    \matr 0 & \matr 0 \\
                    \matr 0 & \matr \Delta_{21} \matr L_{11} \matr \Delta_{21}
                \end{bmatrix} \begin{bmatrix}
                    \matr B_{11}^{-1} & \matr 0\\
                    -\matr B_{22}^{-1} \matr B_{21} \matr B_{11}^{-1} & \matr B_{22}^{-1}
                \end{bmatrix} \dVec z_2\\
                &=\dVec z_1^\top \begin{bmatrix}
                    \matr B_{11}^{-1} \matr B_{21}^\top  \\
                    -\matr I_{m+1} 
                \end{bmatrix} \matr B_{22}^{-\top} \matr \Delta_{21} \matr L_{11} \matr \Delta_{21} \matr B_{22}^{-1}\begin{bmatrix}
                    \matr B_{11}^{-1} \matr B_{21}^\top  \\
                    -\matr I_{m+1} 
                \end{bmatrix}^\top \dVec z_2
            \end{align*}
            We now conclude the proof by noticing that:
            \begin{align*}
                \matr \Delta_{21} \matr L_{11} \matr \Delta_{21} &= \diag\square{\frac{(m+1+i)!}{i!}}_{i=0}^m \square{\frac{l^{i+j+1}}{i+j+1}
                }_{i,j=0}^{m} \diag\square{\frac{(m+1+i)!}{i!}}_{i=0}^m\\
                &=\square{\frac{(m+1+i)!\; (m+1+j)!\;l^{i+j+1}}{i!\; j!\; (i+j+1)}}_{i,j=0}^{m}.
            \end{align*}
        \end{proof}
        
        \begin{propositionrep}
            \label{prop:matrix_M_is_PSD_and_rank}
            The matrix $\matr M_l$ is symmetric, positive semidefinite and has rank $m+1$.
        \end{propositionrep}
        \begin{proof}[Proof of Proposition \ref{prop:matrix_M_is_PSD_and_rank}]
            The facts that $\matr M$ is symmetric and positive semidefinite are immediate consequences of Equation (\ref{eq:matrix_M_definition}). Indeed, for any $\dVec z_1, \dVec z_2$:
            \begin{align*}
                &\dVec z_1^\top \matr M_l \dVec z_2 = \int_{0}^{l}{p_1^{(m+1)}\;p_2^{(m+1)}} = \int_{0}^{l}{p_2^{(m+1)}\;p_1^{(m+1)}} = \dVec z_2^\top \matr M_l \dVec z_1\\
                \intertext{and}
                &\dVec z_1^\top \matr M_l \dVec z_1 = \integ{0}{l}{\round{p_1^{(m+1)}(x)}^2}{x}\geq 0.
            \end{align*}
            To show that $\rank \matr M = m+1$, it is sufficient to invoke the Sylvester’s inequality, noticing that $\matr M_l$ is the product of $5$ matrices, all of rank $m+1$. The central matrix $\matr L_{11}$ is full rank as it is a Hilbert matrix.
        \end{proof}
        
        The process $Z_\mu$ is completely determined by its value (and the one of its derivatives) at the vertices $V$. However, the continuity constraint $Z_\mu \in C^m$ imposes the Kirchhoff conditions of order $0,1,\dots,m$ at each vertex. As a consequence, we may not need all the derivative information at any node, since some are redundant. Here we write the process $Z_\mu$ at $V$ in a parsimonious way: the vector $\dVec z$ will store the necessary information about the process at $V$. In the following, we assume the existence of a total order relation on the vertices.
        \begin{notation}[$\dVec z$ vector]
            \label{notat:z_vector}
             For each vertex $v \in V$, we define the $(m+1)$-block-vector:
            \begin{equation*}
                \dVec z_v := \begin{bmatrix}
                    f(v)\\
                    \dVec \partial f(v)\\
                    \partial^2 f(v)\\
                    \dVec \partial^3 f(v)\\
                    \vdots
                \end{bmatrix},\qquad \dim(\dVec z_v) = \begin{cases}
                    \frac{m}{2}d_v+1 \qquad &\text{$m$ even}\\
                    \frac{m+1}{2}d_v\qquad&\text{$m$ odd}, 
                \end{cases}
            \end{equation*}
            where:
            \begin{itemize}
                \item the even-order derivatives ($\partial^{2k}f(v)$) have just one value, corresponding to the exiting derivative of order $2k$, which necessarily has to be equal among all directions;
                \item the odd-order derivatives ($\dVec \partial^{2k+1}f(v)$) store $d_v - 1$ values, corresponding to the exiting derivatives without the last one (always ordered w.r.t. the vertices total order). To recover the last derivative, by the odd-order Kirchhoff conditions, it is sufficient to flip the sign of the sum of the other values:
                \begin{equation*}
                    \begin{bmatrix}
                        \partial_{E_v(1)}^{2k+1} f(v)\\
                        \vdots\\
                        \partial_{E_v(d_v)}^{2k+1} f(v)\\
                    \end{bmatrix} = \begin{bmatrix}
                        \matr I_{d_v-1}\\
                        -\dVec 1_{d_v-1}^\top
                    \end{bmatrix} \dVec \partial^{2k+1} f(v).
                \end{equation*}
            \end{itemize}
            With such a definition, each vertex $v$ has all the information stored in $\dVec z_v$. Now, we define $\dVec z$ as the block-vector of all the $\dVec z_v$'s:
            \begin{equation*}
                \dVec z := \begin{bmatrix}
                    \dVec z_{v_1}\\
                    \vdots\\
                    \dVec z_{v_n}
                \end{bmatrix}, \qquad \qquad d_{\dVec z} := \dim(\dVec z) = \begin{dcases}
                    n+mq\qquad \text{$m$ even},\\
                    q+mq\qquad \text{$m$ odd}.
                \end{dcases}
            \end{equation*}
        \end{notation}
        \begin{notation}[Matrix $\matr H_e$]
            \label{not:matrix_H_e}
            Now, for a given edge $e\in E$, we define the matrix $\matr H_e\in\R^{(2m+2)\times d_{\dVec z}}$, that extracts from $\dVec z$ the information relative to the edge $e$. This is quite straightforward: we need the values and all the exiting derivatives in the two endpoints $\ul e, \ol e$, with the odd-order derivatives in $\ol e$ with flipped sign. Basically, we are moving from the process on the graph $G$ to the process on the edge $e$, considered as the segment $[0,l]$. As a consequence, $\matr H_e$ will contain only the values $0, 1, -1$.
        \end{notation}

    \subsection{The Laplacian matrix}
        We are (finally) able to write the generalised Laplacian matrix:
        \begin{equation}
            \label{eq:laplacian_definition}
            \matr L =  \sum_{e \in E} \matr H_e^\top \matr M_{l_e} \matr H_e.
        \end{equation}
        Albeit perhaps a bit ugly, such a matrix enjoys some nice properties, stated next.
        \begin{propositionrep}
            \label{prop:laplacian_consistency_m=0}
            For $m=0$, $\matr L$ coincides with the standard Laplacian matrix.
        \end{propositionrep}
        \begin{proof}[Proof of Proposition \ref{prop:laplacian_consistency_m=0}]
            For $m=0$, we have:
            \begin{equation*}
                \matr M_e = \frac{1}{l_e}\begin{bmatrix}
                    1 & -1\\
                    -1 & 1
                \end{bmatrix}.
            \end{equation*}
            In addition, $\dVec z$ will store only the 0-order derivative values, therefore $d_{\dVec z}=n$ and the matrix $\matr H_e$ will be
            \begin{equation*}
                \matr H_e = \begin{bmatrix}
                    \dVec \phi_{\ul e}^\top\\
                    \dVec \phi_{\ol e}^\top
                \end{bmatrix} \in \R^{2 \times n},
            \end{equation*}
            where $\dVec \phi_v$ is the vector of the canonical basis of $\R^n$ corresponding to the vertex $v\in V$. This means that $\matr H_e^\top \matr M_e \matr H_e$ will simply return an $n\times n$ matrix full of zeros, but the $2\times 2$ submatrix corresponding to the vertices $\ul e, \ol e$. Now, since each edge $e \in E$ is considered exactly once in (\ref{eq:laplacian_definition}), each entry $\matr L[v_1,v_2]$ will be
            \begin{equation*}
                \matr L[v_1,v_2] = \begin{dcases}
                    -\frac{1}{l(v_1,v_2)} \qquad&\text{for $v_1\sim v_2$},\\
                    0 \qquad &\text{for $v_1\not \sim v_2$},\\
                    \sum_{w\sim v_1} \frac{1}{l(v_1,w)}\qquad&\text{for $v_1=v_2$},
                \end{dcases}
            \end{equation*}
            which is the definition of the graph Laplacian matrix.
        \end{proof}
        
        \begin{propositionrep}
            \label{prop:matrix_L_Sym_PD_rankL}
            $\matr L$ is symmetric and positive semidefinite. In addition, the null space of $\matr L$ has dimension one and is the one generated by the vector $\dVec j \in \R^{d_{\dVec z}}$, which stores 1 in the positions corresponding to the values of the 0-order derivatives and 0 in all the other positions.
        \end{propositionrep}
        \begin{proof}[Proof of Proposition \ref{prop:matrix_L_Sym_PD_rankL}]
            Symmetry and positive-definiteness trivially follows directly from the definition (\ref{eq:laplacian_definition}) and from Proposition \ref{prop:matrix_M_is_PSD_and_rank}. Regarding the rank and the kernel, we want to characterise the space $\dVec z^\top \matr L \dVec z = 0$. From the definition of $\matr L$ and $\matr M_e$, it follows that for any $f,g \in \mathcal{F}_\mu$ (\ie $f,g$ are piecewise polynomials on each edge, satisfying the Kirchhoff conditions), if $\dVec z_f, \dVec z_g$ store all the derivative information about $f,g$ then we have
            \begin{equation*}
                \dVec z_f^\top \matr L \dVec z_g = \angled{f,g}_{\mathcal H_\mu} = \sum_{e \in E} \integ{0}{l_e}{f^{(m+1)}(x)\,g^{(m+1)}(x)}{x}.
            \end{equation*}
            Therefore, we are considering the following equation:
            \begin{equation*}
                0=\dVec z^\top \matr L \dVec z = \sum_{e \in E} \integ{0}{l_e}{\round{f^{(m+1)}(x)}^2}{x}.
            \end{equation*}
            Now, the above equation implies (see Lemma \ref{lem:f_zero_m+1_integ_implies_constant}) that $f$ is constant on $G$. This implies that $\dVec z$, which stores all the information about $f$, will assign the same value at each vertex, and $0$ to all the derivatives of any order. Clearly the dimension of this space is $1$. This concludes the proof. 
        \end{proof}
        
    \subsection{Finally \texorpdfstring{$Z_\mu$}{Z mu}}
        We are finally able to define the process $Z_\mu$.
        \begin{definition}
            \label{def:process_Z_mu}
            Let $\matr L$ as defined in (\ref{eq:laplacian_definition}) and let $\rVec Z_V \sim \Norm{\dVec 0}{\matr L^-}$, where $\matr L^-$ denotes the Moore-Penrose inverse of $\matr L$. We define $Z_\mu$ on $G$ (conditionally on $\rVec Z_V$) to be the polynomial process that interpolates the values of $\rVec Z_V$, that is:
            \begin{itemize}
                \item on each edge $e\in E$, $Z_\mu$ on $e$ is a polynomial of degree $2m+1$,
                \item in each vertex $v \in V$, $Z_\mu$ should match the exiting derivatives in $\rVec Z_V$.
            \end{itemize}
            This can be expressed as
            \begin{equation}
                \round{Z_\mu(x) \, \big|\,\rVec Z_V} = \dVec x^\top \matr B_e^{-1} \matr H_e \rVec Z_V,
            \end{equation}
            where $e$ is the edge containing $x$ and $\dVec x$ is defined in (\ref{eq:polynomial_vector_expression_variables}). Notice that this expression holds even for $x\in V$, in which case it is independent on the choice of the edge $e\in E_x$ (since $Z_\mu$ is continuous at the vertices).
        \end{definition}
        \begin{remark}[Distribution of $Z_\mu$]
            \label{rem:distribution_Z_mu}
            As a direct consequence of Definition \ref{def:process_Z_mu}, we obtain that $Z_\mu$ is a zero-mean Gaussian process on $G$, with covariance function and variogram given by:
            \begin{gather}
                \label{eq:cov_Z_mu}
                C_\mu(x_1,x_2) = \dVec x_1^\top \matr B_{e_1}^{-1} {\matr H}_{e_1} \matr L^- {\matr H}_{e_2}^\top \matr B_{e_2}^{-\top} \dVec x_2,\\
                \label{eq:variog_Z_mu}
                \gamma_{Z_\mu}(x_1, x_2) = \round{\matr H_{e_1}^\top \matr B_{e_1}^{-\top} \dVec x_1 - \matr H_{e_2}^\top \matr B_{e_2}^{-\top} \dVec x_2}^\top \matr L^- \round{\matr H_{e_1}^\top \matr B_{e_1}^{-\top} \dVec x_1 - \matr H_{e_2}^\top \matr B_{e_2}^{-\top} \dVec x_2},
            \end{gather}
            where $e_1,e_2$ are the edges containing $x_1,x_2 \in G$ and $\dVec x_1, \dVec x_2$ are defined as in (\ref{eq:polynomial_vector_expression_variables}).
        \end{remark}
        \begin{propositionrep}
            \label{prop:Z_mu_is_C^m}
            $Z_\mu \in C^m(G)$ in the sample-path sense.
        \end{propositionrep}
        \begin{proof}[Proof of Proposition \ref{prop:Z_mu_is_C^m}]
            Assume that $Z_\mu$ is defined on the probability space $(\Omega, \mathcal{F}, \P)$ and fix $\omega \in \Omega$. Then $\rVec Z_V(\omega)$ as in Definition \ref{def:process_Z_mu} will be an element of $\R^{d_{\dVec z}}$. On each edge $e\in E$, $Z_\mu$ is defined as the unique $(2m+1)$-degree polynomial interpolating the values and the derivatives at the endpoints. As a consequence, it is a $C^\infty$ function on $e$. Regarding the Kirchhoff conditions at each vertex $v\in V$, they are satisfied by construction of the vector $\dVec z$ as of the matrices $\matr H_e$ (see Notations \ref{notat:z_vector} and \ref{not:matrix_H_e}).
        \end{proof}

    \subsection{Construction of \texorpdfstring{$Z_E$}{Z E}}
        Now we proceed to define the additional noise component $Z_E$ of our process $Z$. $Z_E$ will be defined on each edge $e$ as an \emph{integrated Brownian bridge}. The interested reader could find some properties of the integrated Wiener process and Brownian bridge in \cite{Lachal_ClassBridgesIterated_2011}. Here we will provide a definition \cite[219]{vanZantenvanderVaart_ReproducingKernelHilbert_2008} and some relevant properties. 

        To be more precise, let $W_0(t)$ the standard Wiener process on $t \geq 0$. Define, for $m \in \N$, the \emph{$m$-times integrated Wiener process} via the following recursive equation:
        \begin{equation}
            \label{eq:def_integrated_Wiener_process}
            W_{i+1}(t) := \integ{0}{t}{W_i(\tau)}{\tau}.
        \end{equation}
        Define, in addition, the ($m+1$)-dimensional process storing $W_m$ and its derivatives (notice that by construction $\frac{d}{dt}W_{i+1}(t) = W_i(t)$ for $i\in \N$):
        \begin{equation}
            \rVec W_m(t) := \begin{bmatrix}
                W_0(t)\\
                W_1(t)\\
                \vdots\\
                W_m(t)
            \end{bmatrix}.
        \end{equation}
        Notice that, since $W_0(0) = 0$, we will have $\rVec W_m(0) = \dVec 0$ a.s..
        We then define the \emph{$m$-times integrated Brownian bridge} $B_m:[0,l] \to \R$ as follows:
        \begin{equation}
            \label{eq:integrated_Brownian_bridge}
            B_m(t) := \round{W_m(t) \, \Big| \, \rVec W_m(l) = \dVec 0}.
        \end{equation}
        Here, we report some relevant facts.
        \begin{remark}[Relevant properties of $W_m$ and $B_m$]
            The following properties have been taken from \cite{Lachal_ClassBridgesIterated_2011}.
            \begin{itemize}
                \item \textbf{Distribution.} Both $W_m$ and $B_m$ are Gaussian processes with zero mean.
                \item \textbf{Regularity.} $W_m, B_m \in C^m([0,l])$ in the sample-path sense.
                \item \textbf{Markovianity.} The process $\rVec W_m$ is Markovian.
                \item \textbf{Covariance.} The covariances of $\rVec W_m$ are 
                \begin{equation*}
                    \Cov{W_j(s)}{W_k(t)} = \integ{0}{s \wedge t}{\frac{(s-u)^j}{j!}\,\frac{(t-u)^k}{k!}}{u}.
                \end{equation*}
                \item \textbf{Decomposition.} If we condition $W_m$ to arbitrary derivative values in, say, $0<t_1<t_2$, then the resulting process has the same distribution as the sum of the interpolation polynomial (deterministic) and the process $B_m$. More precisely: let $\dVec w_1, \dVec w_2 \in \R^{m+1}$, then, on $[t_1,t_2]$,
                \begin{equation*}
                    \round{W_m(t) \, \Big| \, \rVec W_m(t_1) = \dVec w_1, \ \rVec W_m(t_2) = \dVec w_2} \eqDistr p(t; \dVec w_1, \dVec w_2) + B_m(t-t_1),
                \end{equation*}
                where $p(t; \dVec w_1, \dVec w_2)$ is the ($2m+1$)-degree interpolation polynomial whose derivatives are $\dVec w_1$ and $\dVec w_2$, and $B_m$ has $l=t_2-t_1$. This decomposition is particularly significant in the additive principle $Z = Z_\mu + Z_E$ (see below). Indeed, such a construction actually produces a process that is edge-wise an integrated Brownian bridge with prescribed derivatives at its endpoints.
                \begin{figure}[t]
                    \centering
                    \begin{subfigure}{0.31\textwidth} 
                        \centering
                        \includesvg[width = \linewidth]{Figures/IntBB_m=0.svg}
                    \end{subfigure}
                    \hfill
                    \begin{subfigure}{0.31\textwidth}
                        \centering
                        \includesvg[width = \linewidth]{Figures/IntBB_m=1.svg}
                    \end{subfigure}
                    \hfill
                    \begin{subfigure}{0.31\textwidth}
                        \centering
                        \centering
                        \includesvg[width = \linewidth]{Figures/IntBB_m=2.svg}
                    \end{subfigure}
                    \caption{Example of realisations of the processes $B_m$ and the conditioned $W_m$, for $m\in\curly{0,1,2}$. The dashed lines represent independent samples of $B_m$, whilst the solid lines represent the process $W_m$ conditioned on some values of its derivatives at the endpoints ($\dVec w_1$, $\dVec w_2$). The black, dotted line is the interpolation polynomial $p(t; \dVec w_1, \dVec w_2)$.}
                    \label{fig:examples_integrated_BB}
                \end{figure}
            \end{itemize} 
        \end{remark}
        \begin{remark}
            The covariance function of $B_m$ can be easily computed applying the standard conditional formula for Gaussian random vector. Indeed, let us consider the Gaussian random vector and its partitioned variance-covariance matrix
            \begin{equation*}
                \rVec X := \begin{bmatrix}
                    W_m(s)\\
                    W_m(t)\\
                    \rVec W_m(l)
                \end{bmatrix},\qquad \Var\round{\rVec X} =: \begin{bmatrix}
                    \matr \Sigma_{11} & \matr \Sigma_{12}\\
                    \matr \Sigma_{12}^\top & \matr \Sigma_{22}
                \end{bmatrix}.
            \end{equation*}
            It is well known that $\square{W_m(s), W_m(t)}^\top$ conditioned on $\rVec W_m(l) = \dVec 0$ is a Gaussian random vector with $0$ mean and covariance $\matr \Sigma_{11} - \matr \Sigma_{12} \matr \Sigma_{22}^{-1} \matr \Sigma_{12}^\top$. The $(1,2)$ entry of this resulting matrix will store the covariance $\Cov{B_m(s)}{B_m(t)}$. This means that 
            \begin{equation}
                \mathbb C\text{ov}\round{W_m(s), W_m(t) \ \Big| \ \rVec W_m(l) = \dVec 0} = C_{W_m}(s,t) - \dVec \sigma_s^\top \matr \Sigma_{22}^{-1} \dVec \sigma_t, 
            \end{equation}
            where $\dVec \sigma_s = \Exp{W_m(s) \rVec W_m(l)}$ and $\matr \Sigma_{22} = \Exp{\rVec W_m (l) \rVec W_m^\top (l)}$.
            
            Another method to obtain the covariance function of $B_m$ is discussed in \cite[Theorem 4.3]{Lachal_ClassBridgesIterated_2011}. Next we provide some examples for  $m = 0,1,2$ (see also Figure \ref{fig:examples_integrated_BB}). These hold for $0\leq s \leq t \leq l$:
            {\small
                \begin{align*}
                    &C_{B_0}\left( s, t \right) \ = \ \frac{{\left(l - t\right)} s}{l}, \qquad \text{(covariance of the standard Brownian bridge)}\\
                    &C_{B_1}\left( s, t \right) \ = \ \frac{{\left(3 \, l t - l s - 2 \, s t\right)} {\left(l - t\right)}^{2} s^{2}}{6 \, l^{3}},\\
                    &C_{B_2}\left( s, t \right) \ = \ \frac{{\left(l^{2} s^{2} - 5 \, l^{2} s t + 3 \, l s^{2} t + 10 \, l^{2} t^{2} - 15 \, l s t^{2} + 6 \, s^{2} t^{2}\right)} {\left(l - t\right)}^{3} s^{3}}{120 \, l^{5}}.
                \end{align*}
            }%
            Notice the presence of the factors $(l-t)^{m+1} s^{m+1}$, ensuring that $C_{B_m}$ goes to zero sufficiently fast as $s,t$ approach the boundary $\curly{0,l}$.
        \end{remark}
        \begin{definition}
            \label{def:process_Z_E}
            We define $Z_E$ piecewise on each edge: $Z_E$ on $e$, written $Z_e$, will be the $m$-times integrated Brownian bridge on $[0, l(e)]$ (see Equation (\ref{eq:integrated_Brownian_bridge})). We assume that each $Z_e$ is independent from all the other $Z_{e'}$'s and from $Z_\mu$. 
        \end{definition}
        The above definition implies that
        \begin{equation}
            \label{eq:cov_Z_E}
            C_E(x_1,x_2) := \Cov{Z_E(x_1)}{Z_E(x_2)} = \begin{dcases}
                C_{B_m}(x_1, x_2; l_e) \qquad&\text{if $e_1 = e_2$},\\
                0 \qquad&\text{if $e_1 \ne e_2$},
            \end{dcases}
        \end{equation}
        where $e_1, e_2$ are the edges containing $x_1,x_2$ respectively, and where $C_{B_m}(x_1, x_2; l_e)$ is the covariance between $Z_e(x_1)$ and $Z_e(x_2)$ when $e_1=e_2=:e$.
        \begin{propositionrep}
            \label{prop:Z_E_is_sample-path_C^m}
            $Z_E \in C^m(G)$.
        \end{propositionrep}
        \begin{proof}[Proof of Proposition \ref{prop:Z_E_is_sample-path_C^m}]
            The integrated Wiener process $W_m$ on $[0,l]$ is sample-path $m$-times differentiable \cite[see][Section 3]{Lachal_ClassBridgesIterated_2011}. As a consequence of the polynomial drift description for $B_m$ \cite[Theorem 4.1]{Lachal_ClassBridgesIterated_2011}, we also get that $B_m$ is sample-path $m$-times differentiable on $[0,l]$. The Kirchhoff conditions are trivially satisfied as the value and all the first $m$ derivatives of $B_m$ at $\curly{0,l}$ are null. This translates immediately on the derivatives of $Z_E$ at the vertices. 
        \end{proof}

    \subsection{Compendium}
        We are finally able to re-define the process $Z$ on $G$.
        \begin{definition}
            \label{def:process_Z_construction}
            Let $G$ be a metric graph and let $Z_\mu$ and $Z_E$ as defined above. Then we define
            \begin{equation}
                \label{eq:process_Z_additive_definition}
                Z(x) := Z_\mu(x) + Z_E(x).
            \end{equation}
        \end{definition}
        \begin{remark}
            Being the sum of two sample-path $C^m$ processes, we have that also $Z\in C^m(G)$ in the sample-path sense.
        \end{remark}
        \begin{remark}
            Since $Z_\mu \Ind Z_E$, the covariance (variogram) of $Z$ is the sum of the covariances (variograms) of $Z_\mu$ and $Z_E$:
            \begin{align}
                \label{eq:cov_Z} &C_Z(x_1,x_2) = C_\mu(x_1,x_2) + C_E(x_1,x_2),\\
                \label{eq:variog_Z} &\gamma_p(x_1,x_2) = \gamma_{Z_\mu}(x_1,x_2) + \gamma_{Z_E}(x_1,x_2).
            \end{align}
        \end{remark}
        \begin{propositionrep}
            \label{prop:RKHS_of_Z_is_H}
            Let $Z$ the process defined in Definition \ref{def:process_Z_construction}. Then, the RKHS of $Z$ is $\mathcal{H} = (\mathcal F, \angled{\cdot,\cdot}_{\mathcal{F}})$, as per Definition \ref{def:our_RKHS}. In other words, the two processes are the same.
        \end{propositionrep}
        \begin{proof}[Proof of Proposition \ref{prop:RKHS_of_Z_is_H}]
            Lemma \ref{lem:RKHS_Lemma3_Anderes} (C) has shown that $C_Z$ is the reproducing kernel of $\mathcal{H}$: being $C_Z$ the covariance function of $Z$ on $G$, this ensures that $\mathcal{H}$ is \emph{the} RKHS associated to $Z$.
        \end{proof}

        \begin{remark}
            \label{rem:euclidean_edges_not_a_problem}
            The definition of $Z_\mu$ (and, consequently, of $Z$) may appear to be dependent on the particular choice of the vertices. However, Proposition \ref{prop:RKHS_of_Z_is_H}, in concert with Proposition \ref{prop:invariance}, shows that the variogram $\gamma_p$ and, as a consequence, any process $Y=\Sch(Z)$ do not depend on the \emph{representation} of the metric graph (\ie on the presence of vertices of degree $2$), but are intrinsic properties of the graph itself. Of course, the explicit construction given in this section requires a Laplacian $\matr L$ built on graphs that have no self-loops or repeated edges, but this can easily be circumvented by sufficiently splitting such edges. 
        \end{remark}

\newpage

\section{Technical result}
    \label{app:technical_results}


    Here we collect several useful results necessary to show the ones presented in the main text.
    
    \begin{lemmarep}
        \label{lem:quadratic_kirchhoff_is_constant}
        Let $p \in C^2(G)$ be a function. If $p''$ is constant on each $e \in E$, then $p$ is constant on $G$.
    \end{lemmarep}
    \begin{proof}[Proof of Lemma \ref{lem:quadratic_kirchhoff_is_constant}]
        On $\R$, the only functions with constant second derivative are the degree-2 polynomials. This plainly translates into each $e\in E$: $p_e$ is a parabola. In addition, since $p$ satisfies the second-derivative Kirchhoff condition, the curvature $c$ must be constant on all the graph. Since $G$ is a compact set, there will be a maximum and a minimum: call them $x_+$ and $x_-$, respectively. \par
        We now show that the curvature $c$ should be non-positive in a neighbour of $x_+$. We consider two cases.
        \begin{itemize}
            \item If $x_+ \not\in V$ (\ie $x_+ \in e$), then $c\leq 0$ as $x_+$ is the local maximum of a parabolic function.
            \item If, instead, $x_+ \in V$, the order-1 Kirchhoff condition applies on $x_+$. Since the sum of exiting derivatives at $x_+$ should be null, then there exist (at least) one edge $e=(x_+, v)\in E_{x_+}$ such that the exiting derivative $\partial_e p(x_+) \geq 0$. This, in turn, implies that $c\leq 0$ in a neighbour of $x_+$ (otherwise, $p(v) = p(x_+) + \partial_e p(x_+) \, l_e + c \, l_e^2/2 > p(x_+)$, contradicting the maximum hypothesis on $x_+$).
        \end{itemize}
         This show that $c\leq 0$ on a neighbour of $x_+$ and, since $c$ is constant on $G$, we have that $c\leq 0$ on $G$. The same reasoning applies on $x_-$, leading to the conclusion that $c\geq 0$ on $G$. Therefore $c=0$ on $G$.\par
         Therefore $p$ is an affine function on each $e \in E$. Since $p$ is affine on each edge, both $x_+$ and $x_-$ should belong to $V$. Again, the order-1 Kirchhoff condition applies on $x_+$. This means that, $\forall v\sim x_+$, $p(v) = p(x_+)$, that is: $v$ is a global maximum too. Thus, $\forall w\sim v$, $p(w) = p(v) = p(x_+)$. Since $G$ is connected, there is a path connecting $x_+$ to $x_-$: it is sufficient to apply the above idea to all the vertices on this path to conclude that $p(x_-) = p(x_+)$, videlicet $p$ is constant on $G$.
    \end{proof}
    
    \begin{lemmarep}
        \label{lem:f_zero_m+1_integ_implies_constant}
        Let $f \in C^m(G)$ with $f^{(m)}$ absolutely continuous and such that
        \begin{equation*}
            \sum_{e \in E} \integ{0}{l_e}{\round{f^{(m+1)}(x)}^2}{x} = 0.
        \end{equation*}
        Then $f$ is constant on $G$.
    \end{lemmarep}
    \begin{proof}[Proof of Lemma \ref{lem:f_zero_m+1_integ_implies_constant}]
        The main hypothesis implies that $f^{(m+1)}(x) = 0$ a.e. on $G$. We consider two cases.
        \begin{itemize}
            \item \textbf{$m$ even.} We have that $f^{(m)}$ is constant on $G$. By applying Lemma \ref{lem:quadratic_kirchhoff_is_constant}, we get that $f^{(m-2)}$ is constant on $G$, which in turn implies that $f^{(m-4)}$ is constant on $G$, and so on. At the end we get that $f$ is constant on $G$.
            \item \textbf{$m$ odd.} In this case, we apply Lemma \ref{lem:quadratic_kirchhoff_is_constant} directly on $f^{(m+1)}=0$ (which is trivially constant). Again, by repeating the argument, we end up with $f$ constant on $G$.
        \end{itemize}
    \end{proof}


    \begin{lemmarep}[See {\citet[Lemma 1]{AnderesEtAl_IsotropicCovarianceFunctions_2020}}]
        \label{lem:RKHS_Lemma1_Anderes}
        $\mathcal H$ is an inner product vector space, with metric given by
        \begin{equation*}
            \norm{f}_{\mathcal F}^2 = \sum_{e \in E} \integ{0}{l_e}{\round{f_e^{(p)}(x)}^2}{x}.
        \end{equation*}
    \end{lemmarep}
    \begin{proof}[Proof of Lemma \ref{lem:RKHS_Lemma1_Anderes}]
        The metric expression comes directly from (\ref{eq:dot_product_RKHS}). In addition $\norm{f}_{\mathcal F}^2 = 0$ implies
        \begin{equation*}
            \sum_{e \in E} \integ{0}{l_e}{\round{f_e^{(p)}(x)}^2}{x} = 0,
        \end{equation*}
        from which we obtain (see Lemma \ref{lem:f_zero_m+1_integ_implies_constant}) that $f$ is constant on $G$. But since $\sum_v f(v)=0$, we conclude that $f=0$ on $G$. Symmetry, bilinearity and positive definiteness come directly from (\ref{eq:dot_product_RKHS}).
    \end{proof}

    \begin{definition}
        \label{def:projectors_P_mu_and_P_e}
        We consider the following projectors $\mathcal{P}_\mu, \mathcal{P}_e:\mathcal{F} \to \mathcal{F}$, for $e\in E$. For $f \in \mathcal{F}$, we define $\mathcal{P}_\mu(f)$ as the piecewise ($2m+1$)-degree polynomial function that interpolates the derivatives of $f$ at $\ul e, \ol e$ on each edge $e$. In addition, on each edge $e\in E$, we define $\mathcal{P}_e$ as follows:
        \begin{equation}
            \label{eq:definition_P_e}
            \mathcal{P}_e f(x) := \begin{cases}
                f(x) - \round{\mathcal{P}_\mu f}(x) &\qquad\text{on $x\in e$}\\
                0 &\qquad\text{on $x\not\in e$.}
            \end{cases}
        \end{equation}
        For the sake of simplicity, we set:
        \begin{align*}
            &\mathcal{F}_\mu := \mathcal{P}_\mu \mathcal F & & \angled{\cdot,\cdot}_\mu := \angled{\cdot,\cdot}_\mathcal{F}\big|_{\mathcal{F}_\mu \times \mathcal{F}_\mu} & & f_\mu := \mathcal{P}_\mu f,\\
            &\mathcal{F}_e := \mathcal{P}_e \mathcal F & & \angled{\cdot,\cdot}_e := \angled{\cdot,\cdot}_\mathcal{F}\big|_{\mathcal{F}_e \times \mathcal{F}_e} & & f_e := \mathcal{P}_e f.
        \end{align*}
    \end{definition}
    
    \begin{lemmarep}
        \label{lem:RKHS_Lemma2_Anderes}
        See {\citet[Lemma 2]{AnderesEtAl_IsotropicCovarianceFunctions_2020}}. Let $G$ be a metric graph, then $\mathcal{P}_\mu$ and all the $\mathcal{P}_e$ are mutually-orthogonal projectors (\ie $\mathcal P^2 = \mathcal{P}$). In addition,
        \begin{equation}
            \label{eq:direct_sum}
            \mathcal{F} = \mathcal{F}_\mu \oplus \bigoplus_{e \in E} \mathcal{F}_e.
        \end{equation}
    \end{lemmarep}
    \begin{proof}[Proof of Lemma \ref{lem:RKHS_Lemma2_Anderes}]
        It is immediate to show that $\mathcal{P}_\mu$ is a projector. Indeed, on each edge $e\in E$, $\mathcal{P}_\mu(f)\big|_e$ is the interpolation polynomial that matches the value and the derivatives at $\ul e, \ol e$. As a consequence, applying $\mathcal{P}^2_\mu(f)\big|_e$ means applying $\mathcal P_\mu$ on that polynomial, which is a fixed point for $\mathcal{P}_\mu$. Since this applies on each $\bigcup_{e\in E} e = G$, we have that $\mathcal{P}_\mu$ is a projector on $G$. Now, let $f \in \mathcal{F}$ and $e\in E$: on $x\in e$ we have:
        \begin{align*}
            \mathcal{P}_e^2 \ f &= (\mathcal{I} - \mathcal{P}_\mu)^2 \ f = (\mathcal{I}^2 - 2\mathcal{P}_\mu + \mathcal{P}_\mu^2) \ f = (\mathcal{I} - \mathcal{P}_\mu) f = \mathcal{P}_e \ f;
        \end{align*}
        while for $x\not \in e$, trivially $\mathcal{P}_e^2 f = 0 = \mathcal{P}_e f$. Therefore $\mathcal{P}_e$ is a projector as well. \par
        Notice that, for any $f\in \mathcal{F}$,
        \begin{equation*}
            f = \mathcal{P}_\mu f + \sum_{e\in E} \mathcal{P}_e f,
        \end{equation*}
        which comes directly from the fact that, on each $e$, $\mathcal{P}_\mu + \mathcal{P}_e = \mathcal{I}$. Finally, for any $e\in E$, we can integrate by parts $p$ times: the boundary terms vanish because all the derivatives of $\mathcal{P}_e$ at the vertices are null. Thus, we have
        \begin{align*}
            \angled{\mathcal{P}_\mu f, \mathcal{P}_e f} = \angled{L^{p} (\mathcal{P}_\mu f), \mathcal{P}_e f}_{L^2}=\angled{0, \mathcal{P}_e f}_{L^2} = 0,
        \end{align*}
        since $\mathcal{P}_\mu f$ is a polynomial of degree not greater than $2p-1$, thus its ($p$)-order derivative is null.
        This shows (\ref{eq:direct_sum}) and concludes the proof.
    \end{proof}

    \begin{lemmarep}[See {\citet[Lemma 3]{AnderesEtAl_IsotropicCovarianceFunctions_2020}}]
        \label{lem:RKHS_Lemma3_Anderes}
        We introduce the following notation.
        \begin{itemize}
            \item Let $\mathcal H_\mu := (\vspan \matr L, \angled{\cdot,\cdot}_{\matr L})$ denote the finite-dimensional Hilbert space with inner product given by $\angled{\dVec z_1, \dVec z_2}_{\matr L} := \dVec z_1^\top \matr L \dVec z_2$. Here $\vspan \matr L = \curly{\dVec v \in \R^{d_\dVec z} \ : \ \dVec j^\top \dVec v = 0}$ denotes the range or column space of $\matr L$.
            \item Let $\mathcal H_e$ denote the infinite-dimensional Hilbert space of $C^m$ functions on $e=[0,l(e)]$ such that $f^{(m)}$ is absolutely continuous and $f^{(p)}\in L^2([0,l_e])$ and, for each $i\in\curly{0,\dots,m}$, $f^{(i)}(0)=f^{(i)}(l_e)=0$. We endow $\mathcal H_e$ with the dot product
            \begin{equation*}
                \angled{f,g}_{\mathcal{H}_e} := \integ{0}{l_e}{f^{(p)}(x) \ g^{(p)}(x)}{x}.
            \end{equation*}
            Then, we have the following:
            \begin{enumerate}[label=\normalfont(\Alph*)]
                \item $\round{\mathcal F_\mu, \angled{\cdot,\cdot}_\mu}$ is a finite-dimensional Hilbert space, which is isomorphic to $\mathcal H_\mu$ and has reproducing kernel $C_\mu$ defined in (\ref{eq:cov_Z_mu}). In particular, for all $f,g \in \mathcal{F}_\mu$, it holds:
                \begin{equation}
                    \label{eq:simplified_dot_product_on_H_mu}
                    \angled{f,g}_\mu = \dVec z_f^\top \matr L \dVec z_g,
                \end{equation}
                where $\dVec z_f, \dVec z_g$ store the values of the derivatives of $f,g$ at $V$, as described in Notation \ref{notat:z_vector}.
                \item For each $e\in E$, $\round{\mathcal{F}_e, \angled{\cdot,\cdot}_{e}}$ is a infinite-dimensional Hilbert space, which is isomorphic to $\mathcal{H}_e$ and has reproducing kernel $C_e\big|_e$, where $C_e$ is defined in (\ref{eq:cov_Z_E}). In particular, for all $f,g \in \mathcal{F}_e$, it holds:
                \begin{equation}
                    \angled{f,g}_{e} = \angled{f,g}_{\mathcal H_e}.
                \end{equation}
                \item $\mathcal H$ is a infinite-dimensional Hilbert space, which is isomorphic to $\R^{d_{\dVec z}} \oplus \bigoplus_{e \in E} \mathcal{H}_e$ and has reproducing kernel $C_Z$ defined in (\ref{eq:cov_Z}). In particular, for all $f,g \in \mathcal{F}$, it holds:
                \begin{equation}
                    \angled{f,g}_\mathcal{F} = \angled{f_\mu, g_\mu}_\mu + \sum_{e \in E} \angled{f_e, g_e}_e.
                \end{equation}
            \end{enumerate}
        \end{itemize}
    \end{lemmarep}
    
    \begin{proof}[Proof of Lemma \ref{lem:RKHS_Lemma3_Anderes}]
        \begin{enumerate}[label=\normalfont(\Alph*)]
            \item Let $f \in \mathcal F_\mu$. Being $f$ a piecewise polynomial of degree $2p-1$ on each edge, it is characterised by its derivatives at the endpoints, therefore, all the information in $f$ is contained in $\dVec z_f \in \vspan \matr L$. This defines the bijection between the two spaces $\round{\mathcal F_\mu, \angled{\cdot,\cdot}_\mu}$ and $\mathcal H_\mu$. Now we need to show that such a bijection is an isomorphism. Let $f,g \in \mathcal{F}_\mu$ and $e\in E$: clearly, $f,g$ on $e$ are completely determined by $\matr H_e \dVec z_f$ and $\matr H_e \dVec z_g$. Now we can invoke the definition of $\matr M_e$ (see Definition \ref{def:matrix_M}) to have
            \begin{equation*}
                \integ{0}{l_e}{f_e^{(p)}(x) \ g_e^{(p)}(x)}{x} = \round{\matr H_e \dVec z_f}^\top \matr M_e \round{\matr H_e \dVec z_g}.
            \end{equation*}
            Therefore, starting from (\ref{eq:dot_product_RKHS}), we get:
            \begin{align*}
                \angled{f,g}_\mathcal{F} &= \sum_{e \in E} \integ{0}{l_e}{f_e^{(p)}(x) \ g_e^{(p)}(x)}{x}\\
                &=\sum_{e \in E} \dVec z_f^\top \matr H_e^\top \matr M_e \matr H_e \dVec z_g\\
                &= \dVec z_f^\top \round{\sum_{e\in E} \matr H_e^\top \matr M_e \matr H_e} \dVec z_g \\
                &= \dVec z_f^\top \matr L \dVec z_g.
            \end{align*}
            This shows (\ref{eq:simplified_dot_product_on_H_mu}) and that $\round{\mathcal F_\mu, \angled{\cdot,\cdot}_\mu}$ is isomorphic to $\mathcal{H}_\mu$. It remains to show that $\mathcal H_\mu$ is the RKHS of $C_\mu$, namely that, for each $f \in \mathcal{F}_\mu$ and for each $x\in G$, $f(x) = \angled{f, C_\mu(\cdot, x)}_\mu$. We are going to use the just-shown isomorphism (\ref{eq:simplified_dot_product_on_H_mu}) between Hilbert spaces. Recall the expression of $C_\mu$ (\ref{eq:cov_Z_mu}):
            \begin{align*}
                C_\mu(y,x)&=\dVec y^\top \matr B_{e_y}^{-1} \matr H_{e_y} \matr L^- \matr H_{e_x}^\top \matr B_{e_x}^{-\top} \dVec x.
            \end{align*}
            This implies that the vector $\dVec z_{C_\mu(\cdot, x)}$, which stores all the information of $C_\mu(\cdot, x)$, reads
            \begin{equation*}
                \dVec z_{C_\mu(\cdot, x)} = \matr L^- \matr H_{e_x}^\top \matr B_{e_x}^{-\top} \dVec x.
            \end{equation*}
            Therefore, by (\ref{eq:simplified_dot_product_on_H_mu}), we get
            \begin{align*}
                \angled{f,C_\mu(\cdot, x)} &= \dVec z_f^\top \matr L \dVec z_{C_\mu(\cdot, x)} = \dVec z_f^\top \matr L \matr L^- \matr H_{e_x}^\top \matr B_{e_x}^{-\top} \dVec x \\
                &= \dVec z_f^\top \matr H_{e_x}^\top \matr B_{e_x}^{-\top} \dVec x = f(x).
            \end{align*}
            \item Let $f \in \mathcal{F}_e$. This means that $f=0$ on $G \setminus e$ and that the value and all the first $m$ derivatives are null in $\ul e$ and $\ol e$. Clearly this is a bijective map between $\mathcal F_e$ and the function space of $\mathcal{H}_e$. In addition, if $f,g \in \mathcal{F}_e$ and $f',g'$ are their restriction on the domain $e$ (and therefore on $[0,l_e]$), it is immediate that:
            \begin{align*}
                \angled{f,g}_e &= \angled{f,g}_\mathcal{F} = \sum_{e'\in E} \integ{0}{l_{e'}}{f_{e'}^{(p)}(x) \ g_{e'}^{(p)(x)}}{x} \\
                &= \integ{0}{l_e}{f_e^{(p)}(x) \ g_e^{(p)}(x)}{x} = \angled{f', g'}_{\mathcal{H}_e}.
            \end{align*}
            To show that $\mathcal{H}_e$ is the RKHS of $C_e$, we start by recalling that the RKHS of the unconstrained $m$-times integrated Wiener process $W_m$ is the space of $C^m$ functions on $[0,l(e)]$ such that $f^{(m)}$ is absolutely continuous and $f^{(p)}\in L^2([0,l_e])$ and, for each $i\in\curly{0,\dots,m}$, $f^{(i)}(0)=0$ \cite[see][Example 10]{vanZantenvanderVaart_ReproducingKernelHilbert_2008}. The endpoint-conditioned process is a generalised Gaussian bridge in the sense of \cite[Definition 1.3]{SottinenYazigi_GeneralizedGaussianBridges_2014}. By their orthogonal bridge representation, Gaussian conditioning on finitely many path functionals is the Hilbert-space projection removing the span of the conditioned Gaussian variables. Hence, the RKHS of the bridge is the closed subspace of the RKHS of $W_m$ satisfying the corresponding endpoint constraints.
            \item This follows directly from Lemma \ref{lem:RKHS_Lemma2_Anderes} and {\normalfont(A),(B)}.
        \end{enumerate}
    \end{proof}

    \begin{lemmarep}
        \label{lem:int_by_parts_cancels_non-int_term}
        Let $f\in C^{m_1}(G)$ and $g\in C^{m_2}(G)$. Then, for each $0\leq i \leq m_1$ and $0\leq j \leq m_2$ such that $i+j$ is odd, it holds:
        \begin{equation}
            \label{eq:int_by_parts_cancels_non-int_term}
            \sum_{e\in E} \round{f_e^{(i)}(\ol e) \ g_e^{(j)}(\ol e) - f_e^{(i)}(\ul e) \ g_e^{(j)}(\ul e)} = 0.
        \end{equation}
    \end{lemmarep}
    \begin{proof}[Proof of Lemma \ref{lem:int_by_parts_cancels_non-int_term}]
        Since $i+j$ is odd, either $i$ is even and $j$ is odd or the opposite. Without loss of generality, we assume the former case (if that was not the case, simply invert the roles of $f$ and $g$). Now, the sum (\ref{eq:int_by_parts_cancels_non-int_term}) is essentially taking the sum over all the vertices: 
        \begin{align*}
            \sum_{e\in E} & \round{f_e^{(i)}(\ol e) \ g_e^{(j)}(\ol e) - f_e^{(i)}(\ul e) \ g_e^{(j)}(\ul e)} \\
            &= \sum_{e=(v,w)\in E} \round{\frac{\partial^i}{\partial_e x^i}f(w) \ \frac{\partial^j}{\partial_e x^j}g(w) - \frac{\partial^i}{\partial_e x^i}f(v) \ \frac{\partial^j}{\partial_e x^j}g(v)}\\
            &=-\sum_{v\in V} \sum_{e \in E_v} \frac{\partial^i}{\partial_e x^i}f(v) \ \frac{\partial^j}{\partial_e x^j}g(v)\\
            &=-\sum_{v\in V} \round{\frac{\partial^i}{\partial x^i}f(v) \sum_{e \in E_v}  \frac{\partial^j}{\partial_e x^j}g(v)}=-\sum_{v\in V} \frac{\partial^i}{\partial x^i}f(v) \cdot 0 = 0.
        \end{align*}
        In the last line, we have used the fact that $i$ is even, therefore all the exiting $i$-order derivatives at each vertex $V$ are equal: $\frac{\partial^i f(v)}{\partial_e x^i} =: \frac{\partial^i f(v)}{\partial x^i}$, regardless of $e\in E_v$. In addition, we used the fact that, being $j$ odd, the sum of the exiting derivatives is null.
    \end{proof}
    
    \begin{lemmarep}
        \label{lem:our_dot_to_L2}
        Let $m\in \N$, and $f,g\in \mathcal{F}$, where, in addition, $f\in C^{2m+2}(G)$. Then, it holds:
        \begin{equation}
            \label{eq:our_dot_to_L2}
            \angled{f,g}_{\mathcal{F}} = \angled{L^{p}f,g}_{L^2}.
        \end{equation}
    \end{lemmarep}
    \begin{proof}[Proof of Lemma \ref{lem:our_dot_to_L2}]
        We are going to apply the integration by parts formula $p$ times in concert with Lemma \ref{lem:int_by_parts_cancels_non-int_term}.
        \begin{align*}
            \angled{f,g}_{\mathcal{F}} &= \sum_{e\in E} \integ{0}{l_e}{f_e^{(p)} \ g_e^{(p)}}{x}\\
            &=\overbrace{\sum_{e\in E} \square{f_e^{(p)} \ g_e^{(m)}}_0^{l_e}}^{=0} -\sum_{e\in E} \integ{0}{l_e}{f_e^{(m+2)} \ g_e^{(m)}}{x}\\
            &=-\round{\overbrace{\sum_{e\in E} \square{f_e^{(m+2)} \ g_e^{(m-1)}}_0^{l_e}}^{=0} -\sum_{e\in E} \integ{0}{l_e}{f_e^{(m+3)} \ g_e^{(m-1)}}{x}}\\
            &=...= (-1)^{p} \sum_{e\in E} \integ{0}{l_e}{f_e^{(2m+2)} \ g_e^{(0)}}{x}\\
            &=\sum_{e\in E} \integ{0}{l_e}{\round{L^{p}f_e} \ g_e}{x} = \angled{L^{p}f,g}_{L^2}.
        \end{align*}
    \end{proof}

\end{toappendix}

\begin{toappendix}
    \newpage
    \section{Proofs of the main results}
    \label{app:proofs}
    \begingroup\makeatletter
    \let\axp@oldsection\subsection
    \makeatother
\end{toappendix}

\section{Introduction}
    In recent years there has been an increased interest in the definition of stochastic processes on metric graphs, these being (undirected and weighted) graphs, in which every edge represents not only a sheer relation between vertices, but also has a \emph{physical} meaning: edges are identified with real segments of prescribed positive length. Metric graphs can be used to describe domains such as river networks or road networks \citep{CressieEtAl_SpatialPredictionRiver_2006a, OkabeSugihara_SpatialAnalysisNetworks_2012}. \par
    One of the main challenges of defining continuously-indexed stochastic processes on metric graphs is the presence of branches: they make it inappropriate to embed the network in a Euclidean space (usually $\R^2$) and define covariances between locations as covariances of well-studied processes defined on the whole Euclidean space at those locations. Indeed, even though this strategy produces valid covariance functions, they may not represent the graph topology appropriately. For instance, two river locations could be spatially close, but very distant on the shortest-path river-network topology. As a consequence, if we want to model some phenomena that lives entirely on the river network, the embedding approach may be inappropriate. Two complementary approaches have been explored to study stochastic processes on metric graphs. The former defines a family of partial stochastic differential equations (SDE) on the metric graph and studies the properties of its solution \citep{BolinEtAl_GaussianWhittleMatern_2024, BolinEtAl_NewClassNonstationary_2026}. The latter directly models the covariance function via Hilbert-space embeddings: the covariance function is defined as a composition of either the shortest-path or the resistance metric with a suitable family of functions \citep{AnderesEtAl_IsotropicCovarianceFunctions_2020, FilosiEtAl_TemporallyEvolvingGeneralisedNetworks_2024}. Up to now, the choices of distances used in this embeddings have produced classes of covariance functions that could model only continuous, yet non-differentiable processes. On the contrary, the SDE approach has been able to define processes with arbitrary differentiability \citep{BolinEtAl_NewClassNonstationary_2026}.\par 
    In this paper, we change the geometry of a metric graph by defining a new class of distances (called polyharmonic throughout), which extend both the effective resistance distance defined in \cite{AnderesEtAl_IsotropicCovarianceFunctions_2020} and the less-studied biharmonic distance introduced on two-dimensional surfaces in \cite{LipmanEtAl_BiharmonicDistance_2010} and recently defined on combinatorial graphs \cite{BlackEtAl_BiharmonicDistanceGraphs_2025}. They will be defined as the square root of variograms of a suitable class of stochastic processes and will enjoy the following spectral representation:
    \begin{equation}
        \label{eq:intro_spectral_form_d}
        d_p^2(x_1,x_2) = \sum_{k=1}^{\infty} \frac{\round{\phi_k(x_1) - \phi_k(x_2)}^2}{\lambda_k^{p}}, \qquad x_1,x_2\in G
    \end{equation}
    where $\curly{(\lambda_k,\phi_k)}_k$ are the non-zero eigenvalues and associated eigenfunctions of the Kirchhoff-Laplace operator defined on the metric graph $G$ and provide an orthonormal basis of $L^2(G)$. In particular $d_1^2$ is the effective resistance distance defined in \cite{AnderesEtAl_IsotropicCovarianceFunctions_2020}, whilst $d_2$ is the biharmonic distance. The parameter $p\in\N^+$ plays a crucial role and will be related to the differentiability of the isotropic process having covariance $\psi(d_p^2(x_1,x_2))$, being $\psi:[0,+\infty)\to\R$ a non-constant completely monotonic function. \par
    The distance $d_p$ is associated to the $p$-order energy $\integ{G}{}{(f^{(p)}(x))^2}{x}$ and both its definition (via a reproducing kernel Hilbert space having squared norm such energy) and variational characterisation will rely on it. Moreover, we prove that it is invariant under changes of graph representations. These properties provide a nice interpretation of the polyharmonic distance $d_p$ and a sound motivation. Furthermore, we prove the validity of the class of covariance functions $\psi(d_p^2(x_1,x_2))$ on the whole graph $G$. Next, we characterise the differentiability properties of the process defined via such a covariance function (by means of the celebrated Shoenberg theorem \cite{Schoenberg_MetricSpacesPositive_1938}), in terms of the order $p$ of the distance and of the derivatives of $\psi$ in a right-neighbourhood of $0$: we will show that, under appropriate assumptions, such processes will inherit a mean-square ($p-1$)-order differentiability on the interior of each edge, while satisfying Kirchhoff conditions of order $1$ at the vertices. In addition, we that the auxiliary process $Z$ (whose variogram is $d_p^2$) is a centred version of the limit for $\kappa\to 0$ of the Whittle-Mat\`ern random fields defined in \cite{BolinEtAl_GaussianWhittleMatern_2024}. We stress that the two classes of processes (the auxiliary ones defining the polyharmonic distances and the resulting isotropic ones) are different and the properties of one do not transfer automatically to the other. \par 
    The remainder of the paper is organised as follows. Section \ref{sec:background} presents the necessary mathematical framework and objects involved in our construction. Section \ref{sec:polyharmonic_distances} defines polyharmonic distances, provides their variational and spectral representations and proves some of their properties. Section \ref{sec:smooth_isotropic_covariances} provides the link between the differentiability properties of a process defined on the real line and on metric graphs and of the ones having isotropic covariance functions w.r.t. its variogram. Furthermore, it provides some negative results. Section \ref{sec:examples} presents some simple graph examples and their polyharmonic distances for some values of $p$. Section \ref{sec:conclusion} concludes the manuscript. \par
    In addition, Appendix \ref{app:explicit_construction_Z} provides an explicit constructive definition of the auxiliary process $Z$. Appendix \ref{app:technical_results} is devoted to additional technical results useful to show the main-text ones. Finally, Appendix \ref{app:proofs} collects all the proofs of the main results.

\section{Mathematical background}
\label{sec:background}
\subsection{Metric graphs}
    Linear networks are collections of segments in the real plane $\R^2$, resulting in a graph having \emph{physical} points on the edges. These topologies have been vastly used in spatial statistics over the last decades \citep[see, for instance,][]{CressieEtAl_SpatialPredictionRiver_2006a,OkabeSugihara_SpatialAnalysisNetworks_2012,BaddeleyEtAl_AnalysingPointPatterns_2021}. However, the constraint of living in $\R^2$ brings several restrictions: first, each edge length and shape are fixed given its endpoints. In addition, the represented graph cannot have crossing and non-intersecting edges (that may represent tunnels or bridges). As a consequence, linear networks may not be appropriate to model several real-world networks. Metric graphs overcome these issues by removing the $\R^2$-embedding of linear networks: indeed metric graphs can be seen as collections of segments with an equivalence relation on their endpoints that defines which edges share an endpoint, that is a node \citep[see, for instance,][Section 1]{Mugnolo_WhatActuallyMetric_2021}. This is a quite abstract definition and allows to build graphs with arbitrary edge lengths, self-loops (edges connecting a node with itself) and multiple edges between the same pair of vertices. Throughout the manuscript, we will assume that $G$ is finite and connected and each edge $e\in E$ has positive, finite length $l_e$. In addition, for the sake of simplicity, we will assume that the $G$ is \emph{simple}, \ie it has no self-loops or multiple edges. Although this may seem restrictive, it is a hypothesis merely aimed at simplifying the exposition. Indeed, we will prove that the distance family will not depend on addition or removal of degree-$2$ vertices. As a consequence, since for any non-simple $G$ there is always another equivalent simple metric graph (simply add vertices on the ``non-simple'' edges), the distances and processes defined here are well-defined also on non-simple metric graph.

\subsection{Isotropic processes on metric graphs}
    \label{ssec:covariance_functions_stochastic_processes}
    \cite{AnderesEtAl_IsotropicCovarianceFunctions_2020} extended the so-called resistance distance (see Subsection \ref{ssec:resistance_and_biharmonic_distances}), usually defined between the vertices of a graph, to the whole topology of metric graphs. In addition, by means of the Schoenberg theorem (\cite{Schoenberg_MetricSpacesPositive_1938}), stated below, they were able to define a class of isotropic (w.r.t. such distance) covariance functions valid on any metric graph. As we are going to use the same argument, we report \cite[Theorem 6]{AnderesEtAl_IsotropicCovarianceFunctions_2020}, together with \cite[Theorem 8]{AnderesEtAl_IsotropicCovarianceFunctions_2020}.\par
    \begin{theorem}
        \label{theo:schoenberg}
        Let $(X,d)$ be a distance space and $x_0\in X$. Then, the following statements are equivalent.
        \begin{enumerate}
            \item $(X,d) \overset{\rm id}{\hookrightarrow} \mathcal H$, for some Hilbert space $\mathcal H$, namely there exist $\varphi:X\to \mathcal{H}$ such that
            \begin{equation*}
                d(x_1,x_2) = \norm{\varphi(x_1) - \varphi(x_2)}_{\mathcal{H}}.
            \end{equation*}
            \item $d^2(x_1, x_0) + d^2(x_2,x_0) - d^2(x_1,x_2)$ is positive definite on $X\times X$.
            \item \label{item:shoenberg_CM_is_PD} For every non-constant completely monotonic function $\psi:[0,+\infty]\to \R$ such that $\psi(0) < +\infty$, the function 
            \begin{equation}
                (x_1,x_2) \mapsto \psi(d^2(x_1,x_2))
            \end{equation}
            is strictly positive definite on $X\times X$.
            \item \label{item:shoenberg_CND} $d^2$ is conditionally negative definite, that is: for all $n\in\N^+$, $\dVec x \in X^n$ and $\dVec c \in \R^n$ such that $\dVec 1_n^\top \dVec c = 0$, it holds
            \begin{equation*}
                \sum_{i,j=1}^n c_i \, c_j \, d^2(x_i,x_j) \leq 0.
            \end{equation*}
        \end{enumerate}
    \end{theorem}
    To this article, the implication $(\ref{item:shoenberg_CND}) \implies (\ref{item:shoenberg_CM_is_PD})$ is of particular interest, since the variogram of every square-integrable process on $X$ is conditionally negative definite. As a consequence, Theorem \ref{theo:schoenberg} provides the following useful recipe to build covariance functions on arbitrary sets.
    \begin{tcolorbox}[
        colback=white!95!black,
        colframe=white!50!black,
        title = {Isotropic covariance construction on a set $X$},
        fonttitle = \bfseries,
        width = 0.9\textwidth,
        center
    ]
        \begin{itemize}[leftmargin=*]
            \item Define a square-integrable process $Z: X \to \R$ such that, if $x_1\ne x_2$, then $\Var\round{Z(x_1) - Z(x_2)} > 0$.
            \item Define the distance $d(x_1,x_2) := \sqrt{\Var\round{Z(x_1) - Z(x_2)}}$.
            \item Pick a non-constant completely monotonic function $\psi:[0,+\infty)\to \R$ such that $\psi(0) < +\infty$.
            \item Define $C_Y:X\times X\to\R$ as
            \begin{equation*}
                C_Y(x_1,x_2):=\psi(d^2(x_1,x_2)).
            \end{equation*}
        \end{itemize}
    \end{tcolorbox}
    In the following, whenever $Y$ is the (unique in law) centred Gaussian process built from $Z$ with the above procedure, we will write $Y:=\Sch_\psi(Z)$. We stress that the \emph{separation} requirement ($x_1\ne x_2\implies \Var\round{Z(x_1) - Z(x_2)} > 0$) is needed to guarantee that $d$ separates the space (\ie $d(x_1,x_2) = 0 \implies x_1=x_2$), yet it is not required for the positive definiteness of $\psi(d^2(x_1,x_2))$.
\subsection{Resistance distance and biharmonic distance}
    \label{ssec:resistance_and_biharmonic_distances}
    Effective resistance distance has been widely used and analysed in graph literature \citep[see, for instance,][]{KleinRandic_ResistanceDistance_1993,JorgensenPearse_HilbertSpaceApproach_2010}: it defines the resistance between two nodes as the voltage drop resulting by injecting one Ampere of current in one and withdrawing it from the other. One of its advantages over the shortest-path distance is that it depends on \emph{how many} paths connect the two nodes: the more they are, the lower is the resistance, accordingly to the parallel resistors law. In addition, \cite{AnderesEtAl_IsotropicCovarianceFunctions_2020} have provided an ingenious method to define and compute it between arbitrary points of (a particular class of) metric graphs, not just between vertices. Furthermore, they have shown not only that it is possible to compose the resistance distance with a flexible class of functions to define isotropic stochastic processes indexed on the whole graph (vertices and edges), but also that such a method does not provide valid covariance functions when used with the shortest-path distance.\par
    Biharmonic distance has been introduced on two-dimensional surfaces by \cite{LipmanEtAl_BiharmonicDistance_2010} and belongs to a broader family of spectral distances \citep{PataneSpagnuolo_InteractiveAnalysisHarmonic_2013}. It has been recently specialised to (combinatorial) graphs and proven to be a good measurement of the importance of edges in the general graph topology by \cite{BlackEtAl_BiharmonicDistanceGraphs_2025}. In both cases, its definition involves the inverse second power of the Laplace (continuous or discrete) operator. 
    
\subsection{Differentiability on metric graphs and Kirchhoff conditions}
    Here we provide how the definition of differentiability translates on metric graphs. Indeed, a metric graph is not locally-Euclidean at the vertices having degree not equal to $2$. As a consequence, it is in order to devote special care to the definition of differentiability of functions (and processes) defined over metric graphs. The next definition introduces the concept of \emph{directional derivative}. 
    \begin{definition}[Exiting directional derivative]
        Let $G$ be a metric graph and let $f:G\to\R$ be a function. Let $x\in G$ be a point on $G$: we indicate $E_x$ the set of edges stemming from $x$. This can be formalised as follows: if $x$ is on an edge $(v_1,v_2)$, then $E_x:=\curly{(v_1,x),(x,v_2)}$. If $x$ is a vertex, then $E_x:=\curly{e \in E: \exists v\in V\, s.t. \,e=(x,v)}$. Now, we define the \emph{exiting} directional derivative of $f$ in $x\in G$ among the edge $e\in E_x$ as follows:
        \begin{equation}
            \label{eq:def_directional_derivative}
            \partial_e f(x) := \frac{\partial f(x)}{\partial_{e} x} := \lim_{h\to 0^+} \frac{f(x+he)-f(x)}{h},
        \end{equation}
        where $x+he$, for $h$ sufficiently small, represents the point on $e$ whose (shortest-path) distance from $x$ is $h$.
    \end{definition}
    Clearly, the operator $\partial_e$ can be repeated to define the $j$\textsuperscript{th}-order exiting derivative, denoted $\partial_e^j$ henceforth. We will say that a function belongs to $C^m(e)$ if, for each $x$ on the interior of $e$, the two exiting derivatives up to order $m$ exist finite and match (should be equal for even orders, opposite for odd orders). This is the usual concept of differentiability on $\R$. Next, we introduce the so-called Kirchhoff conditions, which define what differentiability is on the vertices.
    \begin{definition}[Class $C^m_j$ on a metric graph]
        \label{def:class_Cm_on_metric_graph}
        Let $G$ be a metric graph and let $f:G\to \R$ be a (non-random) function. For integers $0\leq j\leq m \leq \infty$, we say that $f\in C_j^m(G)$ if for each $e\in E$, $f|_e \in C^m(e)$; and for each vertex $v \in V$, $f$ satisfies all the \emph{Kirchhoff conditions} up to degree $j$, namely for all $i \in \curly{0,\dots, j}$:
        \begin{itemize}
            \item if $i$ is even, all the exiting derivatives should be equal:
            \begin{equation}
                \label{eq:Kirchhoff_conditions_even}
                \forall e_1,e_2\in E_v, \ \partial_{e_1}^{(i)} f(v) = \partial_{e_2}^{(i)} f(v);
            \end{equation}
            \item if $i$ is odd, the sum of the exiting derivatives should be null:
            \begin{equation}
                \label{eq:Kirchhoff_conditions_odd}
                \sum_{e\in E_v} \partial_e^{(i)} f(v) =0.
            \end{equation}
        \end{itemize}
        We remark that, if $v$ is a vertex of degree 2 (namely, it is actually a point on an edge), then the two Kirchhoff conditions coincide and are equivalent to the classical differentiability (left and right derivatives coincide).
        \par
        Let $Z:G \to \R$ be a stochastic process on $G$. We say that $Z\in C_j^m(G)$ if, for each $e\in E$, $Z|_E$ has continuous $m$-order derivative in the mean-square sense and if, on each $v\in V$, the Kirchhoff conditions (\ref{eq:Kirchhoff_conditions_even}) and (\ref{eq:Kirchhoff_conditions_odd}) hold almost surely for any $i\in\curly{0,...,j}$.
    \end{definition}
    
    In the following, we will mainly use $C_m^m$ and $C_1^m$: for this reason, we will write $C^m := C_m^m$. Furthermore, we will say that $f\in C_j^m(x)$,  if the conditions above hold only at $x\in G$.
    
\section{Polyharmonic resistance distances}
    \label{sec:polyharmonic_distances}
    We are going to define a class of distances of negative type on the graph $G$. To this purpose, we define the (squared) distance as the variogram of a suitable zero-mean Gaussian process, called $Z$ throughout, defined on $G$. \cite{AnderesEtAl_IsotropicCovarianceFunctions_2020} defined $Z$ be a continuous, yet not differentiable, process as the sum of two independent processes on $G$. Afterwards, they have shown \citep[Proposition 5]{AnderesEtAl_IsotropicCovarianceFunctions_2020} that such a process has a very clean RKHS. In this article, we aim to generalise their construction, by building a process $Z$ that has arbitrary differentiability order. However, since the explicit construction of $Z$ will be quite involved, we prefer to define the process $Z$ via its RKHS, and defer its explicit construction in Appendix \ref{app:explicit_construction_Z}. Anyhow, our construction will prove to be very similar to the one in \cite{AnderesEtAl_IsotropicCovarianceFunctions_2020}, yet produce a sufficiently-differentiable process.

    In the following, we will write $\integ{G}{}{f(x)}{x}$ as a shortcut for $\sum_{e\in E} \integ{0}{l_e}{f_e(x)}{x}$.
    \subsection{Definition and basic properties}
    Throughout this manuscript, we will set $m\in\N$ and $p:=m+1$. This double index is motivated by the different interpretations that these numbers have: whilst $m \in \N$ will be the differentiability order of $Z$ (and, as we will show, of $\Sch_\psi(Z)$), $p \in \N^+$ denotes the order of the dot product and a natural candidate for the distance. In the following, $Z$ will have subscript $m$ and $\gamma,d$ will have subscript $p$, unless otherwise stated. With this notation, for instance, $Z_0$ is  related to $\gamma_1,d_1$.
    \begin{definition}
        \label{def:our_RKHS}
        Let $G$ be a metric graph. Let $\mathcal{F}$ the set of functions $f:G\to \R$ such that:
        \begin{itemize}
            \item $f \in C^{m}(G)$,
            \item on each $e\in E$, $f_e^{(m)}$ is absolutely continuous and $f_e^{(m+1)} \in L^2([0, l_e])$,
            \item $\sum_{v\in V} f(v) = 0$.
        \end{itemize}In addition, consider the following quadratic form on $\mathcal{F}$:
        \begin{equation}
            \label{eq:dot_product_RKHS}
            \angled{f,g}_\mathcal{F} := \integ{G}{}{f^{(m+1)}(x) \ g^{(m+1)}(x)}{x}.
        \end{equation}
        We will write $\mathcal H := \round{\mathcal F, \angled{\cdot,\cdot}_{\mathcal F}}$.
    \end{definition}
        
    \begin{propositionrep}
        \label{prop:RKHS_evaluation_functional_continuous}
        $\mathcal{H}$ is a RKHS, that is: it is a Hilbert space and the point-evaluation functional $f \mapsto f(x)$
        is continuous in $\mathcal{H}$.
    \end{propositionrep}
    \begin{proof}[Proof of Proposition \ref{prop:RKHS_evaluation_functional_continuous}]
        Let us show that $\angled{\cdot,\cdot}_{\mathcal F}$ is an inner product. From (\ref{eq:dot_product_RKHS}), it is immediate to show that it is symmetric, bilinear and positive semidefinite. Now, assume that $\angled{f,f}_{\mathcal{F}}=0$, then, by Lemma \ref{lem:f_zero_m+1_integ_implies_constant}, we have that $f$ is constant on $G$, that is $f=0$ (recall that $\sum_v f(v) = 0$). This means that $\angled{\cdot,\cdot}_{\mathcal F}$ is strictly positive definite. \par 
        We will now show that the point-evaluation is continuous. For simplicity, we set $\ell_G:=\sum_{e\in E}l_e>0$. 
        Let $f\in\mathcal F$ and define
        \begin{equation*}
            A_j := \norm{f^{(j)}}_{L^2} = 
            \round{
                \integ{G}{}{
                    \abs{f^{(j)}(t)}^2}{t}
            }^{\frac{1}{2}},
            \qquad j \in \curly{0, ..., p},
        \end{equation*}
        where $f_e^{(0)}=f_e$. In particular,
        $A_{p}=\norm{f}_{\mathcal F}$. Now, for $x_1,x_2 \in G$, let $p\subseteq G$ be a shortest-path from $x_1$ to $x_2$. Then, we have:
        \begin{align*}
            \abs{f(x_1) - f(x_2)} 
            &= \abs{\integ{p}{}{f'(t)}{t}} 
            \leq \integ{p}{}{\abs{f'(t)}}{t} \\
            &\overset{C.S.}{\leq} \round{\integ{p}{}{1}{t}}^\frac{1}{2}\, \round{\integ{p}{}{\abs{f'(t)}^2}{t}}^\frac{1}{2} = \sqrt{d_{SP}(x_1,x_2)}\, A_1,
        \end{align*}
        and, for all $x\in G$,
        \begin{equation}
            \label{eq:proof_RKHS_evaluation_functional_is_continuous_f_bound}
            \abs{f(x)} = \abs{\frac{1}{n}\sum_{v\in V}\round{\vphantom{\big|} f(x) - f(v)}}
            \leq \frac{1}{n} \sum_{v\in V} \sqrt{d_{SP}(x,v)} \, A_1 \leq \sqrt{\ell_G} \, A_1,
        \end{equation}
        which implies that $\abs{f(x)}^2 \leq \ell_G A_1^2$, and therefore
        \begin{equation*}
             \integ{G}{}{\abs{f(t)}^2}{t} \leq \ell_G^2 A_1^2
             \implies 
             A_0 \leq \ell_G A_1.
        \end{equation*}
        Now, for every $j\in\curly{1,...,m}$, using integration by parts and recalling that only the integral part remains (Lemma \ref{lem:int_by_parts_cancels_non-int_term}), we have
        \begin{align*}
            A_j^2 &= \abs{\integ{G}{}{f^{(j-1)}(t) \, f^{(j+1)}(t)}{t}} 
            \leq \integ{G}{}{\abs{f^{(j-1)}(t) \, f^{(j+1)}(t)}}{t} \\
            &\overset{C.S.}{\leq}\round{\integ{G}{}{\abs{f^{(j-1)}(t)}^2}{t}}^\frac{1}{2}\round{\integ{G}{}{\abs{f^{(j+1)}(t)}^2}{t}}^\frac{1}{2} = A_{j-1} A_{j+1}.
        \end{align*}
        Now, by induction we show that, for all $j\in\curly{0,...,m}$, it holds $A_j\leq \ell_G A_{j+1}$. The case $j=0$ has already been shown. If $A_{j-1}\leq \ell_G A_{j}$, then $A_j^2 \leq A_{j-1} A_{j+1} \leq \ell_G A_j A_{j+1}$, thus, dividing by $A_j$ (if $A_j=0$ the inequality is trivial), we obtain $A_{j}\leq \ell_G A_{j+1}$. From this, we have
        \begin{equation}
            \label{eq:proof_RKHS_evaluation_functional_is_continuous_A_bounds}
            A_1 \leq \ell_G A_2 \leq ... \leq \ell_G^m A_{p} = \ell_G^m \norm{f}_{\mathcal F},
        \end{equation}
        which, in concert with (\ref{eq:proof_RKHS_evaluation_functional_is_continuous_f_bound}), yields $\abs{f(x)}\leq \ell_G^{m+\frac{1}{2}} \norm{f}_{\mathcal F}$. Thus, for every $x\in G$, the linear evaluation functional
        $f\mapsto f(x)$ is bounded, hence continuous. \par 
        We are left to show that $\mathcal{H}$ is complete. From \eqref{eq:proof_RKHS_evaluation_functional_is_continuous_A_bounds}, we get that, for each $j \in \curly{0, ..., p}$, it holds
        \begin{equation*}
            \norm{f^{(j)}}_{L^2} \leq \ell_G^{p-j} \norm{f^{(p)}}_{L^2}.
        \end{equation*}
        In particular, if we define the edge-wise Sobolev space
        \begin{equation*}
            W := \bigoplus_{e\in E} H^{p}(0,l_e), \qquad \norm{f}_W^2 := \sum_{j=0}^{p} \norm{f^{(j)}}_{L^2}^2,
        \end{equation*}
        then, for each $f\in\mathcal F$, it holds:
        \begin{equation*}
            \norm{f}_{\mathcal F} \leq \norm{f}_{W} \leq \round{
                \sum_{j=0}^{p} \ell_G^{2(p-j)}
            }^{\frac{1}{2}} \norm{f^{(p)}}_{L^2} = \round{
                \sum_{j=0}^{p} \ell_G^{2j}
            }^{\frac{1}{2}} \norm{f}_{\mathcal F}.
        \end{equation*}
        That is: the two norms $\norm{\cdot}_{\mathcal F}$ and $\norm{\cdot}_{W}$ are equivalent on $\mathcal F$, therefore they induce the same topology. \par 
        Now, we are going to show that $\mathcal{F}$ is a closed subspace of $W$. Assume that $\curly{f_k}_k \subset \mathcal{F}$ and that $f_k \to f$ in $W$ as $k \to +\infty$: we want to show that $f\in\mathcal F$. On each edge $e \in E$, $f_k\big|_e \to f\big|_e$ in $H^{p}(0,l_e)$. Now, on each bounded interval $(0,l_e)$, membership in $H^{p}$ provides a $C^m$ representative whose $m$\textsuperscript{th}-order derivative is absolutely continuous. Indeed, since each weak derivative (up to order $m$) belongs to $H^1(0,l_e)$, it has a representative of the form $u(x) = u(0) + \integ{0}{x}{u'(t)}{t}$. Therefore, $f$ satisfies the edge-wise regularity conditions required by $\mathcal F$. \par 
        We will now show that $f$ satisfies the Kirchhoff conditions at $V$. For $u\in H^1(0,l_e)$ and $x\in \curly{0,l_e}$, the endpoint trace estimate gives
        \begin{equation*}
            \abs{u(x)} \leq l_e^{-1/2} \norm{u}_{L^2(0,l_e)} + l_e^{1/2} \norm{u'}_{L^2(0,l_e)}.
        \end{equation*}
        For each $j\in \curly{0, ..., m}$, we apply this estimate to $u := \round{f_k\big|_e - f_k\big|_e}^{(j)}$. By noticing that $f_k^{(j)}\big|_e \to f^{(j)}\big|_e$ in $L^2(0,l_e)$, we get that $\abs{f_k^{(j)}\big|_e(x) - f^{(j)}\big|_e(x)} \to 0$. Now, since all the Kirchhoff conditions that define $C^m(G)$ are finite and homogeneous linear relations among these endpoint traces, they pass to the limit. In particular, the $0$-order continuity conditions make $f$ a well-defined continuous function on $G$, and
        \begin{equation*}
            \sum_{v\in V} f(v) = \lim_{k\to\infty} \sum_{v\in V} f_k(v) = 0.
        \end{equation*}
        This shows that $f \in \mathcal{F}$, that is: $\mathcal{F}$ is closed in $W$. To conclude, notice that each $H^{p}(0,l_e)$ is a Hilbert space. Therefore, $W = \bigoplus_e H^{p}(0,l_e)$ is complete and so is $\mathcal{F}$ equipped with the (restricted) norm $\norm{\cdot}_W$. Since the two norms are equivalent on $\mathcal{F}$, $\mathcal{F}$ is complete also with respect to $\norm{\cdot}_{\mathcal{F}}$. Thus, any Cauchy sequence in $\mathcal{F}$ is Cauchy in $W$, admits a limit in $W$ and it belongs to $\mathcal{F}$, since it is closed. This concludes the proof.
    \end{proof}

    Proposition \ref{prop:RKHS_evaluation_functional_continuous} justifies the following definition.
    \begin{definition}
        \label{def:process_Z}
        Let $m\in\N$, $p:=m+1$ and let $Z = Z_m = Z:G\to\R$ the (unique) zero-mean Gaussian process whose RKHS is $\mathcal{H}$, let $\gamma_p:G\times G\to\R$ denote its variogram and let $d_p:G\times G\to\R$ defined as $d_p:=\sqrt{\gamma_p}$. The distance $d_p$ will be called ``polyharmonic resistance distance of order $p$''. 
    \end{definition}
    
    \begin{propositionrep}
        \label{prop:Z_is_C^m}
        The process $Z_m$ belongs to $C^m(G)$ in the sample-path sense.
    \end{propositionrep}
    \begin{proof}[Proof of Proposition \ref{prop:Z_is_C^m}]
        Comes immediately from the additive construction \eqref{eq:process_Z_additive_definition}, the fact that $Z_\mu$ and $Z_E$ are (independent) sample-path $m$-times differentiable and the equivalence of the two definitions given in Proposition \ref{prop:RKHS_of_Z_is_H}.
    \end{proof}

    \begin{lemmarep}
        \label{lem:separation_property_gamma}
        The variogram $\gamma_p$ separates the space $G$, that is: if $\gamma_p(x_1,x_2) = 0$, then necessarily $x_1 = x_2$.
    \end{lemmarep}
    \begin{proof}[Proof of Lemma \ref{lem:separation_property_gamma}]
        It is sufficient to prove that if $\forall f \in \mathcal{F}$, $f(x_1) = f(x_2)$, then necessarily $x_1 = x_2$. Indeed
        \begin{equation*}
            \gamma_p(x_1,x_2) = 0 \iff \norm{C_Z(x_1, \cdot) - C_Z(x_2, \cdot)}_{\mathcal F}=0,    
        \end{equation*}
        that is $C_Z(x_1, \cdot) = C_Z(x_2, \cdot)$. This, by the reproducing property, implies that $f(x_1) = f(x_2)$ for all $f \in \mathcal{F}$.\par 
        We will show the contrapositive. Assume that $x_1\ne x_2$ and $x_1\not\in V$. Then, there exist an $\varepsilon > 0$ such that $B_\varepsilon(x_1) \cap \round{V \cup \curly{x_2}} = \emptyset$. Consider the ``mollifier-like'' function $f$ on $G$ with radius $\varepsilon$: 
        \begin{equation*}
            f(x) := \begin{cases}
                \exp\round{-\frac{1}{\varepsilon^2-d_{SP}^2(x_1, x)}} \qquad &\text{if $x\in B_\varepsilon(x_1)$},\\
                0 \qquad &\text{if $x\not\in B_\varepsilon(x_1)$}.
            \end{cases}
        \end{equation*}
        Certainly $f(x_1) = e^{-1/\varepsilon^2} > 0 = f(x_2)$ and all the Kirchhoff conditions are satisfied at $V$ (since $f$ is null) and at $\partial B_\varepsilon(x_1)$ (since the mollifier is smooth). Finally, the linear constraint on $V$ is trivially satisfied since $f(V) = \dVec 0$. \par 
        Assume now that $x_1,x_2 \in V$. In this case as well there exist an $\varepsilon > 0$ such that $(B_\varepsilon(x_1) \cup B_\varepsilon(x_2)) \cap V = \curly{x_1,x_2}$ and $B_\varepsilon(x_1) \cap B_\varepsilon(x_2) = \emptyset$. We define $f$ as 
        \begin{equation*}
            f(x) := \begin{cases}
                \exp\round{-\frac{1}{\varepsilon^2-d_{SP}^2(x_1, x)}} \qquad &\text{if $x\in B_\varepsilon(x_1)$},\\
                -\exp\round{-\frac{1}{\varepsilon^2-d_{SP}^2(x_2, x)}} \qquad &\text{if $x\in B_\varepsilon(x_2)$},\\
                0 \qquad &\text{otherwise}.
            \end{cases}
        \end{equation*}
        Clearly, $f(x_1) = -f(x_2) \ne 0$, thus the linear constraint on $V$ is satisfied. In addition, all the Kirchhoff conditions of any order are satisfied at both $x_1,x_2$ (recall that the mollifiers are symmetric around $x_1,x_2$, therefore all the odd derivatives are null). This concludes the proof.
    \end{proof}
    From Theorem \ref{theo:schoenberg}, Theorem \ref{theo:C_1^m_implies_Schoenberg_C_1^m} and Lemma \ref{lem:separation_property_gamma} we immediately get the following result.
    \begin{theorem}
        \label{theo:gamma_induces_metric_and_Schoenberg}
        For each metric graph $G$ and for each $p\in\N^+$, $d_p = \sqrt{\gamma_p}$ is a distance on $G$. In addition, for each non-constant completely monotonic function $\psi:[0,+\infty)\to\R$, the function $C_Y(x_1,x_2):= \psi(d_p^2(x_1,x_2))$ is strictly positive definite on $G\times G$. Finally, if $p \geq 2$ and $\abs{\psi^{(p-1)}(0^+)} < +\infty$, the process $Y$ having covariance $C_Y$ belongs to $C_1^{p-1}(G)$. 
    \end{theorem}

    The Definition of $Z_m$ depends on an explicit linear constraint on the vertices of $G$. However, the variogram $\gamma_p$ and, as a consequence, the distance $d_p=\sqrt{\gamma_p}$ are intrinsic properties of the graph topology. This is formally stated in the next invariance result.
    \begin{propositionrep}[Invariance]
        \label{prop:invariance}
        The variogram $\gamma_p$ is invariant under reordering of the graph vertices $V$, under reorientation of the edges $E$ and under splitting or merging edges via vertices of degree 2. In particular, $\gamma_p$ and $d_p$ are intrinsic properties of the graph topology. 
    \end{propositionrep}
    \begin{proof}[Proof of Proposition \ref{prop:invariance}]
        The reordering and reorientation properties come almost immediately from Definition \ref{def:process_Z}. Indeed, it is sufficient to show that the RKHS of $Z$ is invariant under reordering of vertices and reorientation of edges. Since such operations do not alter the topology of the graph, the space of functions $\mathcal{F}$ will be the same, therefore it is left to show that the dot product is the same as well. (\ref{eq:dot_product_RKHS}) is invariant since, for a given edge $e$, if $f,g$ are reparametrised as 
        $\tilde{f}_e(x) = f_e(l_e - x)$ and $\tilde{g}_e(x) = g_e(l_e - x)$, then
        \begin{align*}
            &\integ{0}{l_e}{\tilde{f}_e^{(p)}(x) \ \tilde{g}_e^{(p)}(x)}{x} = \integ{0}{l_e}{\frac{\text{d}^{p}}{\text{d}x^{p}} f_e(l_e - x) \ \frac{\text{d}^{p}}{\text{d}x^{p}} g_e(l_e - x)}{x}\\
            &=\integ{0}{l_e}{(-1)^{p} f_e^{(p)}(l_e - x) \  (-1)^{p} g_e^{(p)}(l_e - x)}{x}\\
            &= \integ{0}{l_e}{ f_e^{(p)}(l_e - x) \   g_e^{(p)}(l_e - x)}{x}\\
            & = \integ{l_e}{0}{(-1) f_e^{(p)}(y) \   g_e^{(p)}(y)}{y} = \integ{0}{l_e}{ f_e^{(p)}(y) \   g_e^{(p)}(y)}{y},
        \end{align*}
        where we changed the variable $y:= l_e - x$. \par

        Regarding the degree-2 splitting, consider $G'$ obtained by splitting an edge $e$ of $G$ at an internal point $s$. If we show that the variogram is invariant in this case, then it will be invariant by an arbitrary composition of splitting operations. That is: $G'=(V',E')$, where $V' = V\cup\curly{s}$, $E'=E\setminus\curly{e} \cup\curly{e_1,e_2}$, and $l_{e_1} + l_{e_2} = l_e$. We set $n:=\abs{V}$, so that $n+1=\abs{V'}$. Notice that the spaces of unnormalised functions $\mathcal{F}_{G}^\star$ and $\mathcal{F}_{G'}^\star$ (defined in Definition \ref{def:our_RKHS}, but without the constraint $\sum_v f(v) = 0$) coincide. Indeed, at $s$, the degree-$m$ vertex conditions reduce to matching derivatives up to order $m$. This also ensures that the two absolutely continuous $m$-order derivatives concatenate to an absolutely continuous function on $e$. In addition, from the additivity of the integral (\emph{mutatis mutandis} from \citet[Proof of Proposition 3]{AnderesEtAl_IsotropicCovarianceFunctions_2020}), we obtain that the dot products coincide: $\angled{f,g}_{\mathcal{F^\star}_G}$ = $\angled{f,g}_{\mathcal{F^\star}_{G'}}$. \par 
        However, the constrained spaces $\mathcal{F}_G$ and $\mathcal{F}_{G'}$ are different. Indeed, $V\ne V'$, thus the constraint actually changes the functions. To link them, consider the bijective map $T:\mathcal{F}_G \to \mathcal{F}_{G'}$
        \begin{equation*}
            (Tf)(x) := f(x) - \frac{f(s)}{n+1}, \qquad (T^{-1}g)(x) = g(x) - \frac{1}{n}\sum_{v\in V} g(v).
        \end{equation*}
        Clearly, $T$ is well-defined since 
        \begin{equation*}
            \sum_{v\in V'} (Tf)(v) = \sum_{v\in V} f(v) + f(s) - (n+1)\frac{f(s)}{n+1} = 0.
        \end{equation*}
        Now, as $T$ only shift the functions by a constant, it preserves both the dot product and pointwise differences:
        \begin{equation*}
            \angled{Tf, Tg}_{\mathcal{F}_{G'}} = \angled{f, g}_{\mathcal{F}_{G}}, \qquad (Tf)(x) - (Tf)(y) = f(x) - f(y).
        \end{equation*}
        Therefore, $T$ is a bijective isometry that preserves increments. We now conclude the proof by means of Proposition \ref{prop:variogram_exprs}:
        \begin{align*}
            \gamma_{p, G'}(x_1,x_2) &= \max\curly{\round{g(x_1) - g(x_2)}^2 \ : \ g\in \mathcal{F}_{G'}, \ \norm{g}_{\mathcal{F}_{G'}}=1}\\
            &=\max\curly{\round{(Tf)(x_1) - (Tf)(x_2)}^2 \ : \ f\in \mathcal{F}_{G}, \ \norm{f}_{\mathcal{F}_{G}}=1}\\
            &= \max\curly{\round{f(x_1) - f(x_2)}^2 \ : \ f\in \mathcal{F}_{G}, \ \norm{f}_{\mathcal{F}_{G}}=1}\\
            &=\gamma_{p,G}(x_1,x_2).
        \end{align*}
    \end{proof}

    The next result provides the connection between our generalisation, the original construction in \cite{AnderesEtAl_IsotropicCovarianceFunctions_2020} and the biharmonic distance.
    \begin{propositionrep}
        \label{prop:m=0_resistance_Anderes}
        On any metric graph $G$, $\gamma_1:G\times G \to \R$ is the resistance distance on $G$ as defined in \citet{AnderesEtAl_IsotropicCovarianceFunctions_2020}. Furthermore, $d_2$ is the natural choice for the biharmonic distance on metric graphs.
    \end{propositionrep}
    \begin{proof}[Proof of Proposition \ref{prop:m=0_resistance_Anderes}]
        First, we prove this result for the vertices: let $v_1,v_2 \in V$. Since the process $Z_E$ as in Definition \ref{def:process_Z_E} and the process $Z_e$ as defined in \citet[Equation (14)]{AnderesEtAl_IsotropicCovarianceFunctions_2020} are both null on $V$, we just need to show that the variograms of the processes $Z_\mu$ agree between this article construction and the one in \citeauthor{AnderesEtAl_IsotropicCovarianceFunctions_2020}. First, let us recall from Remark \ref{prop:laplacian_consistency_m=0}, that the Laplacian matrix $\matr L$ we defined coincides with the standard Laplacian (which is the one used in \citeauthor{AnderesEtAl_IsotropicCovarianceFunctions_2020}). From their Equation (13), or from standard literature, \eg \citet[Equation 6]{GhoshEtAl_MinimizingEffectiveResistance_2008}, it is immediate to obtain the following expression for $d_R$ on $V$:
        \begin{equation*}
            d_R(v_1,v_2) = (\dVec \phi_1 - \dVec \phi_2)^\top \matr L_A^{-1} (\dVec \phi_1 - \dVec \phi_2),
        \end{equation*}
        where $\matr L_A$ denotes the modified Laplacian matrix defined in \citet[Equation (8)]{AnderesEtAl_IsotropicCovarianceFunctions_2020}, and where $\dVec \phi_i$ represents the element of the canonical basis of $\R^n$ corresponding to the vertex $v_i$. Now, \citet[Lemma 1 of Supplementary Material]{FilosiEtAl_TemporallyEvolvingGeneralisedNetworks_2024} shows that the variogram of the process $Z_\mu$ is invariant under the rank-1 update $\matr L + \dVec x \dVec x^\top$ as long as $\onen^\top \dVec x \ne 0$. Therefore, the variogram of \citeauthor{AnderesEtAl_IsotropicCovarianceFunctions_2020} (they use $\dVec x := \dVec \phi_i$ for a given $i\in\curly{1,\dots,n}$) is the same as the one built with the singular covariance matrix $\matr L^-$. This  shows that $\gamma_1 = d_R$ on the vertices. \par 
        To show that this holds everywhere on $G$, it is sufficient to invoke the edge-splitting invariance (see Proposition \ref{prop:invariance}). Indeed, if $x_1,x_2 \in G \setminus V$, then the variogram $\gamma_1(x_1,x_2)$ can be computed by splitting the edges containing $x_1,x_2$ at $x_1,x_2$. In this way, in the new graph, $x_1,x_2 \in V$ and by the above argument their distance is the same under the two constructions. This conclude the proof.
    \end{proof}
    Whilst the resistance distance has already been defined in \cite{AnderesEtAl_IsotropicCovarianceFunctions_2020}, no one has, to our knowledge, defined the biharmonic distance on a \emph{metric} graph. However, the spectral definitions in \cite{LipmanEtAl_BiharmonicDistance_2010} and \cite{BlackEtAl_BiharmonicDistanceGraphs_2025} match the spectral representation \eqref{eq:spectral_form_gamma} for $p=2$, providing a sound basis to the second statement of Proposition \ref{prop:m=0_resistance_Anderes}. 
    
    \subsection{Variational and spectral representations}
    
    Next, we state a variational formulation of the variogram $\gamma_p$ and an expression for the covariance of $Z$. 
    \begin{propositionrep}[Variational representations]
        \label{prop:variogram_exprs}
        The following variational representations for the variogram $\gamma_p$ hold (notice the analogy with the classical resistance distance \cite[see, for instance,][Theorem 2.3]{JorgensenPearse_HilbertSpaceApproach_2010}): 
        \begin{align}
            \label{eq:max_representation_variogram}
            \gamma_p(x_1,x_2) &= \max \curly{\round{f(x_1) - f(x_2)}^2 \ : \ f\in \mathcal F, \ \norm{f}_{\mathcal{F}} = 1},\\
            \label{eq:min_representation_variogram}
            & = \round{\min\curly{\norm{f}_{\mathcal F}^2 \ : f\in\mathcal F, \ \abs{f(x_1) - f(x_2)} = 1}}^{-1},
        \end{align}
        where the latter expression holds for $x_1\ne x_2$. In addition, the covariance of $Z_m$ has the following expression:
        \begin{equation}
            \label{eq:C_Z_variogram_representation}
            C_{Z_m}(x_1,x_2) = c + \frac{1}{2n} \sum_{v\in V}\round{\vphantom{\Big|}\gamma_p(v, x_1) + \gamma_p(v,x_2) - \gamma_p(x_1,x_2)},
        \end{equation}
        where $c=-\frac{1}{2n^2}\sum_{v,w\in V} \gamma_p(v,w)$ does not depend on $x_1,x_2$.
    \end{propositionrep}
    \begin{proof}[Proof of Proposition \ref{prop:variogram_exprs}]
        This proof is very close to the proof of \citet[Proposition 5]{AnderesEtAl_IsotropicCovarianceFunctions_2020}. To show, \ref{eq:max_representation_variogram}, we consider $x_1,x_2 \in G$ and $f \in \mathcal{F}$ such that $\norm{f}_\mathcal{F} \leq 1$ and use the Cauchy-Schwarz inequality.
        \begin{align*}
            \round{f(x_1) - f(x_2)}^2 &= \angled{f, C_Z(x_1,\cdot) - C_Z(x_2,\cdot)}_\mathcal{F}^2\\
            &\leq 1 \cdot \norm{C_Z(x_1,\cdot) - C_Z(x_2,\cdot)}_\mathcal{F}^2 \\
            &= \angled{C_Z(x_1,\cdot) - C_Z(x_2,\cdot), C_Z(x_1,\cdot) - C_Z(x_2,\cdot)}_\mathcal{F}\\
            &=C_Z(x_1,x_1) + C_Z(x_2,x_2) - 2 C_Z(x_1,x_2) \\
            &= \gamma_p(x_1,x_2)
        \end{align*}
        In addition, for $x_1\ne x_2$, the function
        \begin{equation*}
            f_0(x) := \frac{C_Z(x, x_1) - C_Z(x, x_2)}{\norm{C_Z(\cdot, x_1) - C_Z(\cdot, x_2)}_\mathcal{F}}
        \end{equation*}
        has clearly norm $1$ and satisfies
        \begin{align*}
            \round{f_0(x_1) - f_0(x_2)}^2 &= \frac{\round{C_Z(x_1, x_1) - C_Z(x_1, x_2) - C_Z(x_2,x_1) + C_Z(x_2,x_2)}^2}{\norm{C_Z(\cdot, x_1) - C_Z(\cdot, x_2)}_\mathcal{F}^2} \\
            &= \frac{\gamma_p^2(x_1,x_2)}{\angled{C_Z(\cdot, x_1) - C_Z(\cdot, x_2), C_Z(\cdot, x_1) - C_Z(\cdot, x_2)}_\mathcal{F}}\\
            &= \frac{\gamma_p^2(x_1,x_2)}{\gamma_p(x_1,x_2)} = \gamma_p(x_1,x_2).
        \end{align*}
        This shows (\ref{eq:max_representation_variogram}). To show (\ref{eq:min_representation_variogram}), we start from the just-shown result. This writes:
        \begin{align*}
            \gamma_p(x_1,x_2)&=\max \curly{\round{f(x_1) - f(x_2)}^2 \ : \ f\in \mathcal F, \ \norm{f}_{\mathcal{F}} = 1}\\
            &=\max\curly{\frac{(f(x_1) - f(x_2))^2}{\norm{f}_{\mathcal{F}}^2} \ : \ f\in \mathcal F, \ \norm{f}_{\mathcal{F}} \ne 0}\\
            &= 1 \bigg/ \min \curly{\frac{\norm{f}_{\mathcal{F}}^2}{(f(x_1) - f(x_2))^2} \ : \ f\in \mathcal F, \ \norm{f}_{\mathcal{F}} \ne 0} \\
            &= 1 \bigg / \min \curly{\norm{f}_{\mathcal F}^2 \ : \ f\in\mathcal F, \ (f(x_1) - f(x_2))^2 = 1}\\
            &= 1 \bigg / \min \curly{\norm{f}_{\mathcal F}^2 \ : \ f\in\mathcal F, \ \abs{f(x_1) - f(x_2)} = 1}.
        \end{align*}

        Now we show (\ref{eq:C_Z_variogram_representation}). By expanding the variogram and recalling that $C_Z(x,\cdot)\in \mathcal F$, we obtain
        \begin{equation*}
            \sum_{v \in V} \gamma_p(v,x) = \sum_{v\in V} C_Z(v,v) + n C_Z(x,x) - 2 \overbrace{\sum_{v\in V} C_Z(v,x)}^{=0},
        \end{equation*}
        from which $C_Z(x,x) = \frac{1}{n} \sum_{v\in V} \round{\gamma_p(v,x) - C_Z(v,v)}$. Using this at $x\in\curly{x_1,x_2}$, we have:
        \begin{equation}
            \label{eq:proof_sum_gamma_vertices_function_C}
            \gamma_p(x_1,x_2) = \frac{1}{n} \sum_{v\in V} \round{\vphantom{\Big|}\gamma_p(v,x_1) + \gamma_p(v,x_2) - 2C_Z(v,v)} - 2C_Z(x_1,x_2).
        \end{equation}
        Finally, consider
        \begin{align*}
            \sum_{v_1,v_2\in V} \gamma_p(v_1,v_2) &= \sum_{v_1,v_2} \round{C_Z(v_1,v_1) + C_Z(v_2,v_2)} - 2\sum_{v_1} \overbrace{\sum_{v_2} C_Z(v_1,v_2)}^{=0}\\
            &=2n\sum_{v\in V} C_Z(v,v).
        \end{align*}
        By substituting the expression for $\sum_{v\in V} C_Z(v,v)$ back in (\ref{eq:proof_sum_gamma_vertices_function_C}) and solving in $C_Z(x_1,x_2)$, we obtain (\ref{eq:C_Z_variogram_representation}).
    \end{proof}
    
    Notice that Equation (\ref{eq:max_representation_variogram}), in concert with the norm induced by $\angled{\cdot,\cdot}_{\mathcal F}$, namely
    \begin{equation*}
        \norm{f}_{\mathcal F}^2 =  \integ{G}{}{\round{f^{(p)}(x)}^2}{x},
    \end{equation*}
    provides the following interpretation to our distance. Whilst $\gamma_1$ (proven to be the effective resistance distance) measures how easy is to produce a increment $\abs{f(x_1)-f(x_2)}$ between two points for a given slope energy $\norm{f}_{\mathcal F}$, for $p\geq 2$, $\gamma_p$ measures how easy is to produce the same increment, yet for a given higher-order ``smooth'' energy. \par
    
    \begin{remark}
        \label{rem:Kirchhoff-Laplace_operator}
        In the next Proposition, we are making use of the Kirchhoff-Laplace operator $L$ on $G$ (we recall that $G$ is a finite, connected metric graph), defined as $Lf_e := -f_e''$ on each edge $e$ and with the Kirchhoff conditions of order $0,1$ at the vertices: $f$ must be continuous and the sum of exiting derivatives must be null. The natural domain of $L$ is the Sobolev space $H^{2}$ on $G$, namely the set of function that edge-wise belong to $H^2(e)$, with the Kirchhoff conditions above. We stress that the space $\mathcal F$ with $m=1$ coincides with the space $H^2$ constrained to have null sum over the vertices, that is:
        \begin{equation*}
            \mathcal{F}_{m=1} = \curly{f \in H^2(G) \ : \ \sum_{v\in V} f(v) = 0}.
        \end{equation*}
        We remark that $L$ is an unbounded, non-negative, self-adjoint operator with compact resolvent $(I+L)^{-1}$ \citep[Chapter 3]{BerkolaikoKuchment_IntroductionQuantumGraphs_2012}. As a consequence there exist a real orthonormal basis $\curly{\phi_k}_{k=0}^{\infty}$ of $L^2(G)$ such that 
        \begin{equation*}
            L \phi_k = \lambda_k \phi_k, \qquad 0=\lambda_0<\lambda_1 \leq \lambda_2 \leq ..., \qquad \lambda_k \to +\infty,
        \end{equation*}
        where eigenvalues are repeated according to their multiplicity.
    \end{remark}
    
    \begin{propositionrep}[Spectral representations]
        \label{prop:spectral_form_gamma}
        Let $L$ the Kirchhoff-Laplace operator on $G$ as defined in Remark \ref{rem:Kirchhoff-Laplace_operator}. Then, it holds
        \begin{equation}
            \label{eq:spectral_form_gamma}
            \gamma_p(x_1,x_2) = \sum_{k=1}^{\infty} \frac{\round{\phi_k(x_1) - \phi_k(x_2)}^2}{\lambda_k^{p}},
        \end{equation}
        where the sum is taken over all the $k\geq 1$ to exclude the constant eigenfunction $\phi_0$ (with null eigenvalue). In addition, by setting $\ol \phi_{k,V}:= \frac{1}{n} \sum_{v \in V} \phi_k(v)$, $C_Z$ admits the following expression:
        \begin{equation}
            \label{eq:spectral_form_C_Z}
            C_Z(x_1,x_2) = \sum_{k=1}^{\infty} \frac{
                \round{\phi_k(x_1) - \ol \phi_{k,V}}
                \round{\phi_k(x_2) - \ol \phi_{k,V}}
            }{
                \lambda_k^{p}
            }.
        \end{equation}
        
        Finally, $Z$ admits the following spectral representation:
        \begin{equation}
            \label{eq:spectral_form_Z}
            Z(x) = \sum_{k=1}^\infty \xi_k \round{\phi_k(x) - \ol \phi_{k,V}}, \qquad \xi_k \overset{\text{ind.}}{\sim} \Norm{0}{\frac{1}{\lambda_k^{p}}},
        \end{equation}
        where the convergence holds in the mean-square sense for every $x\in G$. All the expressions do not depend on the choice of the orthonormal basis within each eigenspace.
    \end{propositionrep}
    \begin{proof}[Proof of Proposition \ref{prop:spectral_form_gamma}]
        We have that $\curly{(\lambda_k, \phi_k)}_k$ forms an orthonormal basis of $L^2$ with respect to its standard dot product $\angled{f,g}_{L^2} = \int_{G}{fg}$. Now, for $k \geq 1$, we define 
        \begin{equation*}
            \psi_k(x) := \phi_k(x) - \ol \phi_{k,V}, \qquad \ol \phi_{k,V}:= \frac{1}{n} \sum_{v \in V} \phi_k(v), 
        \end{equation*}
        and set $\mathcal{F}_N := \vspan\curly{\psi_1, ..., \psi_N}\subset\mathcal F$. These eigenfunctions inherit the Kirchhoff conditions from the $\phi_k$'s, and satisfy them at every order. Indeed, on each edge, $\phi_k'' = -\lambda_k \phi_k$, therefore the higher derivatives are multiples of $\phi_k$ and $\phi_k'$. \par 
        Let $f_N\in \mathcal F_N$, then 
        \begin{equation*}
            f_N(x) = \sum_{k=1}^N a_k \psi_k(x) = c_N + \sum_{k=1}^N a_k \phi_k(x).
        \end{equation*} 
        Now, Lemma \ref{lem:our_dot_to_L2} gives:
        \begin{equation*}
            \norm{f_N}_{\mathcal{F}}^2 = \angled{f_N,f_N}_{\mathcal{F}} = \angled{L^{p}f_N,f_N}_{L^2} = \sum_{k=1}^N \lambda_k^{p} a_k^2,
        \end{equation*}
        as the constant term $c_N$ has no contribution since $L c_N = 0$. \par
        We now want to show that $\bigcup_{N=1}^{\infty} \mathcal{F}_N$ is dense in $\mathcal F$. Assume that $f\in \mathcal F$ is orthogonal to every $\psi_k$, then, using again Lemma \ref{lem:our_dot_to_L2}, we have:
        \begin{align*}
            0 = \angled{f, \psi_k}_{\mathcal F} = \angled{f, L^{p}\phi_k}_{L^2} = \lambda^{p} \angled{f, \phi_k}_{L^2},
        \end{align*}
        \ie $\angled{f, \phi_k}_{L^2} = 0$ for every $k\geq 1$. This, by completeness of the basis $\curly{\phi_k\ : \ k\in \N}$ in $L^2$, implies that $f = a_0 \phi_0$, namely that $f$ is constant on $G$. However, the vertex-sum constraint implies that $f=0$. This shows that $\bigcup_{N=1}^{\infty} \mathcal{F}_N$ is dense in $\mathcal F$. \par 
        Consider now $x_1,x_2\in G$ and define $d_k := \phi_k(x_1) - \phi_k(x_2)$. As constants cancel in differences, if $f_N = \sum_{k=1}^N a_k \psi_k$, then
        \begin{align*}
            \round{f_N(x_1) - f_N(x_2)}^2 &= \round{\sum_{k=1}^N a_k d_k}^2 = \round{\sum_{k=1}^N \round{a_k \lambda_k^{\frac{p}{2}}} \frac{d_k}{\lambda_k^{\frac{p}{2}}}}^2\\
            &\overset{C.S.}{\leq} \round{
                \sum_{k=1}^N a_k^2 \lambda_k^{p}
            } \round{
                \sum_{k=1}^N \frac{d_k^2}{ \lambda_k^{p}}
            } = \norm{f_N}_{\mathcal F}^2 \sum_{k=1}^N \frac{d_k^2}{ \lambda_k^{p}},
        \end{align*}
        and equality is attained by choosing $a_k := d_k / \lambda_k^{p}$. This means that 
        \begin{equation*}
            \sup \curly{
                \frac{
                    \round{f(x_1) - f(x_2)}^2
                }{
                    \norm{f}_{\mathcal F}^2
                } \ : \ f \in \mathcal F_N, \ f\ne 0
            } = \sum_{k=1}^N \frac{d_k^2}{ \lambda_k^{p}}.
        \end{equation*}
        Now, by the density of $\bigcup_N \mathcal F_N$ in $\mathcal F$ and using Proposition \ref{prop:variogram_exprs}, we conclude:
        \begin{align*}
            \gamma_p(x_1,x_2) &= \sup_{N\geq 1} \sup \curly{
                \frac{
                    \round{f(x_1) - f(x_2)}^2
                }{
                    \norm{f}_{\mathcal F}^2
                } \ : \ f \in \mathcal F_N, \ f\ne 0
            }\\
            &= \sup_{N\geq 1} \sum_{k=1}^N \frac{d_k^2}{ \lambda_k^{p}} = \sum_{k=1}^{\infty} \frac{(\phi(x_1) - \phi(x_2))^2}{\lambda_k^{p}},
        \end{align*}
        that is (\ref{eq:spectral_form_gamma}).\par
        To show (\ref{eq:spectral_form_C_Z}), we are going to use (\ref{eq:spectral_form_gamma}) in concert with (\ref{eq:C_Z_variogram_representation}). Let $\psi_k := \phi_k - \ol \phi_{k,V}$. Then $\sum_{v\in V} \psi_k(v) = 0$ and (\ref{eq:spectral_form_gamma}) can be rewritten using $\psi_k$ in place of $\phi_k$ (constants cancel out). Now, we apply (\ref{eq:C_Z_variogram_representation}). First, let us compute the term $c$ therein:
        \begin{align*}
            c &= -\frac{1}{2n^2} \sum_{v,w\in V} \sum_k \frac{(\psi_k(v) - \psi_k(w))^2}{\lambda_k^{p}}\\
            &= -\frac{1}{2n^2} \sum_k \frac{1}{\lambda_k^{p}} \sum_{v,w\in V} \round{\vphantom{\big|}\psi_k^2(v) + \psi_k^2(w) - 2\psi_k(v)\psi_k(w)}\\
            &=-\frac{1}{2n^2} \sum_k \frac{1}{\lambda_k^{p}} \sum_{v\in V} \round{
                n \psi_k^2(v) + \sum_{w\in V} \psi_k^2(w) - 2\psi_k(v) \overbrace{\sum_{w\in V} \psi_k(w)}^{=0}
            } \\
            &= -\frac{1}{n} \sum_k \frac{1}{\lambda_k^{p}} \sum_{v\in V} \psi_k^2(v).
        \end{align*}
        The other addend in (\ref{eq:C_Z_variogram_representation}) is
        \begin{align*}
            &\frac{1}{2n} \sum_{v\in V} \round{\vphantom{\big|} \gamma_p(x_1, v) + \gamma_p(x_2,v) - \gamma_p(x_1,x_2)}\\
            &= \frac{1}{2n} \sum_k \frac{1}{\lambda_k^{p}} \sum_v \round{
                2\psi_k^2(v) - 2\psi_k(x_1)\psi_k(v) - 2\psi_k(x_2)\psi_k(v) + 2\psi_k(x_1)\psi_k(x_2) 
            }\\
            &=\sum_k \frac{1}{\lambda_k^{p}}
            \round{
                \psi_k(x_1)\psi_k(x_2) + \frac{1}{n} \sum_v \psi_k^2(v) 
            }.
        \end{align*}
        By summing these last two expressions, we get (\ref{eq:spectral_form_C_Z}). Finally, starting from (\ref{eq:spectral_form_Z}) and computing the covariance $C_Z(x_1,x_2)$ and using the independence of the $
        \xi_k$'s, one immediately gets (\ref{eq:spectral_form_C_Z}). 
    \end{proof}
    
    \begin{remark}
        The limit for $\gamma_p$ in Proposition \ref{prop:spectral_form_gamma} also sheds light on an interesting fact: if $K$ is the multiplicity of the smallest positive eigenvalue, we have
        \begin{equation}
            \lim_{p\to +\infty} \lambda_1^{p} \gamma_p(x_1,x_2) = \sum_{k=1}^K\round{\phi_k(x_1) - \phi_k(x_2)}^2.
        \end{equation}
        In particular, if $K=1$, then the (scaled) limit $\gamma_\infty$ is
        \begin{equation*}
            \gamma_\infty(x_1,x_2) = (\phi_1(x_1) - \phi_1(x_2))^2.
        \end{equation*}
        Now, if $G$ is not a segment, since $\phi_1$ is continuous on $G$, there will be different points $x_1\ne x_2$ with $\phi_1(x_1) = \phi_1(x_2)$. This means that the (scaled) limit distance loses the separation property.
    \end{remark}
    
    Notice how the high-order distances are closely related to the eigenspace corresponding to the smallest non-zero eigenvalue of the Laplacian on $G$. We also remark that the spectrum of a quantum graph is ``usually'' simple \citep[see][Theorem 3.6]{BerkolaikoLiu_SimplicityEigenvaluesNonvanishing_2017}.
    
    \subsection{Minor properties}

    Next, we provide some insights on the variogram $\gamma_p$. Indeed, whilst the case $p=1$ is well studied and has an intuitive physical interpretation, the generalisation given here may appear opaque. However, some nice properties of the resistance distance continue to hold for a general $p$.
    
    \begin{remark}
        Let $G=(V,E)$ be a metric graph and let $G'$ be a modification of $G$ obtained by adding an edge to $G$. More precisely, we have $G'=(V,E')$, where $E' = E \cup \curly{e'}$. A notable feature of the effective resistance distance is that the resistance on the new graph $G'$ is not greater than the one on $G$. However, this is no longer true for $m \geq 1$: indeed, adding a new edge can alter the Kirchhoff conditions in the junction points, resulting in a higher, lower or equal resistance.
    \end{remark}

    \begin{propositionrep}
        \label{prop:scaling}
        Let $G$ be a metric graph and let $G':=\alpha G$, meaning that the sets of vertices $V$ and edges $E$ remain the same, but each edge in $G'$ is scaled by the same factor $\alpha > 0$. Then, it holds:
        \begin{equation}
            \label{eq:scaling}
            \gamma_{p,G'}(\alpha x_1, \alpha x_2) = \alpha^{2p-1} \ \gamma_{p,G}(x_1,x_2) \qquad \forall x_1,x_2\in G,
        \end{equation}
        where $\alpha x_i$ indicates the point on $G'$ corresponding to $x_i$ in $G$.
    \end{propositionrep}
    \begin{proof}[Proof of Proposition \ref{prop:scaling}]
        We are going to make use of (\ref{eq:max_representation_variogram}) in the two graphs $G$ and $G'$. First, let $f_G \in \mathcal F_{G}$ and consider its analogous $f_{G'}(\alpha x) = f_G(x)$. Then, we have:
        \begin{align*}
            \norm{f_{G'}}_\mathcal{F_{G'}}^2 &= \sum_{e \in E} \integ{0}{\alpha l_e}{\round{f_{G'}^{(p)}(x)}^2}{x} = \sum_{e \in E} \integ{0}{\alpha l_e}{\round{D^{p}f_{G}\round{\frac{x}{\alpha}}}^2}{x}\\
            &=\sum_{e \in E} \integ{0}{\alpha l_e}{\round{\frac{1}{\alpha^{p}}f_{G}^{(p)}\round{\frac{x}{\alpha}}}^2}{x}\\
            &= \frac{1}{\alpha^{2p}} \sum_{e \in E} \integ{0}{l_e}{\round{f_G^{p}(y)}^2}{(\alpha y)} = \frac{1}{\alpha^{2p-1}} \norm{f_G}_{\mathcal{F_G}}^2.
        \end{align*}
        As a consequence, we get:
        \begin{align*}
            \gamma_{p,G'}(\alpha x_1, \alpha x_2) &= \max\curly{\frac{(f_{G'}(\alpha x_1) - f_{G'}(\alpha x_2))^2}{\norm{f_{G'}}_{\mathcal{F}_{G'}}^2} \ : \ \norm{f_{G'}}_{\mathcal{F}_{G'}} \ne 0}\\
            &=\max\curly{\alpha^{2p-1} \ \frac{(f_{G}(x_1) - f_{G}(x_2))^2}{\norm{f_{G}}_{\mathcal{F}_{G}}^2} \ : \ \norm{f_{G}}_{\mathcal{F}_{G}} \ne 0} \\
            &= \alpha^{2p-1} \ \gamma_{p,G}(x_1,x_2).
        \end{align*}
    \end{proof}

    \begin{propositionrep}
        \label{prop:gamma_belongs_to_C^m}
        If $p \geq 2$, for any $x_0 \in G$, the function $y \mapsto \gamma_p(x_0, y)$ belongs to $C_1^{p-1}(G)$. 
    \end{propositionrep}
    \begin{proof}[Proof of Proposition \ref{prop:gamma_belongs_to_C^m}]
        Let us show that, on the interior of the edges, $\gamma_p \in C^m$. Since $Z\in C^m$, for each $x_0\in G$, the function $y\mapsto Z(y) - Z(x_0)$ is a $C^m$ map in $L^2(\Omega)$. As a consequence, its squared $L^2$-norm (that is: $\gamma_p(x_0,y)$) is a $C^m$ function.\par 
        Assume now that $y\in V$: we have
        \begin{align*}
            \sum_{e\in E_y} \frac{\partial \gamma_p(x_0, y)}{\partial_e y} &= \sum_{e\in E_y} \frac{\partial}{\partial_e y} \Exp{Z(y) - Z(x_0)}^2 \\
            &=\sum_{e\in E_y} 2\, \Exp{\vphantom{\Big|}\round{Z(y) - Z(x_0)} \partial_e Z(y)} = 0,
        \end{align*}
        since, being $Z$ continuous at $y$, $Z(y) - Z(x_0)$ does not depend on the edge $e \in E_y$. Thus, the Kirchhoff condition of order $1$ holds.
    \end{proof}
    Proposition \ref{prop:gamma_belongs_to_C^m} in particular implies that, for $p\geq 2$, $\gamma_p(x_0,y)$ is flat whenever $y$ is a leaf, and, if $y$ is a branching point, then the sum of exiting derivatives of $\gamma_p(x_0,y)$ is null. This, in concert with the fact that the distance $\gamma_p$ cannot be negative, immediately yields the following corollary.
    
    \begin{corollary}
        \label{cor:flatness_gamma_at_x=x_0_and_leaves}
        Let $x_0\in G$ and $p\geq 2$. Then $\gamma_p(x_0,x)$ is flat at $x=x_0$ and at $x=x_{\text{leaf}}$, where $x_{\text{leaf}}$ is a leaf of $G$. More precisely, for each $e\in E_{x_0}$, it holds:
        \begin{equation*}
            \frac{\partial \gamma_p(x_0,x)}{\partial_e x} \Big|_{x=x_0} = 0,
        \end{equation*}
        and, for $e \in E_{x_\text{leaf}}$: 
        \begin{equation*}
            \frac{\partial \gamma_p(x_0,x)}{\partial_e x} \Big|_{x=x_\text{leaf}} = 0.
        \end{equation*}
    \end{corollary}
    
    \begin{propositionrep}
        \label{prop:natural_distance_inequality_fails}
        Let $p\geq 1$. Then, in general, it is \textbf{not} true that, if $S\subset G$ separates $x_1$ from $x_2$, then 
        \begin{equation}
            \label{eq:natural_distance_inequality_fails}
            \gamma_p(x_1,x_2) \geq \gamma_p(x_1, S).
        \end{equation}
        As a counterexample, consider the mandarin graph with $N\geq 2$ segments (see Figure \ref{fig:natural_distance_counterexample}), where all $N$ paths $x_1\to x_2$ have length $l$ and $s_e$ is halfway on the path $x_1\to x_2$ via $e$. By letting $S:=\curly{s_e \ : \ e\in E}$, we have that, for any $N\geq 4$ and any $m\in\N$, $\gamma_p(x_1,x_2) < \min_{s\in S} \gamma_p(x_1, s)$.
    \end{propositionrep}
    \begin{figure}[ht]
        \centering
        \begin{tikzpicture}[scale = 1]
            \draw[fill = gray!20] (2,3/2) ellipse [x radius= 1/4, y radius= 2];
            
            \coordinate (v1) at (0,3/2) {};
            \coordinate (v2) at (4,3/2) {};
            \coordinate (v3) at (2,0) {};
            \coordinate (v4) at (2,1) {};
            \coordinate (v5) at (2,2) {};
            \coordinate (v6) at (2,3) {};
            
            \fill (v1) circle (2pt);
            \fill (v2) circle (2pt);
            \fill (v3) circle (2pt);
            \fill (v4) circle (2pt);
            \fill (v5) circle (2pt);
            \fill (v6) circle (2pt);
            
            \draw (v1) -- (v3);
            \draw (v1) -- (v4);
            \draw (v1) -- (v5);
            \draw (v1) -- (v6);
            \draw (v2) -- (v3);
            \draw (v2) -- (v4);
            \draw (v2) -- (v5);
            \draw (v2) -- (v6);
            
            \node[left] at (v1) {$x_1$};
            \node[right] at (v2) {$x_2$};
            \node[above] at (2,3.5) {$S$};
        \end{tikzpicture}
    
        \caption{The mandarin graph with $N=4$ segments: an example of graph for which the distance inequality (\ref{eq:natural_distance_inequality_fails}) fails.}
        \label{fig:natural_distance_counterexample}
    \end{figure}
    \begin{proof}[Proof of Proposition \ref{prop:natural_distance_inequality_fails}]
        We parametrise the $N$ paths $x_1\to x_2$ as $[0,l]$, for $e\in E:=\curly{1,...,N}$: $0,l/2,l$ correspond to $x_1,s_e,x_2$ respectively, where $s_e\in S$ is halfway of the path $x_1\to x_2$ via edge $e$. \par
        Now, on each edge $e\in E$, a generic eigenfunction of the Laplace operator $L$ will be
        \begin{equation*}
            \phi_e(x) = a_e\cos(\omega x) + b_e \sin(\omega x).
        \end{equation*}
        Adding the Kirchhoff conditions, the continuity constraint at $x_1 = 0$ yields $a_e = a, \forall e\in E$, whilst the one at $x_2 = l$ implies $\sin(\omega l) = 0$, namely $\omega = \frac{k\pi}{l}$, for $k\in \Z$. Next, the first-order Kirchhoff conditions at $x_1$ and $x_2$ implies $\sum_{e\in E} b_e = 0$. This means that every function 
        \begin{equation}
            \label{eq:proof_mandarin_generic_eigenfunction_phi}
            \phi_e(x) = a \cos(\omega x) + b_e \sin(\omega x),
        \end{equation}
        with $\sum_e b_e = 0$ and $\omega = \frac{k\pi}{l}$ is an eigenfunction. Now, fix $k\in \N^+$ and consider the eigenspace associated to $\lambda_k=\round{\frac{k\pi}{l}}^2$: by labelling $[a, b_1, ..., b_N]$, it is clear that it is isomorphic to $\R \times \R^N$ where the first component ($a$) is free and the second component ($\dVec b\in\R^N$) satisfies $\dVec 1_N^\top \dVec b = 0$. In addition, if $\phi^{(j)},\phi^{(j')}$ are eigenfunction (\ref{eq:proof_mandarin_generic_eigenfunction_phi}), then
        \begin{align}
            &\nonumber \angled{\phi^{(j)}, \phi^{(j')}}_{L^2}\\
            &\nonumber= \sum_{e\in E} \integ{0}{l}{\round{a^{(j)}\cos(\omega x) + b_e^{(j)} \sin(\omega x)}\round{a^{(j')}\cos(\omega x) + b_e^{(j')} \sin(\omega x)}}{x}\\
            \label{eq:proof_mandarin_dot_product}&=\frac{1}{2} l \round{N a^{(j)} a^{(j')} + \sum_{e\in E} b_e^{(j)}b_e^{(j')}} =: \angled{[a, \dVec b]^{(j)}, [a, \dVec b]^{(j')}}.
        \end{align}
        As a consequence, finding an orthonormal basis of the eigenspace associated to $\lambda_k$ is the same problem as finding an orthonormal basis of $\curly{[a,\dVec b] \in \R \times \R^N \ : \ \dVec 1_N^\top \dVec b = 0}$ with respect to the dot product (\ref{eq:proof_mandarin_dot_product}). It turns out that a possible choice is the following:
        \begin{equation*}
            a^{(j)} = \begin{cases}
                \sqrt{\frac{2}{lN}} \quad&\text{if $j=1$},\\
                0 \quad&\text{if $2 \leq j \leq N$},
            \end{cases} 
            \qquad\quad  
            b_e^{(j)} = \begin{cases}
                0 \quad&\text{if $j=1$},\\
                \sqrt{\frac{2}{lj(j-1)}} \ u_{e,j-1} \quad&\text{if $2 \leq j \leq N$},
            \end{cases}
        \end{equation*}
        where $u_{e,j-1}$ is defined as
        \begin{equation*}
            u_{e,j-1} := \begin{dcases}
                1 \qquad&\text{if $e \leq j-1$},\\
                1-j \qquad&\text{if $e = j$},\\
                0 \qquad&\text{$e \geq j+1$}.
            \end{dcases}
        \end{equation*}
        This provides the following basis of the eigenspace associated to $\lambda_k$:
        \begin{equation*}
            \phi_e^{(j)}(x) = \begin{cases}
                \sqrt{\frac{2}{lN}}  \ \cos\round{\frac{k\pi x}{l}} \quad&\text{if $j=1$}, \\
                \sqrt{\frac{2}{lj(j-1)}} \ u_{e,j-1} \ \sin\round{\frac{k\pi x}{l}} \quad&\text{if $2 \leq j \leq N$}.
            \end{cases}
        \end{equation*}
        We remark that the dot product between eigenfunctions of different eigenspaces is null, as the integral of sines and cosines for different frequencies vanishes. As a consequence, we have found an orthonormal basis of $L^2(G)$. We are now interested in computing $\gamma_p(x_1,x_2)$ and $\gamma_p(x_1,S) = \gamma_p(x_1,s_1)$, where the latter equality follows from the symmetry of $G$. To this purpose, we use Proposition \ref{prop:spectral_form_gamma}.
        \begin{align*}
            \gamma_p&(x_1,x_2)= \sum_{k=1}^\infty \sum_{j=1}^N \frac{\round{\phi_{k,e}^{(j)}(x_1) - \phi_{k,e}^{(j)}(x_2)}^2}{\round{\frac{k\pi}{l}}^{2p}}\\
            &= \frac{l^{2p}}{\pi^{2p}} \sum_{k=1}^\infty \frac{1}{k^{2p}} \round{\frac{2}{lN} \round{\cos(0) - \cos(k\pi)}^2 + \sum_{j=2}^N \frac{2u_{e,j-1}^2}{lj(j-1)}\round{\sin(0) - \sin(k\pi)}^2} \\
            &=\frac{8l^{2p-1}}{N\pi^{2p}} \sum_{k=0}^\infty \frac{1}{(2k+1)^{2p}}\\
            \gamma_p&(x_1,s_1) = \frac{l^{2p}}{\pi^{2p}} \sum_{k=1}^\infty \frac{1}{k^{2p}} \round{\frac{2}{lN} \round{\cos(0) - \cos\round{\frac{k\pi}{2}}}^2 + \sum_{j=2}^N \frac{2u_{e,j-1}^2}{lj(j-1)}\round{\sin(0) - \sin\round{\frac{k\pi}{2}}}^2}\\
            &=\frac{2l^{2p-1}}{N\pi^{2p}} \round{N+\frac{1}{2^{2m}}}\sum_{k=0}^\infty \frac{1}{(2k+1)^{2p}}
        \end{align*}
        As a consequence, we have 
        \begin{equation*}
            \frac{\gamma_p(x_1,S)}{\gamma_p(x_1,x_2)} = \frac{1}{4} \round{N + \frac{1}{2^{2m}}} > 1 \iff N > 4 - \frac{1}{2^{2m}},
        \end{equation*}
        which holds for any $N\geq 4$ and any $m\in \N$. \par
        A couple of final remarks: first, notice the factor $l^{2p-1}$ in both $\gamma_p(x_1,x_2)$ and $\gamma_p(x_1,s_1)$: this is coherent with Proposition \ref{prop:scaling}. Second, as $N\to +\infty$, we have $\gamma_p(x_1,x_2) \to 0$ whilst $\gamma_p(x_1,s_1)$ converges to a positive constant for any $m\in \N$.
    \end{proof}
    
    Although the above result may seem counter-intuitive, it is instead quite natural: the classical resistance distance ($p=1$) is influenced by \emph{how many} paths connect two given points on $G$, with the resistance decreasing as more paths are added or as these paths shorten. The same behaviour is present also for $p\geq 2$, providing an additional, coherent link between the base, well-known case $p=1$ and its extensions. 

\subsection{Relation with the Whittle-Mat\`ern fields}
    \label{ssec:relation_whittle_matern}
    The spectral representation \eqref{eq:spectral_form_Z} closes a gap between the two approaches for the definition of processes on metric graphs. Indeed, the process $Z=Z_m$ can be seen as a centred limit of the Whittle-Mat\`ern fields defined in \cite{BolinEtAl_GaussianWhittleMatern_2024}.
    \begin{propositionrep}[Connection with Whittle-Mat\`ern processes]
        \label{prop:bolin_connection}
        Let $\kappa > 0$ and let $U_{\kappa}$ be the solution to the stochastic differential equation
        \begin{equation}
            \label{eq:bolin_SDE}
            (\kappa^2 I + L)^{p/2} (U_{\kappa}) = \mathcal{W},
        \end{equation}
        where $I$ is the identity operator and $\mathcal W$ is the Gaussian white noise on $G$ \cite[see][Section 3]{BolinEtAl_GaussianWhittleMatern_2024}.
        In addition, let
        \begin{equation}
            U_{\kappa,V}(x) := U_{\kappa}(x) - \frac{1}{n} \sum_{v\in V} U_{\kappa}(v).
        \end{equation}
        Then, as $\kappa \to 0^+$, $U_{\kappa, V}$ converges in finite-dimensional distributions to $Z_m$, that is: for each $N\in \N^+$ and $x_1, ..., x_N \in G$, then 
        \begin{equation*}
            \begin{bmatrix}
                U_{\kappa,V}(x_1) & \dots & U_{\kappa,V}(x_N)
            \end{bmatrix}^\top \overset{\mathcal D}{\longrightarrow} \begin{bmatrix}
                Z_m(x_1) & \dots & Z_m(x_N)
            \end{bmatrix}^\top.
        \end{equation*}
    \end{propositionrep}
    \begin{proof}[Proof of Proposition \ref{prop:bolin_connection}]
        Let $\psi_k := \phi_k - \ol \phi_{k,V}$, as in the proof of Proposition~\ref{prop:spectral_form_gamma}. \citet{BolinEtAl_GaussianWhittleMatern_2024}[Proposition 8] gives the spectral representation for $U_\kappa$ (notice that here $\tau = 1$):
        \begin{equation*}
            U_\kappa(x) = \sum_{k=0}^\infty \frac{\xi_k}{\round{\kappa^2 + \lambda_k}^{p/2}} \phi_k(x),
        \end{equation*}
        where $\xi_k \simiid \Norm{0}{1}$. As a consequence, we have
        \begin{align*}
            U_{\kappa,V}(x) &= \sum_{k=0}^{\infty} \round{
                \frac{\xi_k}{\round{\kappa^2 + \lambda_k}^{p/2}} \phi_k(x) - \frac{1}{n} \sum_{v\in V} \frac{\xi_k}{\round{\kappa^2 + \lambda_k}^{p/2}} \phi_k(v)
            }\\
            &=\sum_{k=0}^\infty  \frac{\xi_k}{\round{\kappa^2 + \lambda_k}^{p/2}}\psi_k(x)\\
            &=\sum_{k=1}^\infty  \frac{\xi_k}{\round{\kappa^2 + \lambda_k}^{p/2}}\psi_k(x).
        \end{align*}
        From which, in concert with \eqref{eq:spectral_form_C_Z}, it is immediate to show that, for all $\kappa>0$,
        \begin{align*}
            C_{U_{\kappa,V}}(x_1,x_2) &= \sum_{k=1}^\infty \frac{\psi_k(x_1) \psi_k(x_2)}{\round{\kappa^2 + \lambda_k}^{p}} \leq \sum_{k=1}^\infty \frac{\abs{\psi_k(x_1) \psi_k(x_2)}}{\lambda_k^{p}}\\
            &\leq \sum_{k=1}^\infty \frac{\psi_k^2(x_1) + \psi_k^2(x_2)}{2 \lambda_k^p} \\
            &= \frac{1}{2}\round{\vphantom{\Big|}C_Z(x_1,x_1) + C_Z(x_2,x_2)}< +\infty.
        \end{align*}
        Hence, by the dominated convergence theorem, we have that $C_{U_{\kappa,V}}(x_1,x_2) \to C_Z(x_1,x_2)$ on $(x_1,x_2) \in G\times G$. Since both $Z$ and $U_{\kappa,V}$ are zero-mean Gaussian processes, the finite-dimensional characteristic function of $U_{\kappa,V}$, $\varphi_{U_{\kappa,V}}(\dVec t)$, converges for any $\dVec t \in \R^N$ to the one of $Z$ as $\kappa\to 0^+$: this proves the convergence in law.
    \end{proof}

\section{Smooth isotropic covariance functions}
\label{sec:smooth_isotropic_covariances}
    In this subsection, we explore some relevant results about the relationship between isotropic processes built via the Schoenberg theorem and their differentiability. First, we start from a general result holding on the real line, then we will delve into processes on metric graph. Throughout this section, for two real functions $f,g$ we will write $f \succeq g$ to denote that $f-g$ is positive semidefinite. In addition, we will say that a function is positive definite if it is either positive semidefinite or strictly positive definite, whilst strict positive definiteness will be explicitly stated.
    \begin{lemmarep}
        \label{lem:differentiability_under_covariance_domination}       
        Let $Z \in C^m(\R)$ and let $Y : \R \to \R$ be a stochastic process. Assume that both $Y$ and $Z$ have $0$-mean. If $C_Z \succeq C_Y$, then $Y \in C^m(\R)$ and, at $x=y=x_0$, it holds:
        \begin{equation}
            \label{eq:differentiability_under_covariance_domination}
            0 \leq \frac{\partial^{2m}}{\partial x^m \partial y^m} C_Y(x,y) \leq \frac{\partial^{2m}}{\partial x^m \partial y^m} C_Z(x,y).
        \end{equation}
    \end{lemmarep}
    \begin{proof}[Proof of Lemma \ref{lem:differentiability_under_covariance_domination}]
        The condition $C_Z \succeq C_Y$ is equivalent to
        \begin{equation}
            \label{eq:proof_differentiability_under_cov_mean_square_inequality}
            \Exp{
                \sum_{i=1}^N a_i Y(x_i) 
            }^2 \leq
            \Exp{
                \sum_{i=1}^N a_i Z(x_i) 
            }^2,
        \end{equation}
        for each $N\in\N^+$, $a_1,...,a_N \in \R$ and $x_1,...,x_N \in \R$. Assuming $m\geq 1$, we will now show that $Y$ has mean-square first derivative. Fix $x_0 \in \R$ and define the incremental ratio $(D_h Z)(x) := \frac{Z(x+h) - Z(x)}{h}$. This implies that
        \begin{equation*}
            D_h Y(x_0) - D_{h'} Y(x_0) = \frac{Y(x_0+h)}{h} - \frac{Y(x_0+h')}{h'} + \round{\frac{1}{h'} - \frac{1}{h}} Y(x_0).
        \end{equation*}
        Applying now (\ref{eq:proof_differentiability_under_cov_mean_square_inequality}) to this equation, we get:
        \begin{align*}
            \Exp{\vphantom{\Big|}D_h Y(x_0) - D_{h'}Y(x_0)}^2 \leq \Exp{\vphantom{\Big|}D_h Z(x_0) - D_{h'}Z(x_0)}^2.
        \end{align*}
        Since $Z \in C^m(\R) \subseteq C^1(\R)$, the right-hand side converges to $0$ as $h,h' \to 0$. As a consequence, $\curly{D_h Y(x_0)}_h$ is a Cauchy sequence in $L^2$: completeness gives the existence of the limit, which is $Y'(x_0)$. \par
        To show that the above argument holds for any $m$, it is sufficient to notice that each derivative of $Y$ can be expressed as a limit of a linear combination of its values (finite difference coefficients). Thus, (\ref{eq:proof_differentiability_under_cov_mean_square_inequality}) continues to apply and the right-hand side goes to $0$ since, by hypothesis, $Z\in C^m$. In addition, (\ref{eq:proof_differentiability_under_cov_mean_square_inequality}) applies on any linear combination of the derivatives of any order (up to $m$). This means that, for each $k\in\curly{1,...,m}$, $C_{Z^{(k)}} \succeq C_{Y^{(k)}}$, namely
        \begin{equation*}
            \frac{\partial^{2k}}{\partial^k x_1 \partial^k x_2} C_Z(x_1, x_2) \succeq 
            \frac{\partial^{2k}}{\partial^k x_1 \partial^k x_2} C_Y(x_1, x_2),
        \end{equation*}
        which, in particular, implies (\ref{eq:differentiability_under_covariance_domination}).
    \end{proof}

    Lemma \ref{lem:differentiability_under_covariance_domination} has a quite natural interpretation. If $Z = Y + X$, where $X,Y$ are independent stochastic processes, then the regularity of $Z$ translates on both $Y$ and $X$. In other words, if we sum a $Y \notin C^m(\R)$ to an arbitrary, yet independent, process, we cannot attain the desired differentiability of the sum.
    
    \begin{theoremrep}
        \label{theo:schoenberg_keeps_diff_R^n}
        Let $Z:\R \to \R$ be a Gaussian zero-mean square-integrable process such that $Z(x_1)=Z(x_2)$ a.s. implies $x_1=x_2$. Let $\psi:[0,+\infty)\to\R$ be a completely monotonic non-constant function. Let $Y := \Sch_{\psi}(Z)$. Then, the following are equivalent:
        \begin{enumerate}
            \item $Y \in C^m(\R)$,
            \item $Z \in C^m(\R)$ and all derivatives of $\psi$ up to order $m$ are finite at $0^+$.
        \end{enumerate}
    \end{theoremrep}
    \begin{proof}[Proof of Theorem \ref{theo:schoenberg_keeps_diff_R^n}]
        First, we consider the case $m=0$. Let us show (1)$\implies$(2). Assume that $Y$ is mean-square continuous at $x_0$ and let $x\in \R$. Then we have
        \begin{equation}
            \label{eq:proof_shoenberg_keeps_diff_m=0_relation_variogram}
            \Exp{Y(x_0) - Y(x)}^2 = 2 \round{\psi(0) - \psi(\gamma_Z(x, x_0))}.
        \end{equation}
        As the left-hand side goes to $0$ as $x\to x_0$, so will the right-hand side, that is $\psi(\gamma_Z(x,x_0)) \to \psi(0)$. Since $\psi$ is completely monotonic and non-constant, it is also strictly decreasing. This means that $\gamma_Z(x,x_0) \to 0$, namely that $Z$ is mean-square continuous at $x_0$. To show (2)$\implies$(1), it is sufficient to notice that, if $Z$ is mean-square continuous at $x_0$, then  \eqref{eq:proof_shoenberg_keeps_diff_m=0_relation_variogram} shows that $Y$ satisfies the same property. This concludes the case $m=0$. \par 
        In the following, $m\geq 1$. Let us show (1)$\implies$(2). We start by showing that $Z$ is $m$-times mean-square differentiable in an arbitrary point $x_0\in\R$. To this aim, we recall that \citep[Theorems 2.2.1, 2.2.2]{Adler_GeometryRandomFields_2010} $Z \in C^m(x_0)$ if and only if the mixed partial derivative $\partial_x^m \partial_y^m C_Z(x,y)$ exists finite in a neighbourhood of $x_0$. If
        \begin{equation*}
            \psi(s) = \integ{0}{+\infty}{e^{-st}}{\mu(t)},
        \end{equation*}
        then, for an interval $[a,b]\subset (0,+\infty)$ such that $\mu([a,b])>0$, we define
        \begin{equation*}
            C_{[a,b]}(x,y) := \integ{a}{b}{e^{-t\gamma_p(x,y)}}{\mu(t)}, \qquad q(x) := C_{[a,b]}(x, x_0) > 0.
        \end{equation*}
        Clearly, $C_{[a,b]}$ and is positive semidefinite over $\R \times \R$ (and so is $C_{[a,b]^C}:= \integ{[a,b]^C}{}{e^{-t\gamma_p(x,y)}}{\mu(t)}$), therefore, since $C_Y - C_{[a,b]} = C_{[a,b]^C} \succeq 0$, by applying Lemma \ref{lem:differentiability_under_covariance_domination}, we get that the process generated by $C_{[a,b]}$ belongs to $C^m(\R)$.\par 
        Now, by Theorem \ref{theo:schoenberg}, we know that there is a Hilbert space embedding $\varphi:\R \to \mathcal H$ such that $\gamma_p(x,y) = \norm{\varphi(x) - \varphi(y)}^2$ (wlog, we will assume that $\varphi(x_0)=0$). This implies that
        \begin{align*}
            &\exp\round{-t\norm{\varphi(x) - \varphi(y)}^2} \\
            &= \exp\round{-t\norm{\varphi(x)}^2} \ \exp\round{-t\norm{\varphi(y)}^2} \ \sum_{i=0}^{+\infty} \frac{(2t)^i}{i!} \angled{\varphi(x), \varphi(y)}^i\\
            &=\sum_{k=0}^{+\infty} \frac{(2t)^k}{k!} \exp\round{-t\norm{\varphi(x)}^2} \ \exp\round{-t\norm{\varphi(y)}^2}\angled{\varphi(x), \varphi(y)}^k.
        \end{align*}
        Therefore, we get 
        \begin{align*}
            &C_Y(x,y) = \integ{0}{+\infty}{\exp\round{-t\norm{\varphi(x) - \varphi(y)}^2}}{\mu(t)}\\
            &= \sum_{k=0}^\infty \overbrace{\integ{0}{+\infty}{\frac{(2t)^k}{k!} \exp\round{-t\norm{\varphi(x)}^2}\exp\round{-t\norm{\varphi(y)}^2}\angled{\varphi(x), \varphi(y)}^k}{\mu(t)}}^{H_k(x,y)},
        \end{align*}
        that is: $C_Y(x,y) = \sum_{k=0}^{\infty} H_k(x,y)$. We will now show that each $H_k$ is a positive definite function over $\R\times \R$: let $N\in \N^+$, $\dVec x\in\R^N$ and consider the $N\times N$ matrix $\matr H_{k,t} = [h_{k,t}(x_i,x_j)]_{i,j=1}^N$, where
        \begin{equation*}
             h_{k,t}(x_i,x_j) := \frac{(2t)^k}{k!} \exp\round{-t\norm{\varphi(x_i)}^2}\exp\round{-t\norm{\varphi(x_j)}^2}\angled{\varphi(x_i), \varphi(x_j)}^k.
        \end{equation*}
        It is clearly possible to express $\matr H_{k,t}$ as 
        \begin{equation*}
            \matr H_{k,t} = \frac{(2t)^k}{k!} \matr D_{t} \, \matr S_{k}^{\circ k} \, \matr D_{t},
        \end{equation*}
        where $\matr D_t$ is the $N\times N$ diagonal matrix whose (positive) entries are $D_{t,ii} = \exp\round{-t\norm{\varphi(x_i)}^2}$, $\matr S_k$ is the $N\times N$ matrix with entries $S_{k,ij} = \angled{\varphi(x_i), \varphi(x_j)}$, and ${}^{\circ k}$ denotes the $k$-th Hadamard matrix power. Now, since the Hadamard product of positive definite matrices is positive definite and since $\matr D_t$ is positive definite as well, we have that $\matr H_{k,t}$ is positive definite for every $N\in \N^+$ and $\dVec x\in\R^N$. In turn, since the integral of positive definite functions is positive definite, this means that $H_k$ is positive definite for any $k\in\N$. We can thus write 
        \begin{equation}
            \label{eq:proof_diff_Schoenberg_C_Y>H_1}
            C_Y(x,y) - H_1(x,y) = \sum_{k\ne 1} H_k(x,y) \succeq 0, \quad \implies \quad  C_Y \succeq H_1.
        \end{equation}
        Define now the covariance function $B$: 
        \begin{align*}
            B(x,y) := q(x) \ q(y) \ \angled{\varphi(x), \varphi(y)}.
        \end{align*}
        We have that, for each $N\in\N^+$, $\dVec x \in \R^N$ and $\dVec \alpha \in \R^N$, 
        \begin{align*}
            &\sum_{i,j=1}^N \alpha_i \alpha_j B(x_i, x_j)\\
            &= \integ{a}{b}{
                \integ{a}{b}{
                    \angled{
                        \overbrace{\sum_{i=1}^N \alpha_i e^{-t \norm{\varphi(x_i)}^2} \varphi(x_i)}^{v(t)}, \
                        \sum_{j=1}^N \alpha_j e^{-s \norm{\varphi(x_j)}^2}\varphi(x_j)
                    }
                }{\mu(s)}
            }{\mu(t)}\\
            &=\angled{\integ{a}{b}{v(t)}{\mu(t)}, \ \integ{a}{b}{v(s)}{\mu(s)}} = \norm{\integ{a}{b}{v(t)}{\mu(t)}}^2\\
            &\leq \round{\integ{a}{b}{\norm{v(t)}}{\mu(t)}}^2 = \round{\integ{a}{b}{\frac{1}{\sqrt{2t}}\, \sqrt{2t}\, \norm{v(t)}}{\mu(t)}}^2\\
            &\overset{C.S.}{\leq}\overbrace{\integ{a}{b}{\frac{1}{2t}}{\mu(t)}}^{c} \ \cdot \ \integ{a}{b}{2t\, \norm{v(t)}^2}{\mu(t)}\\
            & = c \cdot \integ{a}{b}{2t \angled{
                \sum_{i=1}^N \alpha_i e^{-t \norm{\varphi(x_i)}^2} \varphi(x_i), \ 
                \sum_{j=1}^N \alpha_j e^{-t \norm{\varphi(x_j)}^2} \varphi(x_j)}
            }{\mu(t)}\\
            &= c \cdot \sum_{i,j=1}^N \alpha_i \alpha_j \integ{a}{b}{2t \ \exp\round{-t(\norm{\varphi(x_i)}^2 + \norm{\varphi(x_j)}^2)}\angled{\varphi(x_i), \varphi(x_j)}}{\mu(t)}\\
            &\leq c \cdot \sum_{i,j=1}^N \alpha_i \alpha_j H_1(x_i,x_j).
        \end{align*}
        Namely, $c H_1 \succeq B$, which, in concert with (\ref{eq:proof_diff_Schoenberg_C_Y>H_1}), yields $cC_Y \succeq cH_1 \succeq B$. By Lemma \ref{lem:differentiability_under_covariance_domination}, this means that $B$ induces a process which is $m$-times mean-square differentiable. To conclude, it is sufficient to notice that, if $X$ is the process with covariance $B$, then $X/q$ has covariance 
        \begin{equation*}
            \angled{\varphi(x), \varphi(y)}=\Cov{\vphantom{\big|}Z(x) - Z(x_0)}{Z(y) - Z(x_0)}, 
        \end{equation*}
        from which we get that $Z(x,y)$ is $m$-times mean-square differentiable.
        \par
        Now we will show that all the derivatives up to order $m$ of $\psi$ exists finite in $0^+$. Since $\varphi : \R \to \mathcal{H}$ is non-constant, we choose $x_0\in\R$ such that $\varphi'(x_0) \ne 0$. Again, wlog, we assume that $\varphi(x_0) = 0$. Now, let $M>0$ and consider
        \begin{equation*}
            C_M(x,y):=\integ{0}{M}{
                \frac{(2t)^m}{m!} \overbrace{\exp\round{-t\norm{\varphi(x)}^2}}^{a_t(x)}
                \overbrace{\exp\round{-t\norm{\varphi(y)}^2}}^{a_t(y)}
                \angled{\varphi(x), \varphi(y)}^m
            }{\mu(t)}.
        \end{equation*}
        Clearly we have that $C_Y \succeq C_M$. Now, setting $f_x(y) := f_y(x) := \angled{\varphi(x), \varphi(y)}$, we differentiate $C_M$ at $(x_0,x_0)$: 
        \begin{align*}
            &\frac{\partial^{2m} C_M(x_0,x_0)}{\partial x^m \ \partial y^m}= \integ{0}{M}{
                \frac{(2t)^m}{m!} \frac{\partial^m}{\partial x^m} \round{
                    a_t(x) \, \frac{\partial^m}{\partial y^m} \round{a_t(y) \cdot f_x^m(y)}
                }
            }{\mu(t)}.
        \end{align*}
        By the Leibniz rule, we have that 
        \begin{align*}
            \frac{\partial^m}{\partial y^m} \round{\vphantom{\big|}a_t(y) \cdot f_x^m(y)} &= \frac{\partial^m}{\partial y^m} \round{\vphantom{\big|}a_t(y) \cdot f_x(y) \cdot ... \cdot f_x(y)}\\
            &=\sum_{j_0+...+j_m=m} \frac{m!}{j_0! \cdot ... \cdot j_m!} \ 
            a_t^{(j_0)}(y) \prod_{i=1}^m f_x^{(j_i)}(y).
        \end{align*}
        Since $f_x(x_0)=\angled{\varphi(x), \varphi(x_0)} =\angled{\varphi(x), 0}= 0$, when evaluating the above expression at $y=x_0$, the only term that survives is the one with $j_0=0$ and $j_1=...=j_m=1$. That is:
        \begin{equation*}
            \frac{\partial^m}{\partial y^m} \round{\vphantom{\big|}a_t(y) \cdot f_x^m(y)}\Big|_{y=x_0} = m! \ a_t(x_0) \round{f_x'(x_0)}^m = m! \ a_t(x_0) \ \angled{\varphi(x), \varphi'(x_0)}^m.
        \end{equation*}
        Repeating the same argument with the derivative with respect to $x$, we obtain that 
        \begin{equation*}
            \frac{\partial^{2m} C_M(x_0,x_0)}{\partial x^m \ \partial y^m}= 2^m \cdot m! \cdot \norm{\varphi'(x_0)}^{2m} \cdot  \integ{0}{M}{t^m}{\mu(t)}.
        \end{equation*}
        Since $Y \in C^m(\R)$ and $C_Y \succeq C_M$, by Lemma \ref{lem:differentiability_under_covariance_domination}, we get that the above quantity is bounded by the $(m,m)$-order mixed partial derivative of $C_Y$, which is finite since $Y\in C^m$. This holds for any $M>0$, therefore we can pass to the limit $M\to+\infty$ and obtain:
        \begin{equation*}
            \lim_{M\to +\infty}
            \frac{\partial^{2m} C_M(x_0,x_0)}{\partial x^m \ \partial y^m} <+\infty \quad \implies \quad \integ{0}{\infty}{t^m}{\mu(t)} = (-1)^m \psi^{(m)}(0^+) < +\infty,
        \end{equation*}
        namely $\abs{\psi^{(m)}(0^+)} < +\infty$. Lower moments clearly exist as well. This concludes (1)$\implies$(2).
        \par
        We will now show that (2)$\implies$(1). Notice that, for each fixed $t\geq 0$, the covariance $C_t(x,y) := \exp\round{-t\gamma_p(x,y)}$ inherits order-$m$ mean-square differentiability from $Z$. 
        Consider the chain-rule (Fa\`a di Bruno's formula) expansion of $\partial_x^m \partial_y^m C_t$:
        \begin{equation}
            \label{eq:proof_diff_Schoenberg_chain_rule}
            \partial_x^m \partial_y^m C_t(x,y) = e^{-t\gamma_p(x,y)} \sum_{\pi \in \Pi} (-t)^{\abs{\pi}} \prod_{B\in\pi} \partial_x^{b_x} \partial_y^{b_y} \gamma_p(x,y),
        \end{equation}
        where $\Pi$ is the set of partitions of $S_x \cup S_y$, with $S_x := \curly{1_x, ..., m_x}$ and $S_y := \curly{1_y, ..., m_y}$, $b_x := \abs{B \cap S_x}$ and $b_y := \abs{B \cap S_y}$ count the $x,y$ components in $B$. In particular, it holds
        \begin{equation*}
            b_x + b_y \geq 1, \qquad \sum_{B\in \pi} b_x = \sum_{B\in \pi} b_y = m.
        \end{equation*}
        On the diagonal $y=x$, we have $\gamma_p=0$ and $\partial_x \gamma_p = \partial_y \gamma_p = 0$. As a consequence, every partition $\pi$ containing a singleton gives no contribution in \eqref{eq:proof_diff_Schoenberg_chain_rule}, that is: every block in a surviving partition contains at least two labels. Since there are $2m$ labels in total, it means that there can be at most $m$ blocks. Thus:
        \begin{equation}
            \label{eq:proof_diff_Schoenberg_bound_derivative}
            \partial_x^m \partial_y^m C_t(x,y) \Big|_{y=x} = \sum_{j=1}^m A_{m,j}(x) t^j \leq c_x \sum_{j=1}^m t^j,
        \end{equation}
        where all the coefficients $A_{m,j}$ (and, thus, $c_x$) are finite since each derivative in \eqref{eq:proof_diff_Schoenberg_chain_rule} uses at most $m$ differentiation (in either $x,y$) of $\varphi$.\par 
        To conclude the proof, we would like integrate over $[0,+\infty)$ with respect to $d\mu(t)$. However, the dominated convergence theorem requires control of $\partial_x^m \partial_y^m C_t(x,y)$ \emph{near} $y=x$, not just on the diagonal. Therefore, we provide an explicit feature representation. Let $W = \curly{W(u) \ : \ u\in \mathcal{H}}$ an isonormal Gaussian process on the probability space $(\Omega, \mathcal F, \P)$, that is: $W(u)$ is a zero-mean Gaussian random variable and $\Exp{W(u_1) W(u_2)} = \angled{u_1,u_2}_{\mathcal{H}}$. Now we define the complex-valued feature $\Phi_t : \mathcal{H} \to L^2(\Omega,\mathbb{C})$
        \begin{equation*}
            \Phi_t(u) := \exp\round{i\sqrt{2t} W(u)}.
        \end{equation*}
        The Gaussian characteristic function formula gives 
        \begin{equation}
            \label{eq:proof_diff_Shoenberg_covariance_feature_map}
            \Exp{\Phi_t(u_1) \ol{\Phi_t(u_2)}} = \exp\round{-t \norm{u_1 - u_2}_{\mathcal{H}}^2}.
        \end{equation}
        Clearly, $u\mapsto \Phi_t(u)$ is smooth and, if $v_1, ..., v_r \in \mathcal{H}$ are directions along with we take the derivatives, we have that 
        \begin{equation*}
            D_{[v_1, ..., v_r]}^r \Phi_t(u) := \frac{\partial^r \Phi_t(u)}{\partial_{v_1} u \ ... \ \partial_{v_r} u} = \round{i \sqrt{2t}}^r \Phi_t(u) \prod_{j=1}^r W(v_j).
        \end{equation*}
        Now, Gaussian moment bounds imply that, uniformly in $u$,
        \begin{equation}
            \label{eq:proof_diff_Schoenberg_bound_derivative_isonormal_process}
            \norm{D_{[v_1, ..., v_r]}^r \Phi_t(u)}_{L^2(\Omega)} \leq b_r t^{r/2} \prod_{j=1}^r \norm{v_j}_{\mathcal{H}},
        \end{equation}
        where $b_r$ depends only on $r$. We will now combine all the features in one Hilbert space $\widetilde{\mathcal{H}} := L^2(\mu_t \otimes \P_\omega)$, $F(u)(t,\omega) := \Phi_t(u)(\omega)$. The hypothesis of having finite $\integ{0}{\infty}{t^m}{\mu(t)}$ and the finiteness of $\mu$ imply that, for all $r\in\curly{0,...,m}$,
        \begin{equation*}
            \integ{0}{+\infty}{t^r}{\mu(t)} < +\infty.
        \end{equation*}
        As a consequence, squaring \eqref{eq:proof_diff_Schoenberg_bound_derivative_isonormal_process} gives an integrable bound for each derivative up to order $m$. For example, for the first derivative:
        \begin{equation*}
            \frac{\Phi_t(u+\varepsilon v) - \Phi_t(u)}{\varepsilon} = \integ{0}{1}{D_v\Phi_t(u+\theta \varepsilon v)}{\theta},
        \end{equation*}
        which implies that the $L^2(\Omega)$-norm is bounded by $b_1 t^{1/2} \norm{v}_{\mathcal H}$, which, in turn, implies that the squared norm is bounded by $4b_1^2t \norm{v}_{\mathcal H}^2$, which is $\mu$-integrable. The dominated convergence theorem proves the convergence in $\mathcal H$ as $\varepsilon\to 0$. We apply now the same argument to $D^{r-1} \Phi_t$: by mean of \eqref{eq:proof_diff_Schoenberg_bound_derivative_isonormal_process} with order $r$, we get that $F \in C^m$ and its derivatives are obtained by differentiating the features inside the $L^2$ integral.\par 
        Finally, set $G(x) := F(\varphi(x))$. The ordinary Hilbert-space chain rule shows that $G$ is $m$-times differentiable. Therefore, by \eqref{eq:proof_diff_Shoenberg_covariance_feature_map}, we have that 
        \begin{equation*}
            C_Y(x,y) = \integ{0}{+\infty}{C_t(x,y)}{\mu(t)} = \angled{G(x),G(y)}_{\mathcal H}.
        \end{equation*}
        Differentiating this inner product, therefore, gives 
        \begin{align*}
            \partial_x^m \partial_y^m C_Y(x,y) = \angled{G^{(m)}(x), G^{(m)}(y)}_{\mathcal{H}} = \integ{0}{+\infty}{\partial_x^m \partial_y^m C_t(x,y)}{\mu(t)}.
        \end{align*}
        In particular, on the diagonal and using \eqref{eq:proof_diff_Schoenberg_bound_derivative}, this becomes
        \begin{equation*}
            \partial_x^m \partial_y^m C_Y(x,x) = \integ{0}{+\infty}{\partial_x^m \partial_y^m C_t(x,y)\Big|_{y=x}}{\mu(t)} \leq \integ{0}{+\infty}{c_x \sum_{j=1}^m t^j}{\mu(t)} < +\infty.
        \end{equation*}
    \end{proof}
    Theorem \ref{theo:schoenberg_keeps_diff_R^n} defines a strong direct link between the differentiability of the process $Z$ and the one of the process $Y$ built via the Schoenberg theorem (as described in Subsection \ref{ssec:covariance_functions_stochastic_processes}). Being the edges of metric graphs real segments, will hold on the interior of each edge of a metric graph. However, the vertices will need additional care, as the neighbourhood of a vertex of degree at least $3$ is not isomorphic to $\R^n$ and as the concept of ``differentiability'' changes with the parity of the relative order. \par
    Here, we give a characterisation of the differentiability at a given point for a Gaussian process on a metric graph.
    

    \begin{theoremrep}
        \label{theo:arbitrary_order_diff_characterisation}
        Let $Z$ be a zero-mean square-integrable process on $G$ with covariance function $C_Z$ and let $v\in V$. Assume that its edge-wise mean-square derivatives up to order $m$ exist in a neighbourhood of $v$. Then $Z\in C^m(v)$ iff, for each $j\in \curly{0,...,m}$ and $y\in G$, it holds
        \begin{itemize}
            \item if $j$ is even,
            \begin{equation}
                \label{eq:arbitrary_diff_even}
                \frac{\partial_e^j C_Z(x,y)}{\partial_e x^j}\Big|_{x=x_0}
            \end{equation}
            does not depend on $e\in E_{v}$;
            \item if $j$ is odd,
            \begin{equation}
                \sum_{e\in E_{v}} \frac{\partial_e^j C_Z(x,y)}{\partial_e x^j}\Big|_{x=v} = 0.
            \end{equation}
        \end{itemize}
    \end{theoremrep}
    \begin{proof}[Proof of Theorem \ref{theo:arbitrary_order_diff_characterisation}]
        First, we recall that, if $y\in G$, $v\in V$ and $e\in E_v$, we have 
        \begin{equation*}
            \Cov{\partial_e^j Z(v)}{Z(y)} = \frac{\partial^{j}  C_Z(x,y)}{\partial_e x^j}\Big|_{x=v}. 
        \end{equation*}
        As a consequence, for $j$ even, we must have that for each $e,e'\in E_V$, it holds $\partial_e^j Z(v) = \partial_{e'}^j Z(v)$, which implies, for any $y\in G$,
        \begin{equation*}
            0=\Cov{\partial_e^j Z(v) - \partial_{e'}^j Z(v)}{Z(y)} = \frac{\partial^{j}  C_Z(x,y)}{\partial_e x^j}\Big|_{x=v} - \frac{\partial^{j}  C_Z(x,y)}{\partial_{e'} x^j}\Big|_{x=v},
        \end{equation*}
        that is: $\partial_{e,x}^j C_Z(x,y) \big|_{x=v}$ does not depend on $e\in E_v$, for each $y\in G$. This condition is also sufficient to $\partial_e^j Z(v) = \partial_{e'}^j Z(v)$, since, being $\partial_e^j Z(v) - \partial_{e'}^j Z(v)$ uncorrelated with any $Z(y)$ and being both the covariance and the derivative linear, $\partial_e^j Z(v) - \partial_{e'}^j Z(v)$ is also uncorrelated with itself, namely is $0$ a.s..\par
        For $j$ odd, the argument is the same. Indeed, we have $\sum_{e\in E_v} \partial_e^j Z(v) = 0$, which is equivalent to
        \begin{equation*}
            0 = \Cov{\sum_{e\in E_v} \partial_e^j Z(v)}{Z(y)} = \sum_{e\in E_v} \frac{\partial^j C_Z(x,y)}{\partial_e x^j}\Big|_{x=v}.
        \end{equation*}
    \end{proof}

    As an almost immediate consequence, we get the following corollary, which states that there are no non-trivial differentiable processes isotropic w.r.t. distances that behave locally like the shortest-path distance. We stress the connection with \cite[Proposition 6]{AnderesEtAl_IsotropicCovarianceFunctions_2020}.
    
    \begin{corollaryrep}
        \label{cor:isotr_and_diff_implies_constant}
        Let $G$ be a metric graph and let $Z: G\to\R$ be a process isotropic w.r.t. a distance $d$ that is \emph{locally equivalent} to the shortest-path distance $d_{SP}$ around a branching point $v$ (\ie a vertex with $\deg v \geq 3$), namely $\exists \alpha > 0,r > 0$, for which
        \begin{equation*}
            d(x_1,x_2) = \alpha \, d_{SP}(x_1,x_2) + o(d_{SP}(x_1,x_2)), \qquad x_1,x_2 \in B_r(v),
        \end{equation*}
        where $B_r(v)$ is the ball centred in $v$ with (shortest-path) radius $r$. If $Z$ is differentiable on a point on $B_r$, then $Z$ is a.s. constant on $G$.
    \end{corollaryrep}
    \begin{proof}[Proof of Corollary \ref{cor:isotr_and_diff_implies_constant}]
        Let $0<x<r$, $n := \deg(v)$ and let $\dVec x \in G^{n+1}$ the vector containing $v$ and the $n$ points that have (shortest-path) distance $x$ from $v$, each on a different edge $e\in E_v$. Since $\dVec x \subset B_r(v)$, we have that the distance matrix of $\dVec x$ has the following structure:
        \begin{equation*}
            \matr D(\dVec x) := \square{d(x_i, x_j)}_{j=1}^{n+1} = \alpha \begin{bmatrix}
                0 & x & x & \dots & x \\
                x & 0 & 2x & \dots & 2x \\
                x & 2x & 0 & \dots & 2x \\
                \vdots & \vdots & \vdots & \ddots & \vdots \\
                x & 2x & 2x & \dots & 0 
            \end{bmatrix} + o(x).
        \end{equation*}
        Wlog, we assume that $\Var Z \equiv 1$. Now, since $Z$ is isotropic and is differentiable on $B_r(v)$, assuming $Z$ non constant, we have that $C_Z(d) = \psi(d)$, with $\psi$ satisfying $\psi(d) = 1 - c d^\beta + o(d^\beta)$ for $c \geq 0, \beta \geq 2$ and for small $d$. Now, if $c=0$, then 
        \begin{equation*}
            \Cov{Z'(x_1)}{Z'(x_2)} = \frac{\partial^2}{\partial_e x_1 \partial_e x_2} \Cov{Z(x_1)}{Z(x_2)} = 0     
        \end{equation*}
        for sufficiently close $x_1,x_2$. This implies that $Z$ is locally constant, thus globally constant (being $G$ connected), thus necessarily $c>0$. Consider now the covariance matrix $\matr \Sigma = \psi(\matr D(\dVec x))$: it must be positive semidefinite. Therefore, the Schur complement
        \begin{equation*}
            \matr \Sigma / \Sigma_{11} = \begin{bmatrix}
                1 & \psi(2\alpha x) & \dots & \psi(2\alpha x) \\
                \psi(2\alpha x) & 1 & \dots & \psi(2\alpha x) \\
                \vdots & \vdots & \ddots & \vdots \\
                \psi(2\alpha x) & \psi(2\alpha x) & \dots & 1
            \end{bmatrix} - \begin{bmatrix}
                \psi(\alpha x)\\
                \psi(\alpha x)\\
                \vdots\\
                \psi(\alpha x)
            \end{bmatrix} \ \square{1} \ \begin{bmatrix}
                \psi(\alpha x)\\
                \psi(\alpha x)\\
                \vdots\\
                \psi(\alpha x)
            \end{bmatrix}^\top
        \end{equation*}
        must be positive semidefinite as well. In particular, it must satisfy
        \begin{align*}
            0 &\leq \dVec 1_n^\top (\matr \Sigma / \Sigma_{11}) \dVec 1_n = n(1 + (n-1) \psi(2x) - n\psi^2(x)) \\
            \iff 0 & \leq  1 + (n-1) (1 - (2\alpha x)^\beta + o(x^\beta)) - n (1 - (\alpha x)^\beta + o(x^\beta))^2\\
            \iff 0 &\leq \alpha^\beta x^\beta (2n - (n-1) \cdot 2^\beta) + o(x^\beta),
        \end{align*}
        which is false for $x$ sufficiently small, since $n\geq 3$ and $\beta \geq 2$. This means that it cannot be $\psi(d) = 1 - d^\beta + o(d^\beta)$: the only other possibility, assuming $Z$ differentiable is that $\psi \equiv 1$, namely $Z$ is constant.
    \end{proof}

    We stress that the requirement $\alpha>0$ in the statement of Corollary \ref{cor:isotr_and_diff_implies_constant} is instrumental: it is equivalent to requiring that $d$ is not flat on the diagonal. \par

    \begin{remark}
        \label{rem:effective_resistance_is_equivalent_to_SP}
        The effective resistance distance is locally equivalent to the shortest-path distance, with constant $\alpha=1$. Indeed, fix $x\in G$ and assume $y \in G$ sufficiently close to $x$. Then, there exist exactly one path $p\subset G$ from $x$ to $y$ that has the shortest-path distance, and there are no branching points on its interior (namely there is an edge $e\in E$ s.th. $p\subseteq e$). As a consequence, by standard law of parallel and series resistors, it is possible to reduce $G\setminus p$ (if it is connected, otherwise the resistance distance coincides with the shortest-path one and the argument is trivial) to a single resistor. This means that, for the purpose of computing $d(x,y)$, the graph $G$ is equivalent to a circle of length, say, $l$. If now $d$ is the shortest-path distance between $x$ and $y$, then
        \begin{equation*}
            d_R(x,y) = \round{\frac{1}{d} + \frac{1}{l-d}}^{-1} = d - \frac{d^2}{l} = d + o(d).
        \end{equation*}
    \end{remark}

    From Corollary \ref{cor:isotr_and_diff_implies_constant} and Remark \ref{rem:effective_resistance_is_equivalent_to_SP} we immediately get the following result.
    \begin{corollary}
        There are no non-trivial differentiable processes on branching metric graphs that are isotropic w.r.t. the effective resistance distance.
    \end{corollary}
    
    Next, we prove that if a ``distance'' is flat on the diagonal, then it cannot satisfy the triangle inequality.
    \begin{propositionrep}
        \label{prop:no_flat_distance_on_R}
        Let $I\subseteq\R$ with non-empty interior. Then there is no distance $d:I\times I\to \R$ that is flat on the diagonal, \ie that satisfies
        \begin{equation}
            \label{eq:flat_diagonal_def}
            \forall x\in I, \quad\frac{\partial d(x,y)}{\partial y}\bigg|_{x=y} = 0.
        \end{equation}
    \end{propositionrep}
    \begin{proof}[Proof of Proposition \ref{prop:no_flat_distance_on_R}]
        Without loss of generality, we assume that $I=[a,b]$ is a closed interval, with $a<b$. Indeed, any set with non-empty interior contains such an interval.

        Let us rewrite equation (\ref{eq:flat_diagonal_def}): $\forall x \in I$,
        \begin{equation*}
            0 = \lim_{y\to x} \frac{d(x,y)-d(x,x)}{y-x}=\lim_{y\to x} \frac{d(x,y)}{y-x}.
        \end{equation*}
        By definition of limit, one gets $\forall x \in I, \forall \varepsilon>0, \exists \delta_{x,\varepsilon}$ s.th.
        \begin{equation}
            \label{eq:inequality_d_differentiable}
             \forall y\in B(x,\delta_{x,\varepsilon}),\, d(x,y)<\varepsilon \abs{x-y},
        \end{equation}
        being $B(x,\delta):= (x-\delta, x+\delta)$. Now consider the following open cover of $I$:
        \begin{equation*}
            I \subseteq \bigcup_{x\in I} B(x,\delta_{x,\varepsilon}).
        \end{equation*}
         Since $I$ is compact, there exist a finite subcover: there are $x_1<\dots<x_m \in I$ s.th.
        \begin{equation*}
            I \subseteq \bigcup_{i=1}^m B(x_i,\delta_{x_i,\varepsilon}).
        \end{equation*}
        Now, there will exist a $\overline{\delta}>0$ smaller than all the $\delta's$ and all the (half) lengths of the overlapping regions $B(x_i,\delta_{x_i, \varepsilon}) \cap  B(x_{i+1}, \delta_{x_{i+1},\varepsilon})$. This guarantees that, for any $x\in I$, $B(x,\overline{\delta})$ is contained in at least one interval $B(x_i,\delta_{x_i,\varepsilon})$. Therefore, we can apply inequality (\ref{eq:inequality_d_differentiable}).\par
        We conclude the proof by contradiction: consider the grid 
        \begin{equation*}
            \overline{x}_j:=\begin{dcases}
                a + (j-1) \overline{\delta}\qquad&\text{for }j\in\curly{1,\dots,\overline{n}-1}\\
                b\qquad&\text{for }j=\overline{n},
            \end{dcases}    
        \end{equation*}
        where $\overline{n}:=\lceil \frac{b-a}{\overline{\delta}}\rceil$. By the triangle inequality of $d$ and by (\ref{eq:inequality_d_differentiable}), we have:
        \begin{equation*}
            d(a,b)\leq\sum_{j=1}^{\overline n-1} d(x_j,x_{j+1}) \leq \sum_{j=1}^{\overline n-1} \varepsilon \abs{x_{j+1} - x_j}= \varepsilon (b-a).
        \end{equation*}
        The above inequality holds for each $\varepsilon>0$, therefore necessarily $d(a,b)=0$, but this is impossible since $d$ is a distance and $a\ne b$.  
    \end{proof}

    \begin{corollaryrep}
        \label{cor:no_flat_distance_on_loc_Euclidean}
        In any locally-Euclidean space, there is no distance that is flat on the diagonal.
    \end{corollaryrep}
    \begin{proof}[Proof of Corollary \ref{cor:no_flat_distance_on_loc_Euclidean}]
        A locally-Euclidean space is in local homeomorphism with an open subset of $\R^m$, which contains a segment of the real line. Thus, Proposition \ref{prop:no_flat_distance_on_R} applies as any open set of $\R^m$ contains a segment of the real line. 
    \end{proof}

    \begin{corollaryrep}
        \label{cor:d_composed_CM_cannot_be_diff}
        Let $d$ be a (proper) distance on a space $X$ that contains a non-trivial interval and let $\psi$ be a non-constant completely monotonic function such that $\psi(0) < \infty$. Assume that $\psi\circ d$ is positive semidefinite on $X\times X$. Then the covariance $C_Y$ cannot generate a process that is mean-square differentiable.
    \end{corollaryrep}
    Although it may seem that the Corollary \ref{cor:d_composed_CM_cannot_be_diff} clashes with the construction given in Section \ref{sec:polyharmonic_distances}, it does not. Indeed, here we are stating that we cannot compose a distance \emph{directly} with a completely monotonic function and obtain a mean-square differentiable model. However, in Section \ref{sec:polyharmonic_distances}, we claim that the zero-mean process with covariance $\psi\round{\gamma_p(\cdot,\cdot)} = \psi\round{d_p^2(\cdot,\cdot)}$ is $p-1$-times mean-square differentiable: the square makes all the difference.
    
    \begin{theoremrep}
        \label{theo:C_1^m_implies_Schoenberg_C_1^m}
        Let $G$ be a metric graph and let $Z:G\to \R$ be a square-integrable process such that $Z\in C_k^m(x)$, where $x\in G$ and where $1\leq k \leq m \leq \infty$. Let $\psi$ be a non-constant completely monotonic function such that $\psi^{(m)}(0^+)$ is finite (if $m=+\infty$, this reads $\psi^{(j)}(0^+)<+\infty$ for all $j\in \N$). If we set $Y:=\Sch_\psi(Z)$, then $Y\in C_1^m(x)$.
    \end{theoremrep}
    \begin{proof}[Proof of Theorem \ref{theo:C_1^m_implies_Schoenberg_C_1^m}]
        If $x \not \in V$, then there exist a neighbourhood of $x$ which is isomorphic to $\R$, and therefore the result follows directly from Theorem \ref{theo:schoenberg_keeps_diff_R^n}. As a consequence, we have to show the result for $x \in V$: we are going to apply the characterisation of Theorem \ref{theo:arbitrary_order_diff_characterisation}. Let $y\in G$: we need to show that
        \begin{align*}
            &\sum_{e\in E_x} \frac{\partial C_Y(x,y)}{\partial_e x} = 0 \\
            \iff &\sum_{e\in E_x} \frac{\partial \psi(\gamma_p(x,y))}{\partial_e x} = 0 \\
            \iff &\sum_{e\in E_x} \psi'(\gamma_p(x,y)) \frac{\partial}{\partial_e x} \round{\vphantom{\Big|}C_Z(x,x) + C_Z(y,y) - 2C_Z(x,y)}=0\\
            \iff & \sum_{e\in E_x} \round{2 \frac{\partial}{\partial_e x} C_Z(x,y') \big|_{y'=x} - 2 \frac{\partial}{\partial_e x} C_Z(x,y)} = 0,
        \end{align*}
        which is true since $Z \in C_1^1$, by the characterisation in Theorem \ref{theo:arbitrary_order_diff_characterisation}.
    \end{proof}
    
    One may ask whether the above result holds for higher-order derivatives. The next Proposition gives a negative answer.
    \begin{propositionrep}
        \label{prop:Cm_not_implies_Schoenberg_Cm}
        With the same hypotheses of Theorem \ref{theo:C_1^m_implies_Schoenberg_C_1^m}, let $v\in V$ be a vertex with $\deg(v)\geq 3$. Then, the fact that $Z\in C_k^m(v)$ does \textbf{not} imply that $Y\in C_{k'}^m(v)$ for $k' \geq 2$.
    \end{propositionrep}
    \begin{proof}[Proof of Proposition \ref{prop:Cm_not_implies_Schoenberg_Cm}]
        We exhibit a counter-example. Consider $G$ the star graph with $k \geq 2$ leaves and with edges all having length $1$. First, we define $Z$ on the vertices: on the centre of the star, say $v_0$, we define $Z(v_0):=0$. Furthermore, on the leaves, say $v_1, ..., v_k$, we define:
        \begin{equation*}
            \rVec Z_V := \round{Z(v_1),\dots,Z(v_k)}^\top \sim \NormMV{k}{\dVec 0}{\matr \Sigma}, \qquad \matr \Sigma := k \matr I_k - \matr 1_{k\times k},
        \end{equation*}
        where $\matr 1_{k\times k}$ denotes the matrix full of ones. It is immediate to see that $\matr \Sigma$ has rank $k-1$ and its kernel is generated by the vector $\dVec 1_k$. As a consequence, $\dVec 1_k^\top \rVec Z_V = 0$ almost surely. Next, we define $Z$ on each edge $e_i := (v_0, v_i)$, for $i \in \curly{1,...,k}$, to be the linear interpolation of $Z(v_0)=0$ and $Z(v_i)$. Therefore, if $(i, x) \in \curly{1,\dots,k} \times [0,1]$ denotes the point belonging to the edge $e_i$ and having distance $x$ from $v_0$, we have:
        \begin{equation*}
            Z((i, x)) := x Z(v_i).
        \end{equation*}
        Notice that $Z$ is infinitely-many times differentiable in $v_0$: $Z$ is continuous, the exiting derivatives are $\frac{\text{d}}{\text{d}x} Z((i,x)) = Z(v_i)$ (which sum up to $0$), and all the higher-order derivatives are null, as $Z$ is linear on each edge. Therefore all the Kirchhoff conditions are satisfied in $v_0$. \par
        We are now going to show that $Y$ defined via the Schoenberg theorem is not even $C^2(v_0)$. From the definition of $Z$, it is immediate to compute the covariance function (and therefore the variogram):
        \begin{align*}
            C_Z((i,x),(j,y))&=\begin{dcases}
                (k-1) \ xy \qquad&\text{if $i=j$},\\
                -xy \qquad &\text{if $i\ne j$},
            \end{dcases} \\
            \gamma_Z((i,x),(j,y)) &= \begin{dcases}
                (k-1)(x-y)^2 \qquad &\text{if $i=j$},\\
                (k-1)(x^2+y^2) + 2 xy &\text{if $i\ne j$}.
            \end{dcases}
        \end{align*}
        Thus, the covariance function of $Y$ is given by
        \begin{equation*}
            C_Y((i,x),(j,y)) = \begin{dcases}
                \exp\round{-(k-1)(x-y)^2}\qquad &\text{if $i=j$},\\
                \exp\round{-(k-1)(x^2+y^2) - 2 xy} \qquad &\text{if $i\ne j$}.
                \end{dcases}
        \end{equation*}
        This, in turn, implies that the second-order derivative evaluated at $x=0^+$ reads:
        \begin{equation*}
            \frac{\partial^2 C_Y}{\partial^2 x} \Bigg|_{x=0^+} = \begin{dcases}
                2 \, {\left(2 \, k y^{2} - 2 \, y^{2} - 1\right)} {\left(k - 1\right)} e^{(1-k) y^{2}}\qquad &\text{if $i=j$},\\
                2 \, {\left(2 \, y^{2} - k + 1\right)} e^{(1-k) y^{2}}&\text{if $i\ne j$}.
            \end{dcases}
        \end{equation*}
        In order to be continuous at $0$, the two cases should be the same (for any $y$). But this holds if and only if $k = 2$. This confirms the positive result of differentiability on the edges ($k=2$ means that the vertex $v_0$ is a degree-2 vertex, and therefore can be removed), and shows that, in general, the differentiability of $Z$ does not translates on $Y$ on the vertices.
    \end{proof}
    The reason underlying the above negative result is that the composition of $\gamma_p$ with the negative exponential in the Schoenberg construction introduces non-linearities.
    
\section{Examples}
\label{sec:examples}
Here we provide some examples of simple metric graph and their associated variograms $\gamma_p$ and polyharmonic distances $d_p$. The graphs used are depicted in Figure \ref{fig:example_graphs}.

\savebox{\imagebox}{
    \begin{tikzpicture}[scale = 0.7]
        \def\n{3}    
        \def\r{2.3}  
        
        \coordinate (orig) at (0,0);
        \node[label={[label distance=-1mm]{0}:{$v_0$}}] at (orig) {};
        \fill (orig) circle (2pt);

        \coordinate (p) at (120:{-\r/2});
        \node[left] at (p) {$p$}; 
        \fill (p) circle (2pt);
        
        \foreach \i in {1,...,\n}{
            \coordinate (v\i) at ({(\i-1)*360/\n + 180}:\r);
            \node[label={[label distance=-1.5mm]{(\i-1)*360/\n + 180}:{$v_\i$}}] at (v\i) {}; 
            \fill (v\i) circle (2pt); 
            \draw (orig) -- (v\i);       
            
        }
    \end{tikzpicture}
}%
\begin{figure}[t]
    \centering
    \begin{subfigure}{0.31\textwidth} 
        \centering\raisebox{\dimexpr.5\ht\imagebox-.5\height}{
            \begin{tikzpicture}[scale = 0.7]
                \coordinate (v1) at (0,0) {};
                \coordinate (v2) at (3,0) {};
                \coordinate (p) at (3/2,0) {};
                
                \fill (v1) circle (2pt);
                \fill (v2) circle (2pt);
                \fill (p) circle (2pt);
                
                \draw (v1) -- (v2);
                
                \node[below] at (v1) {$v_1$};
                \node[below] at (v2) {$v_2$};
                \node[below] at (p) {$p$};
            \end{tikzpicture}
        }
        \caption*{$G_1$}
    \end{subfigure}
    \hfill
    \begin{subfigure}{0.31\textwidth}
        \centering\usebox{\imagebox}
        \caption*{$G_2$}
    \end{subfigure}
    \hfill
    \begin{subfigure}{0.31\textwidth}
        \centering\raisebox{\dimexpr.5\ht\imagebox-.5\height}{
            \begin{tikzpicture}[scale = 0.7]
                \def\r{2.0}
                
                \foreach \i in {1,...,3}{
                    \coordinate ({v\i}) at ({60+120*\i}:\r);
                    \fill ({v\i}) circle (2pt);
                    \node[label={[label distance=-1.5mm]{60+120*\i}:{$v_\i$}}] at ({v\i}) {};
                }
                
                \draw (0,0) circle (\r);
            \end{tikzpicture}
        }
        \caption*{$G_3$}
    \end{subfigure}
    \caption{The graphs used as examples in Subsection \ref{sec:examples}.}
    \label{fig:example_graphs}
\end{figure}

\begin{propositionrep}
    \label{prop:gamma_on_segment_0l}
    Let $G$ be the segment of length $l>0$. Then, for each $x_1,x_2 \in G$, it holds:
    \begin{align}
        \label{eq:gamma_on_segment_0l} \gamma_p(x_1,x_2) = \frac{2l^{2p-1}}{\pi^{2p}} \sum_{k=1}^\infty \frac{1}{k^{2p}}\round{\cos\round{\frac{\pi k x_1}{l}} - \cos\round{\frac{\pi k x_2}{l}}}^2,\\
        \label{eq:gamma_on_segment_0l_limit} \frac{\pi^{2p}}{2l^{2p-1}}  \ \gamma_p(x_1,x_2) \overset{m\to\infty}{\longrightarrow} \round{\cos\round{\frac{\pi x_1}{l}} - \cos\round{\frac{\pi x_2}{l}}}^2.
    \end{align}
\end{propositionrep}
\begin{proof}[Proof of Proposition \ref{prop:gamma_on_segment_0l}]
    We are going to exploit the Fourier series representation of $f\in \mathcal{F}$. However, we first extend $f$ evenly on $[-l,0]$: $f(-x):=f(x)$ for $x\in[0,l]$. In this way, $f$ is even and the Fourier series will only have even coefficients (\emph{videlicet} only cosine terms). This, in turn, guarantees that all the odd derivatives of $f$ at $0$ and $l$ will be null. Now, the Fourier expansion looks (notice that we changed the coefficient of the $\phi_k$'s as they must be orthonormal on $[0,l]$)
    \begin{equation*}
        f(x) = \sum_{k = 0}^\infty a_k \phi_k(x), \qquad \phi_k(x) = \sqrt{\frac{2}{l}}\cos\round{\frac{\pi k}{l}\ x}, \qquad x\in[-l,l].
    \end{equation*}
     Now, the eigenvalue of the Laplace operator $Lf:=-f''$ associated to $\phi_k$ is $\lambda_k=\round{\frac{\pi k}{l}}^2$. In addition, we can neglect the constant term associated to $a_0$. Now, it is sufficient to use Proposition \ref{prop:spectral_form_gamma} to get
    \begin{align*}
        \gamma_p(x_1,x_2) &= \sum_{k=1}^\infty \frac{l^{2p}}{\pi^{2p} k^{2p}} \round{\sqrt{\frac{2}{l}}\cos\round{\frac{\pi k x_1}{l}} - \sqrt{\frac{2}{l}}\cos\round{\frac{\pi k x_2}{l}}}^2,
    \end{align*}
    which is (\ref{eq:gamma_on_segment_0l}), which in turn leads to (\ref{eq:gamma_on_segment_0l_limit}) by taking $m\to +\infty$.
\end{proof}

\begin{propositionrep}
    \label{prop:gamma_on_circle_S1}
    Let $G$ be the circle graph of length $l>0$. Then, for each $x_1,x_2 \in G$, it holds
    \begin{align}
        \label{eq:gamma_on_circle_S1}
        \gamma_p(x_1,x_2) = \frac{4 l^{2p-1}}{(2\pi)^{2p}} \sum_{k=1}^{+\infty} \frac{1 - \cos\round{\frac{2 \pi k d}{l}}}{k^{2p}},\\
        \frac{(2 \pi)^{2p}}{l^{2p-1}} \ \gamma_p(d) \longrightarrow 4 \round{1 - \cos\round{\frac{2\pi d}{l}}},\nonumber
    \end{align}
    where $d$ is the great-circle distance on $G$ between $x_1$ and $x_2$. See also \citet[Example 3.2]{BerkolaikoLiu_SimplicityEigenvaluesNonvanishing_2017}.
\end{propositionrep}
\begin{proof}[Proof of Proposition \ref{prop:gamma_on_circle_S1}]
    Clearly, $G$ is isomorphic to the segment $[0,l]$ with the identification $0=l$. This also guarantees that all the Kirchhoff conditions are automatically satisfied in $0=l$. Therefore, we can exploit the classical (full) Fourier basis on $[0,l]$:
    \begin{align*}
        \phi_k(x) = \begin{dcases}
            \sqrt{\frac{2}{ l}}\cos\round{\frac{2\pi k}{l}x}\qquad\text{if $k > 0$},\\
            \sqrt{\frac{2}{ l}}\sin\round{\frac{2\pi k}{l}x}\qquad\text{if $k < 0$},
        \end{dcases}
    \end{align*}
    where the $k=0$ has null eigenvalue and therefore is not considered. Clearly $\phi_k$ is associated to the eigenvalue $\lambda_k = \round{\frac{2\pi k}{l}}^2$. This implies that (\ref{eq:spectral_form_gamma}) reads
    \begin{align*}
        \gamma_p(x_1,x_2) &= \sum_k \frac{(\phi_k(x_1) - \phi_k(x_2))^2}{\lambda_k^{p}} \\
        &= \frac{2}{l} \sum_{k=1}^\infty \round{\frac{l}{2\pi k}}^{2p} \round{\round{\cos{\frac{2\pi k x_1}{l}} - \cos{\frac{2\pi k x_2}{l}}}^2 + \round{\sin{\frac{2\pi k x_1}{l}} - \sin{\frac{2\pi k x_2}{l}}}^2}\\
        &=\frac{2 l^{2p-1}}{(2\pi)^{2p}} \sum_{k=1}^\infty \frac{2}{k^{2p}} \round{1 - \cos \frac{2\pi k x_1}{l} \cos \frac{2\pi k x_2}{l} - \sin \frac{2\pi k x_1}{l} \sin \frac{2\pi k x_2}{l}}\\
        &=\frac{4 l^{2p-1}}{(2\pi)^{2p}} \sum_{k=1}^{+\infty} \frac{1 - \cos\round{\frac{2 \pi k (x_1-x_2)}{l}}}{k^{2p}}.
    \end{align*}
\end{proof}

\begin{propositionrep}
    \label{prop:gamma_on_star_graph}
    Let $G$ be the star graph with $N\geq 2$ leaves and with all the edges of length $l$. We parametrise each edge $e_j$, for $j\in\curly{1,...,N}$, as $[0,l]$, where $0$ is the centre of the star and $l$ is the leaf. Finally, let $x_1,x_2\in G$, with $x_1\in e_{j_1}$ and $x_2\in e_{j_2}$. Then, as $p\to +\infty$, we have that $\round{\frac{\pi}{2l}}^{2p}\gamma_p(x_1,x_2)$ converges to
    \begin{equation*}
        \begin{dcases}
            \frac{2(N-1)}{Nl} \round{\sin \frac{\pi x_1}{2l} - \sin \frac{\pi x_2}{2l}}^2 \qquad &\text{if $j_1=j_2$},\\
            \frac{2}{Nl} \round{(N-1)\round{\sin^2\frac{\pi x_1}{2l} + \sin^2 \frac{\pi x_2}{2l}} + 2 \sin \frac{\pi x_1}{2l} \sin \frac{\pi x_2}{2l}}\qquad &\text{if $j_1\ne j_2$}.
        \end{dcases}
    \end{equation*}
    In particular, for $p\to+\infty$ and $N\to+\infty$, we have that $\round{\frac{\pi}{2l}}^{2p}\gamma_p(x_1,x_2)$ converges to
    \begin{equation*}
        \begin{dcases}
            \frac{2}{l} \round{\sin \frac{\pi x_1}{2l} - \sin \frac{\pi x_2}{2l}}^2 \qquad &\text{if $j_1=j_2$},\\
            \frac{2}{l} \round{\sin^2\frac{\pi x_1}{2l} + \sin^2 \frac{\pi x_2}{2l}}\qquad &\text{if $j_1\ne j_2$}.
        \end{dcases}
    \end{equation*}
\end{propositionrep}
\begin{proof}[Proof of Proposition \ref{prop:gamma_on_star_graph}]
    We are making use of Proposition \ref{prop:spectral_form_gamma}. Although it is in principle possible to obtain a spectral representation of $\gamma_p$ for finite $p$, we just derive the limit one for $p\to+\infty$. We follow \cite[Example 1.4.3]{BerkolaikoKuchment_IntroductionQuantumGraphs_2012} for the finding of the eigenvalues and eigenfunctions of the Laplace operator on the star graph.\par
    Let $\lambda=\omega^2$ be a generic eigenvalue. Then, on each edge $e$, we have that the associated eigenfunction writes
    \begin{equation*}
        \phi_e(x) = a_e \cos(\omega x) + b_e \sin(\omega x).
    \end{equation*}
    Now, imposing the Kirchhoff conditions at $0$ (the centre of the star graph) give $a_e = a$ for all $e\in E$ and $\sum_{e} b_e = 0$. In addition, the Kirchhoff conditions at the leaves give $b_e \cos(\omega l) = a \sin(\omega l)$. Now, if $\cos(\omega l)\ne 0$, this implies that 
    \begin{equation*}
        0=\sum_e b_e=\sum_e a \tan(\omega l) \iff \tan(\omega l) = 0 \iff \omega = \frac{k\pi}{l}, k\in \Z.
    \end{equation*}
    If, instead, $\cos(\omega l)=0$, that is $\omega l = \frac{\pi}{2} + k\pi$, $k\in \Z$, then we have $a=0$, which means that 
    \begin{equation*}
        \phi_e(x) = b_e\sin(\omega x), \qquad \omega = \frac{\pi+ 2lk\pi}{2l}, \qquad \sum_e b_e = 0.
    \end{equation*} 
    By comparing the two families of eigenvalues we found, we notice that the smallest non-zero eigenvalue comes from the latter and is $\lambda = \round{\frac{\pi}{2l}}^2$. The associated eigenfunction(s) are 
    \begin{equation*}
        \curly{\phi_{e,j}(x) = b_{e,j} \sin\round{\frac{\pi x}{2l}} \ : \ \sum_e b_{e,j} = 0}_{j=1}^N.
    \end{equation*}
    However, they are not linearly independent. Finding an orthonormal basis of this space is the same as finding an orthonormal basis of $\curly{\dVec b \in \R^N \ : \ \dVec 1^\top \dVec b = 0}$ with respect to the dot product inherited by $\angled{\cdot,\cdot}_{L^2}$. It turns out that a solution is given by:
    \begin{equation*}
        \phi_{e,j}(x) = \sqrt{\frac{2}{l j (j+1)}} \ u_{e,j} \ \sin\round{\frac{\pi x}{2l}}, \qquad u_{e,j} = \begin{dcases}
            1 \quad&\text{if $1\leq e \leq j$} \\
            -j \quad&\text{if $e=j+1$}\\
            0\quad&\text{if $e \geq j+2$},
        \end{dcases}
    \end{equation*}
    where $j\in\curly{1,...,N-1}$ and where the edges $e\in E$ have been arbitrarily labelled by $E = \curly{1,\dots, N}$. Now we need to compute the sum in (\ref{eq:spectral_form_gamma}) for the first $N-1$ eigenfunctions (as they are all associated to the smallest eigenvalue). For the sake of simplicity, we will write $A_i:=\sin\frac{\pi x_i}{2l}$, for $i\in\curly{1,2}$. Let us first consider the case in which $x_1$ and $x_2$ live on the same edge (say $e=1$):
    \begin{align*}
        \sum_{j=1}^{N-1} (\phi_j(x_1) - \phi_j(x_2))^2 &= \sum_{j=1}^{N-1} \frac{2}{lj(j+1)}\round{A_1 - A_2}^2\\
        &=\frac{2(N-1)}{Nl}\round{A_1 - A_2}^2.
    \end{align*}
    For $x_1,x_2$ living on different edges (say $x_1\in e_1 = 1$ and $x_2\in e_2 = 2$), we have:
    \begin{align*}
        &\sum_{j=1}^{N-1} (\phi_j(x_1) - \phi_j(x_2))^2 = \sum_{j=1}^{N-1} \frac{2}{lj(j+1)} \round{A_1 - u_{2,j} A_2}^2\\
        &=\frac{1}{l} \round{A_1 + A_2}^2 + \round{A_1-A_2}^2 \sum_{j=2}^{N-1} \frac{2}{lj(j+1)}\\
        &=\frac{1}{l} (A_1 + A_2)^2 + \frac{N-2}{Nl} (A_1 - A_2)^2 = \frac{2}{Nl}\round{(N-1)\round{A_1^2+A_2^2} + 2 A_1 A_2}.
    \end{align*}
    To conclude, it is sufficient to use (\ref{eq:spectral_form_gamma}) in concert with the fact that the only relevant eigenvalue in the limit $p\to+\infty$ is $\lambda=\round{\frac{\pi}{2l}}^2$. As a final remark, we stress that the choices of $e_1,e_2$ in the above cases have been made to simplify the calculations, but, given the generality of (\ref{eq:spectral_form_gamma}) and the symmetry of $G$, nothing would have changed by considering other edges or another orthonormal basis of the eigenspace associated to $\lambda=\round{\frac{\pi}{2l}}^2$.
\end{proof}

\begin{example}[Segment $G_1$]
    \label{ex:segment_graph}
    Here we explore the very first toy example: a graph consisting of just two nodes $\curly{v_1,v_2}$ connected by an edge $e$ of length $1$. We also consider the additional point $p=\round{v_1,v_2,\frac{1}{2}}$ and the generic point $x=\round{v_1,v_2,x}$, with $x\in [0,1]$. Being $G_1$ a segment, it is natural to consider it as the interval $[0,1]\subset \R$. We consider the variogram $\gamma_p$ and the distance $d_p$ for $p\in\curly{1,\dots,5}$ (see Figure \ref{fig:variogram_segment_graph}). Some comments are in order.
    \begin{itemize}
        \item For $p=1$, the variogram coincides with the resistance distance (see Proposition \ref{prop:m=0_resistance_Anderes}), which, on a tree-graph such as $G_1$, coincides with the Euclidean distance \cite[Proposition 4]{AnderesEtAl_IsotropicCovarianceFunctions_2020}.
        \item The distance $d_p$ increases as the shortest-path distance $d_{SP}$ increases. However, $d_p$ is not a function of $d_E$: this can be seen by noticing that the plots of $d_p(v_1,x)$ and $d_p\round{p, \frac{1}{2} + x}$ for $x\in\square{0,\frac 1 2}$ are different.
        \item The distance $d_p(v_1,v_2)$ converges quickly to the limit as $p$ grows.
    \end{itemize}
    
    \begin{figure}[ht]
        \centering
        \begin{subfigure}{0.49\textwidth}
            \includesvg[width=\linewidth]{Figures/Segment_Distance_id.svg}
        \end{subfigure}
        \hfill
        \begin{subfigure}{0.49\textwidth}
            \includesvg[width=\linewidth]{Figures/Segment_Distance_sqrt.svg}
        \end{subfigure}
        \caption{\small Behaviour of the (scaled) variogram $\pi^{2p}\gamma_p$ and of its square root $\pi^{p}d_p$ on the segment $G_1$ for several values of $p$. The plots show $\gamma_p(v_1, x)$ (solid) and $\gamma_p(p,x)$ (dashed), where $x\in[0,1]$ is an arbitrary point moving from $v_1$ to $v_2$.}
        \label{fig:variogram_segment_graph}
    \end{figure}
\end{example}

\begin{example}[Star $G_2$]
    \label{ex:star_graph}
    The second toy example that we consider is the simplest star graph: $1$ origin ($v_0$) and $3$ leaves ($v_1,v_2,v_3$), with all the edges of length $1$. We will also consider the point $p=\round{v_0,v_2,\frac{1}{2}}$ and a generic point $x$ on $(v_1,v_0) \cup (v_0,v_2)$. $x$ will be parametrised as $x\in [0,2]$, where $0,1,2$ correspond to $v_1,v_0,v_2$, respectively and for $x\in [0,1]$, $x=(v_1,v_0,x)$, while for $x \in [1,2]$, $x=(v_0,v_2,x-1)$. This machinery allows to visualise the (non-linear) set $(v_1,v_0) \cup (v_0,v_2)$ on the easy-to-plot $[0,2]\subset \R$ (see Figure \ref{fig:variogram_star_graph}).  
    \begin{figure}[ht]
        \centering
        \begin{subfigure}{0.49\textwidth}
            \includesvg[width=\linewidth]{Figures/Star_Distance_id.svg}
        \end{subfigure}
        \hfill
        \begin{subfigure}{0.49\textwidth}
            \includesvg[width=\linewidth]{Figures/Star_Distance_sqrt.svg}
        \end{subfigure}
        \caption{\small Behaviour of the (scaled) variogram $\round{\frac{\pi}{2}}^{2p}\gamma_p$ and of its square root $\round{\frac{\pi}{2}}^{p}d_p$ on the star graph $G_2$ for several values of $p$. The plots show $\gamma_p(v_1, x)$ (solid) and $\gamma_p(p,x)$ (dashed), where $x\in[0,2]$ is an arbitrary point moving from $v_1$ to $v_0$ to $v_2$.}
        \label{fig:variogram_star_graph}
    \end{figure}
    Also in this case, we stress some significant facts.
    \begin{itemize}
        \item Again, for $p=1$, $\gamma_p$ coincides with the shortest-path distance.
        \item For $p\geq 2$, $\gamma_p$ presents a derivative jump at the origin $1=v_0$. Albeit perhaps strange, this is an expected behaviour that comes from the Kirchhoff condition of order $1$ and is a direct consequence of Proposition \ref{prop:gamma_belongs_to_C^m}.
    \end{itemize}
\end{example}

\begin{example}[Circle $G_3$]
    \label{ex:circle_graph}
    The third toy example that we consider is the circle: we model it by considering $3$ vertices $v_1,v_2,v_3$ connected by all the $3$ possible edges, all of length $1/3$. This is necessary as, to build the Laplacian matrix $\matr L$, the graph should not have self loops or multiple edges. We consider a generic point $x\in G_3$, parametrised by the arc-length distance $x\in[0,1]$, where $0,1/3,2/3,1$ correspond to $v_1,v_2,v_3,v_1$, respectively. As in the previous examples, we plot the variogram $\gamma_p(v_1,x)$ and the distance $d_p$ for some values of $p$ (see Figure \ref{fig:variogram_circle_graph}).
    \begin{figure}[ht]
        \centering
        \begin{subfigure}{0.49\textwidth}
            \includesvg[width=\linewidth]{Figures/Circle_Distance_id.svg}
        \end{subfigure}
        \hfill
        \begin{subfigure}{0.49\textwidth}
            \includesvg[width=\linewidth]{Figures/Circle_Distance_sqrt.svg}
        \end{subfigure}
        \caption{\small Behaviour of the (scaled) variogram $\round{2\pi}^{2p}\gamma_p$ and of its square root $\round{2\pi}^{p}d_p$ on the circle graph $G_3$ for several values of $p$. The plots show the variogram $\gamma_p(v_1, x)$ (solid) and $\gamma_p(p,x)$ (dashed), with the relative distances $d_p$, where $x\in[0,1]$ is an arbitrary point moving from $v_1$ to $v_2$ to $v_3$ to $v_1$.}
        \label{fig:variogram_circle_graph}
    \end{figure}
    Some comments are in order.
    \begin{itemize}
        \item Clearly, for $p=1$, the distance does not coincides with the shortest-path one, as the graph is no longer a tree.
        \item The distance $d_p$ is symmetric around $1/2$, reflecting the symmetry of the graph, confirming that, in this specific case, $d_p$ is a function of the shortest-path distance.
        \item Another consequence to the symmetry is that $\gamma_p$ is flat at $1/2$.
    \end{itemize}
\end{example}

\section{Discussion and conclusions}
\label{sec:conclusion}
    We have defined the class of polyharmonic distances $\curly{d_p}_{p\in\N^+}$ and have shown how the isotropic function $x_1,x_2 \mapsto \psi(d_p^2(x_1,x_2))$ is a valid covariance on $G\times G$ for any completely monotonic function $\psi$ continuous at $0$. We have shown that $d_p$ is an intrinsic property of the metric graph and, in particular, does not depend on addition or removal of degree-2 vertices. Furthermore, we have shown that $d_1^2$ is the effective resistance distance defined by \cite{AnderesEtAl_IsotropicCovarianceFunctions_2020} and that $d_2$ is the biharmonic distance on metric graphs. The polyharmonic distances change the geometry of metric graphs: we have not merely added another radial profile to be composed with the effective resistance distance to obtain valid covariance functions.\par 
    We have shown how these distances can be interpreted in several equivalent ways. First, $d_p$ is the maximum increment that can be generated by an admissible function on the metric graph having unit $p$-energy norm. In addition, the spectral representation provides a clean relation between $p$ and $d_p$: the contribution to the squared distance $d_p^2$ given by the squared increments of the Laplacian eigenfunction are weighted by the inverse powers of the eigenvalues, namely by $\lambda_k^{-p}$. Finally, we have shown an explicit stochastic construction that, for each $p\geq 1$, provides a Gaussian process whose variogram is $d_p^2$. This may be useful in the computation of the variogram, especially for low values of $p$. \par 
    We have proved that the parameter $p$ influences both the geometry induced by $d_p$ and the mean-square differentiability of the process $Y$ having covariance $\psi(d_p^2(x_1,x_2))$. Indeed, if $\abs{\psi^{(p-1)}}$ is finite at $0^+$ and if $p\geq 2$, then $Y \in C_1^{p-1}(G)$. In addition, we have shown that the higher-order Kirchhoff conditions of $Y$ are, in general, not inherited by the ones of the auxiliary process $Z$. We have also discussed how the (scaled) limit $d_\infty$ may lose the separation property.\par 
    We have proved that the auxiliary process $Z_m$ coincides with the limit (in finite-dimensional distribution) of the vertex-centred Whittle-Mat\`ern field for $\kappa \to 0^+$. This establishes a connection between the covariance-first and SDE-based approaches. The limit relation concerns the process $Z_m$ and not the isotropic process $Y$. \par 
    The construction suggest several directions for future research. First, the following statistical-inference aspects should be explored: for a given metric graph and data, how to choose the distance parameter $p$ and how to estimate the completely monotonic function $\psi$? Furthermore, computational aspects, such as computing or approximating distances for large metric graphs, or simulating the smooth process $Y$, should be assessed. In addition, several extensions of this work could be explored. For instance, will the distance $d_p$ be well-defined for non-integer $p$'s? In such a case, what are its properties? Are there alternative conditions on the derivatives of the auxiliary process $Z$ at the vertices and on the profile $\psi$ that guarantee higher-order Kirchhoff conditions of $Y$? We leave all these questions for future work.

\newpage
\bibliography{bib}

@article{AnderesEtAl_IsotropicCovarianceFunctions_2020,
  title = {Isotropic Covariance Functions on Graphs and Their Edges},
  author = {Anderes, Ethan and M{\o}ller, Jesper and Rasmussen, Jakob G.},
  year = 2020,
  month = aug,
  journal = {The Annals of Statistics},
  volume = {48},
  number = {4},
  pages = {2478--2503},
  publisher = {Institute of Mathematical Statistics},
  issn = {0090-5364, 2168-8966},
  doi = {10.1214/19-AOS1896},
  urldate = {2025-03-11}
}

@article{Lachal_ClassBridgesIterated_2011,
  title = {A {{Class}} of {{Bridges}} of {{Iterated Integrals}} of {{Brownian Motion Related}} to {{Various Boundary Value Problems Involving}} the {{One-Dimensional Polyharmonic Operator}}},
  author = {Lachal, Aim{\'e}},
  year = 2011,
  journal = {International Journal of Stochastic Analysis},
  volume = {2011},
  number = {1},
  pages = {762486},
  issn = {2090-3340},
  doi = {10.1155/2011/762486},
  urldate = {2026-05-07},
  copyright = {Copyright \copyright{} 2011 Aim\'e Lachal.},
  langid = {english}
}

@incollection{vanZantenvanderVaart_ReproducingKernelHilbert_2008,
  title = {Reproducing Kernel {{Hilbert}} Spaces of {{Gaussian}} Priors},
  booktitle = {Pushing the {{Limits}} of {{Contemporary Statistics}}: {{Contributions}} in {{Honor}} of {{Jayanta K}}. {{Ghosh}}},
  author = {{van Zanten}, J. H. and {van der Vaart}, A. W.},
  year = 2008,
  month = jan,
  volume = {3},
  pages = {200--223},
  publisher = {Institute of Mathematical Statistics},
  doi = {10.1214/074921708000000156},
  urldate = {2026-06-24},
  langid = {english}
}

@article{SottinenYazigi_GeneralizedGaussianBridges_2014,
  title = {Generalized {{Gaussian}} Bridges},
  author = {Sottinen, Tommi and Yazigi, Adil},
  year = {2014},
  journal = {Stochastic Processes and their Applications},
  volume = {124},
  number = {9},
  pages = {3084--3105},
  issn = {0304-4149},
  doi = {10.1016/j.spa.2014.04.002}
}

@article{GhoshEtAl_MinimizingEffectiveResistance_2008,
  title = {Minimizing {{Effective Resistance}} of a {{Graph}}},
  author = {Ghosh, Arpita and Boyd, Stephen and Saberi, Amin},
  year = 2008,
  month = jan,
  journal = {SIAM Review},
  volume = {50},
  number = {1},
  pages = {37--66},
  publisher = {{Society for Industrial and Applied Mathematics}},
  issn = {0036-1445},
  doi = {10.1137/050645452}
}

@article{FilosiEtAl_TemporallyEvolvingGeneralisedNetworks_2024,
  title = {Temporally-{{Evolving Generalised Networks}} and {{Their Kernels}}},
  author = {Filosi, Tobia and Agostinelli, Claudio and Porcu, Emilio},
  year = 2024,
  journal = {Statistica Sinica},
  volume = {38},
  number = {3},
  doi = {10.5705/ss.202024.0345},
  langid = {english}
}

@article{JorgensenPearse_HilbertSpaceApproach_2010,
  title = {A {{Hilbert Space Approach}} to {{Effective Resistance Metric}}},
  author = {Jorgensen, Palle E. T. and Pearse, Erin Peter James},
  year = 2010,
  month = nov,
  journal = {Complex Analysis and Operator Theory},
  volume = {4},
  number = {4},
  pages = {975--1013},
  issn = {1661-8262},
  doi = {10.1007/s11785-009-0041-1},
  langid = {english}
}

@article{BerkolaikoLiu_SimplicityEigenvaluesNonvanishing_2017,
  title = {Simplicity of Eigenvalues and Non-Vanishing of Eigenfunctions of a Quantum Graph},
  author = {Berkolaiko, Gregory and Liu, Wen},
  year = 2017,
  month = jan,
  journal = {Journal of Mathematical Analysis and Applications},
  volume = {445},
  number = {1},
  pages = {803--818},
  issn = {0022-247X},
  doi = {10.1016/j.jmaa.2016.07.026}
}

@book{BerkolaikoKuchment_IntroductionQuantumGraphs_2012,
  title = {Introduction to {{Quantum Graphs}}},
  author = {Berkolaiko, Gregory and Kuchment, Peter},
  year = 2012,
  month = dec,
  series = {Mathematical {{Surveys}} and {{Monographs}}},
  volume = {186},
  publisher = {American Mathematical Society},
  address = {Providence, Rhode Island},
  doi = {10.1090/surv/186},
  isbn = {978-0-8218-9211-4 978-0-8218-9455-2},
  langid = {english}
}

@article{BolinEtAl_GaussianWhittleMatern_2024,
  title = {Gaussian {{Whittle}}--{{Mat\'ern}} Fields on Metric Graphs},
  author = {Bolin, David and Simas, Alexandre B. and Wallin, Jonas},
  year = 2024,
  month = may,
  journal = {Bernoulli},
  volume = {30},
  number = {2},
  pages = {1611--1639},
  publisher = {{Bernoulli Society for Mathematical Statistics and Probability}},
  issn = {1350-7265},
  doi = {10.3150/23-BEJ1647}
}

@misc{Mugnolo_WhatActuallyMetric_2021,
  title = {What Is Actually a Metric Graph?},
  author = {Mugnolo, Delio},
  year = 2021,
  month = mar,
  number = {arXiv:1912.07549},
  eprint = {1912.07549},
  primaryclass = {math},
  publisher = {arXiv},
  doi = {10.48550/arXiv.1912.07549},
  archiveprefix = {arXiv}
}

@article{Schoenberg_MetricSpacesPositive_1938,
  title = {Metric Spaces and Positive Definite Functions},
  author = {Schoenberg, I. J.},
  year = 1938,
  journal = {Transactions of the American Mathematical Society},
  volume = {44},
  number = {3},
  pages = {522--536},
  issn = {1088-6850, 0002-9947},
  doi = {10.1090/S0002-9947-1938-1501980-0},
  langid = {english}
}

@misc{BlackEtAl_BiharmonicDistanceGraphs_2025,
  title = {Biharmonic {{Distance}} of {{Graphs}} and Its {{Higher-Order Variants}}: {{Theoretical Properties}} with {{Applications}} to {{Centrality}} and {{Clustering}}},
  shorttitle = {Biharmonic {{Distance}} of {{Graphs}} and Its {{Higher-Order Variants}}},
  author = {Black, Mitchell and Lin, Lucy and Nayyeri, Amir and Wong, Weng-Keen},
  year = 2025,
  month = feb,
  number = {arXiv:2406.07574},
  eprint = {2406.07574},
  primaryclass = {cs.SI},
  publisher = {arXiv},
  doi = {10.48550/arXiv.2406.07574},
  archiveprefix = {arXiv}
}

@article{LipmanEtAl_BiharmonicDistance_2010,
  title = {Biharmonic Distance},
  author = {Lipman, Yaron and Rustamov, Raif M. and Funkhouser, Thomas A.},
  year = 2010,
  month = jun,
  journal = {ACM Transactions on Graphics},
  volume = {29},
  number = {3},
  pages = {1--11},
  issn = {0730-0301, 1557-7368},
  doi = {10.1145/1805964.1805971},
  langid = {english}
}

@article{BaddeleyEtAl_AnalysingPointPatterns_2021,
  title = {Analysing Point Patterns on Networks --- {{A}} Review},
  author = {Baddeley, Adrian and Nair, Gopalan and Rakshit, Suman and McSwiggan, Greg and Davies, Tilman M.},
  year = 2021,
  month = apr,
  journal = {Spatial Statistics},
  series = {Towards {{Spatial Data Science}}},
  volume = {42},
  pages = {100435},
  issn = {2211-6753},
  doi = {10.1016/j.spasta.2020.100435}
}

@book{OkabeSugihara_SpatialAnalysisNetworks_2012,
  title = {Spatial {{Analysis Along Networks}}: {{Statistical}} and {{Computational Methods}}},
  shorttitle = {Spatial {{Analysis Along Networks}}},
  author = {Okabe, Atsuyuki and Sugihara, Kokichi},
  year = 2012,
  month = aug,
  publisher = {John Wiley \& Sons},
  googlebooks = {48GRqj51\_W8C},
  isbn = {978-0-470-77081-8},
  langid = {english}
}

@article{KleinRandic_ResistanceDistance_1993,
  title = {Resistance Distance},
  author = {Klein, D. J. and Randi{\'c}, M.},
  year = 1993,
  month = dec,
  journal = {Journal of Mathematical Chemistry},
  volume = {12},
  number = {1},
  pages = {81--95},
  issn = {0259-9791, 1572-8897},
  doi = {10.1007/BF01164627},
  copyright = {http://www.springer.com/tdm},
  langid = {english}
}

@article{CressieEtAl_SpatialPredictionRiver_2006a,
  title = {Spatial Prediction on a River Network},
  author = {Cressie, Noel and Frey, Jesse and Harch, Bronwyn and Smith, Mick},
  year = 2006,
  month = jun,
  journal = {Journal of Agricultural, Biological, and Environmental Statistics},
  volume = {11},
  number = {2},
  pages = {127},
  issn = {1537-2693},
  doi = {10.1198/108571106X110649},
  langid = {english}
}

@article{PataneSpagnuolo_InteractiveAnalysisHarmonic_2013,
  title = {An Interactive Analysis of Harmonic and Diffusion Equations on Discrete {{3D}} Shapes},
  author = {Patan{\'e}, Giuseppe and Spagnuolo, Michela},
  year = 2013,
  month = aug,
  journal = {Computers \& Graphics},
  volume = {37},
  number = {5},
  pages = {526--538},
  issn = {0097-8493},
  doi = {10.1016/j.cag.2013.03.006}
}

@misc{BolinEtAl_NewClassNonstationary_2026,
  title = {A New Class of Non-Stationary {{Gaussian}} Fields with General Smoothness on Metric Graphs},
  author = {Bolin, David and {Riera-Segura}, Lenin and Simas, Alexandre B.},
  year = 2026,
  month = mar,
  number = {arXiv:2501.11738},
  eprint = {2501.11738},
  primaryclass = {stat.ME},
  publisher = {arXiv},
  doi = {10.48550/arXiv.2501.11738},
  archiveprefix = {arXiv}
}

@book{Adler_GeometryRandomFields_2010,
  title = {The {{Geometry}} of {{Random Fields}}},
  author = {Adler, Robert J},
  year = 2010,
  month = jan,
  series = {Classics in {{Applied Mathematics}}},
  publisher = {{Society for Industrial and Applied Mathematics}},
  doi = {10.1137/1.9780898718980},
  isbn = {978-0-89871-693-1}
}

\newpage

\end{document}